\documentclass[article,11pt,reqno]{amsart}
\usepackage[english]{babel}
\usepackage[utf8]{inputenc}
\usepackage{thm-restate}
\usepackage{amsmath, amssymb, amsthm}
\usepackage[table]{xcolor}
\usepackage{tikz-cd}
\usepackage{graphicx}
\usepackage[]{todonotes}
\usepackage{physics}
\usepackage{setspace}
\usepackage[margin=0.9in]{geometry}

\usepackage{caption}
\usepackage{subcaption}
\usepackage{floatrow}
\usepackage{xpatch}
\usepackage{enumitem}
\usepackage{graphicx}
\usepackage{mathrsfs}
\usepackage{bbm}

\usepackage{verbatim}
\usepackage{adjustbox}
\newcommand{\nocontentsline}[3]{}
\newcommand{\tocless}[2]{\bgroup\let\addcontentsline=\nocontentsline#1{#2}\egroup}

\DeclareMathOperator*{\UC}{\mathtt{UC}}
\DeclareMathOperator*{\UEG}{UEG}

\newcommand{\N}{\mathbb{N}}

\newcommand{\Z}{\mathbb{Z}}

\newcommand{\cc}{\leftrightarrow}
\newcommand{\Prb}{\mathbb{P}}

\newcommand{\Prbcur}{\mathbf{P}}

\usepackage{tikz}
\usepackage{pgfplots}
\usepackage{tikz-3dplot} 
\usetikzlibrary{positioning}
\usetikzlibrary{arrows}

\newcommand{\id}{1\! \!1}

\newcommand{\nn}{\mathbf{n}}

\newcommand\independent{\protect\mathpalette{\protect\independenT}{\perp}}
\def\independenT#1#2{\mathrel{\rlap{$#1#2$}\mkern2mu{#1#2}}}

\usetikzlibrary{positioning,calc}

\newlength\tindent
\makeatletter
\newcommand{\changeoperator}[1]{%
  \csletcs{#1@saved}{#1@}%
  \csdef{#1@}{\changed@operator{#1}}%
}
\newcommand{\changed@operator}[1]{%
  \mathop{%
    \mathchoice{\textstyle\csuse{#1@saved}}
               {\csuse{#1@saved}}
               {\csuse{#1@saved}}
               {\csuse{#1@saved}}%
  }%
}

\usetikzlibrary{calc,arrows.meta}

\newcommand{\Bin}{\operatorname{Bin}}
\newcommand{\Giant}{\mathtt{Giant}}

\newcommand{\ak}{\mathtt{Ann}^k}
\newcommand{\Pis}{\mathtt{Pis}}

\newcommand{\ssd}  {\overset{\mathbf{s}}{\preceq}}

\newcommand{\xiA}{\xi_0}

\usepackage{sidecap}
\usetikzlibrary{patterns,decorations.pathmorphing,arrows.meta,calc}
\usepackage{wrapfig}

\changeoperator{prod}
\pgfplotsset{compat=1.18}
\usepackage{mdframed}
\mdfdefinestyle{MyFrame}{
    linecolor=blue,
    outerlinewidth=30pt,
    roundcorner=20pt,
    innertopmargin=4pt,
    innerbottommargin=5pt,
    innerrightmargin=20pt,
    innerleftmargin=20pt,
    backgroundcolor=gray!15!white}
    
\usepackage[hidelinks,hypertexnames=false]{hyperref}
\usepackage{cleveref}
\newtheorem{theorem}{Theorem}[section]
\newtheorem{definition}[theorem]{Definition}
\newtheorem{proposition}[theorem]{Proposition}

\newtheorem{remark}[theorem]{Remark}
\newtheorem{corollary}[theorem]{Corollary}

\newtheorem{question}[theorem]{Question}

\newtheorem{lemma}[theorem]{Lemma}

\newtheorem{coupling}[theorem]{Coupling}
\crefname{coupling}{coupling}{couplings}
\Crefname{coupling}{Coupling}{Couplings}

\makeatletter
\newcommand*{\centernot}{\mathpalette\@centernot}
\def\@centernot#1#2{%
  \mathrel{%
    \rlap{%
      \settowidth\dimen@{$\m@th#1{#2}$}%
      \kern.5\dimen@
      \settowidth\dimen@{$\m@th#1=$}%
      \kern-.5\dimen@
      $\m@th#1\not$%
    }%
    {#2}%
  }%
}
\makeatother

\author{Ulrik Thinggaard Hansen}
\address{Ulrik Thinggaard Hansen \\ Department of Mathematics,
Universität Innsbruck, Technikerstrasse 13, 6020 Innsbruck, Austria }
\email{ulrik.hansen@uibk.ac.at}

\author{Frederik Ravn Klausen}
\address{Frederik Ravn Klausen, University of Cambridge, DPMMS, Cambridge, United Kingdom}
\email{frk23@cam.ac.uk}

\title{The Supercritical Loop O(1) and Random Current Models: Uniqueness and Mixing}

\begin{document}
\begin{abstract}
Much recent rigorous study of the classical ferromagnetic Ising model has been powered by its graphical representations, such as the random current and loop O(1) model (high temperature expansion).
In this paper, we prove uniqueness of Gibbs measures and exponential ratio weak mixing for the loop O(1) and random current models corresponding to the supercritical Ising model on the hypercubic lattice $\mathbb{Z}^d$ in any dimension $d \geq 2$. 
The methods generalise to $q$-flow models and have natural applications for gradient measures of $\mathbb{Z}/q\mathbb{Z}$-gauge theories. 
\end{abstract}

\vspace{-1cm}
\maketitle
\vspace{-0.5cm}

\section{Introduction}
The loop O(1) model $\ell$ and the random current $\Prbcur$ are percolation models arising as graphical representations of the Ising model, encoding its correlations in terms of connectivity properties of random graphs. Together with the FK-Ising model $\phi$, they have been central to much recent progress on the Ising model \cite{aizenman2021marginal} and have increasingly become objects of study in their own right \cite{ Gri06}. A basic question about any such model is the uniqueness of its infinite-volume Gibbs measure, which, by the work of Pisztora \cite{Pis96} and Bodineau \cite{bodineau2006translation}, see also \cite{raoufi2020translation}, is now well-understood for the FK-Ising model.  In this paper, we resolve this question for the loop $O(1)$ and random current models in the supercritical regime on the hypercubic lattice $\Z^d$ for all $d\geq2$.

We generally work with the standard percolation state space $\{0,1\}^{E(\mathbb{Z}^d)}$. For subgraphs $G\subseteq \mathbb{Z}^d,$ we denote by $\mathcal{A}_G$ the $\sigma$-algebra generated by the restriction to $G$. As is standard, $\Lambda_n=[-n,n]^d\cap \mathbb{Z}^d$ denotes the box of size $n$.
A percolation measure $\nu$ on $\mathbb{Z}^d$ is \emph{(exponentially) ratio weak mixing} if there exist constants $c,C,k>0$ such that, 
\begin{equation}\label{eq:RWM}
\forall n\in \mathbb{N}\; \forall A\in \mathcal{A}_{\Lambda_n}\forall B\in \mathcal{A}_{\mathbb{Z}^d\setminus \Lambda_{kn}}:\quad |\nu[A\cap B]-\nu[A]\nu[B]|\leq C\exp(-c n) \nu[A]\nu[B]. \tag{RWM}
\end{equation}

For a finite subgraph $G\subseteq \mathbb{Z}^d$ and $x\in [0,1],$ the loop O($1$) model on $G$ with sources $A\subseteq G$ (typically a subset of the vertex boundary $\partial_v G),$ denoted $\ell_{G,x}^A$, is  Bernoulli percolation with edge probabilities $\frac{x}{1+x},$ conditioned on all vertices in $A$ having odd degree, and all others having even degree. The loop O($1$) model is tied to the Ising model through the high temperature expansion, Kramers-Wannier duality, and the uniform even subgraph. 
In particular, the infinite volume limit $\ell_{\mathbb{Z}^d,x}$ has exponential decay if $x < x_c = \tanh(\beta_c)$ and a polynomial lower bound on connection probabilities if $x > x_c$, where $\beta_c$ is the critical inverse temperature of the Ising model on $\Z^d$ \cite{hansen2023uniform}. 
Our first theorem establishes the existence of a unique thermodynamic limit in the supercritical regime in arbitrary dimension and shows that this limit is ratio weak mixing. We say that a measure $\nu$ is \emph{Gibbs} for the loop O($1$) model if its conditional distributions on a finite volume $\Lambda$ given its exterior comport with the finite volume measures (see \eqref{SMP} and following discussion).

\begin{theorem} \label{thm:Unique_Loop}
For any $d\geq 2$ and $x>x_c,$ there exists a measure $\ell_{\mathbb{Z}^d,x}$ such that for any exhaustion $G_N\nearrow \mathbb{Z}^d$ and any $A_N\subseteq \partial_v G_N$ with $|A_N|$ even, 
$$
\ell_{\mathbb{Z}^d,x}=\lim_{N\to\infty} \ell_{G_N,x}^{A_N}.
$$
 In particular, $\ell_{\mathbb{Z}^d,x}$ is the unique Gibbs measure for the loop O($1$) model. Furthermore, $\ell_{\mathbb{Z}^d,x}$ is exponentially ratio weak mixing.
\end{theorem}

We remark that the result is not new for $d=2,$ where it is a combination of the Aizenman-Higuchi Theorem \cite{aizenman1980translation,higuchi1981absence} and mixing results for the dual (subcritical) random-cluster model \cite{alexander1998weak,DCsharpness}. Nevertheless, a proof in the framework of our methods is included for completeness.

The single random current measure $\mathbf{P}^A_{G,\beta}$ at inverse temperature $\beta>0$ on a finite graph $G$ with sources $A$ is given by an i.i.d. family $(\mathbf{n}_e)_{e\in E(G)}$ of $\mathbf{Poi}(\beta)$ variables on $E(G)$ conditioned on having $\partial\mathbf{n}:=\{v\in V(G)\mid \sum_{w:vw\in E(G)} \mathbf{n}_{vw} \;\mathrm{odd}\}=A$. Uniqueness of the random current follows relatively painlessly from \Cref{thm:Unique_Loop}.
\begin{theorem} \label{thm:Free-mix}
For any $d\geq 2$ and $\beta>\beta_c,$ there exists a measure $\mathbf{P}_{\mathbb{Z}^d, \beta}$ such that for any exhaustion $G_N\nearrow \mathbb{Z}^d$ and any $A_N\subseteq \partial_v G_N$ with $|A_N|$ even,
$$
\lim_{N\to\infty} \mathbf{P}^{A_N}_{G_N, \beta} =\mathbf{P}_{\mathbb{Z}^d,\beta}
$$
 
 Moreover, $\mathbf{P}_{\mathbb{Z}^d,\beta}$ and $\mathbf{P}_{\mathbb{Z}^d,\beta} \otimes \mathbf{P}_{\mathbb{Z}^d,\beta}$  are exponentially ratio weak mixing. 
\end{theorem}

We state the two main theorems in the case without sources in the bulk for readability and because we take them to be of the most a priori interest. However, our methods transfer with suitable modifications to the setting where sources are left in the bulk (see \Cref{thm:Unique_loop_with_sources} and  \Cref{thm:Unique_cur_with_sources}). In addition, the results generalise to the loop representation of the $q$-state Potts model for $q>2$ (see \Cref{thm:Unique_loop_qversion} - here we caveat that the final theorem is weaker, reflecting the comparatively weaker a priori input available for the associated random-cluster model).

One motivation for proving mixing statements of random currents is that they can be used to gain further insight on the Ising model. 
Previously, polynomial mixing of the critical random current was proven in two dimensions by Duminil-Copin, Lis and Qian using planar techniques \cite{duminil2021conformal}.
A different mixing property of the double random current \cite[Theorem 6.4]{aizenman2021marginal} was mentioned as the core of the proof of marginal triviality of $\varphi^4$-fields and Ising models \cite{panis2024incipient}. Finally, it has recently been speculated that improved mixing statements for the random current would help prove analyticity of the so-called Ursell functions in the supercritical regime \cite[Remark 1.9]{dalimontegassman}.

A supercritical mixing result which goes in the direction of \Cref{thm:Unique_Loop} was previously obtained in \cite[Theorem 4.11]{hansen2023uniform}, but we note that the proof there gives neither uniqueness of Gibbs measures nor stability under conditioning on events of small probability. In words, previous work was concerned with studying measures which can be written as uniform even subgraphs of random-cluster models, whereas the current paper proves that any Gibbs measure of the loop O($1$) model in the supercritical regime \textit{is} the uniform even subgraph of the random-cluster model.

\subsection{Organisation of Paper 
}
The main technical input needed to derive the results for the loop O(1) model and random currents is \Cref{theorem:any_incerasing_event}, which shows that FK crossings of annuli are unique with high probability, even in the presence of sources. In \Cref{sec:proofs_of_main_theorems}, it is shown how \Cref{thm:Unique_Loop}
and \Cref{thm:Free-mix} follow from \Cref{theorem:any_incerasing_event} by combining the relationship between the loop O($1$) and FK-Ising model first developed in \cite{evertz2002new, grimmett2007random} and extended to the setting with sources in \cite{aizenman2019emergent}.
These arguments yield weak mixing, which for our models implies ratio weak mixing by classical work of Alexander \cite[Theorem 3.3]{alexander1998weak}. 

Once uniqueness for the loop O($1$) model is established, one may deduce Theorem \ref{thm:Free-mix} by abstract arguments, which are covered in \Cref{sec:proofs_of_main_theorems}. Furthermore, in \Cref{sec:how_to_handle_sources}, we discuss an adaptation of our methods to the setting with sources in the bulk. Our arguments make substantial use of Pisztora's supercritical sharpness results \cite{Pis96}, which hold throughout the entire supercritical regime by Bodineau's result \cite{Bod05} (see \cite{SeverSlab} for a recent, simpler proof).

The methods of this paper can also be applied to the $q$-flow model defined in \cite{zhang2020loop}. To avoid notational clutter, we defer this discussion to \Cref{sec:qflow}. 
Furthermore, by the general duality of the loop O(1) model and  $\Z/2\Z$-lattice gauge theories, the results here have implications for the gradient measure of lattice gauge theories, which we discuss in \Cref{sec:gauge}.

In the first appendix, we discuss uniqueness of the loop O($1$) and random current measures at the weakest but most general level. In particular, we write down conditions under which the so-called wired and free measures coincide. We take this fact to be well-known, but do not know of any written reference. Along the way, we discuss infinite volume relations between the FK-Ising and loop O($1$) models.

In Appendix \ref{sec:Pisztora_giant meets_the_boundary} and \ref{sec:Catch} we record an alternative proof of a special case of \Cref{theorem:any_incerasing_event}  (see \Cref{proposition:unique_crossing_with_fa_old_version}).  This proposition is also sufficient to prove the main theorem and while the proof given in the appendix is substantially more complicated, it uses techniques we believe to be of independent interest. These sections  are entirely focused on the random-cluster model.
The goal is to prove that the random-cluster model is insensitive to extra required connectivities (obtained by conditioning on an $\mathcal{F}_A$ event, which encodes the sources of the loop O($1$) model). Roughly speaking, we perform a multiscale argument where, at each scale, a positive fraction of the sources from the loop O($1$) model get connected to the unique giant component of the unconditioned random-cluster model, whence the conditioning will be erased after a logarithmic number of scales (see \Cref{fig:catching_figure} below).

\subsection{Open Problems.}
 It is worth noting that we rely heavily on the model being supercritical and, indeed, uniqueness might fail for small $x$ in high dimension. This would be analogous to Dobrushin's construction of non-translation invariant Gibbs states for the Ising model \cite{DobrushinStates}. More concretely, the Ising lattice gauge model (interacting over codimension $1$ cells) having non-translation invariant Gibbs states at very low temperature is equivalent to having non-uniqueness of Gibbs measures for $\ell_x$ when $x$ is small.

\begin{question}
Does $\ell_{\Z^d,x}$ always have a unique Gibbs measure for $x < x_c$? Equivalently, does $\mathbf{P}^{\emptyset}_{\Z^d,\beta}$ have a unique Gibbs measure for $\beta<\beta_c$?
\end{question}

We believe the answer might very well be yes. This would be analogous to the two-dimensional analysis in the celebrated Aizenman-Higuchi Theorem \cite{aizenman1980translation,higuchi1981absence}. It is possible that the techniques from the recent proof \cite{Coquille_Trans} adapt without too much trouble.

In \cite{duminil2020exponential}, Duminil-Copin, Goswami and Raoufi proved exponential decay of truncated correlations and show exponential ratio weak mixing for the FK-Ising model. Here, our statement of \Cref{thm:Unique_Loop} would imply exponential ratio weak mixing of FK-Ising, since this measure, just as the random current, arises as a sprinkling of the loop O(1) model. However, our proof of \Cref{thm:Unique_Loop} (but notably not that of the main technical input \Cref{theorem:any_incerasing_event}), relies on the ratio weak mixing for the FK-Ising model from \cite{duminil2020exponential}. Nevertheless, the heuristic of our present results is that the supercritical loop O($1$) model  (and its $q$-flow cousins) mixes "as well"  as the corresponding random-cluster model does. This invites the following question:
\begin{question}
Can the techniques of this paper be adapted to give a new proof of exponential ratio weak mixing for the FK-Ising model?
\end{question}
A positive answer, along with our applications to the $q$-flow model below in \Cref{sec:qflow}, would yield a proof of exponential decay of truncated correlations for the Potts model above the slab percolation threshold.

In \Cref{sec:how_to_handle_sources} we prove the main theorem for finitely many sources left in the bulk. The corresponding question for infinitely many sources is left open.
\begin{question}
    Obtain analouges of \Cref{thm:Unique_loop_with_sources} and \Cref{thm:Unique_cur_with_sources} when infinitely many sources are left in the bulk. 
\end{question}

The $q$-flow representation (of the Potts model) was introduced by Zhang et al. in \cite{zhang2020loop} and generalised to the plaquette case in \cite{hansen2025general}. It is the natural generalisation of the loop O(1) model to $q>2$. For $x > x_{\text{slab}}$, \Cref{thm:Unique_loop_qversion}  gives a characterisation of its Gibbs measures. The remaining case might be non-trivial - especially when $q$ is large.
 \begin{question}
     For $q\neq 2$, and $x \leq x_{\text{slab}},$ what are the Gibbs measures for the $q$-flow model on $\Z^d$?
 \end{question}
 We also wonder whether something could be said about the critical exponents for the random current and loop O($1$) models, following up on the recent work on the random-cluster model in high-dimensions by van Engelenburg, Garban, Panis and Severo \cite{vanengelenburg2025onearm}. 
 
For a locally finite, infinite graph $\mathbb{G}$ with finitely many ends $\mathfrak{e}_1, \dots \mathfrak{e}_n$, say that an end is \emph{robust} if for any finite $\Lambda$ which is large enough to separate the ends (as the infinite connected components of $\mathbb{G} \setminus \Lambda$), the threshold for Bernoulli percolation on the connected component corresponding to that end is strictly less than 1. Denote the set of robust ends of $\mathbb{G}$ by $\mathfrak{re}(\mathbb{G})$.  
\begin{question}
    Given a locally finite, infinite graph with finitely many ends $\mathbb{G}$, does there exist an $x_c < 1$ such that for any $x \in (x_c,1)$ the set of extremal Gibbs measures of $\ell_{\mathbb{G},x}$ is in natural correspondence with $\{f \in \{0,1\}^{\mathfrak{re}(\mathbb{G})}   \mid \abs{f} \in 2\mathbb{N}_0  \}$?
\end{question}
This would mimic the characterisation of \cite[Cor. 3.17]{hansen2023uniform} of the extremal Gibbs measures of the uniform even subgraph of $\mathbb{G}$ 
in terms of the ends of $\mathbb{G}$ by $\{f \in
\{0,1\}^{\mathfrak{e}(\mathbb{G})} \mid \abs{f} \in 2\mathbb{N} \}$.

On the technical side, while our techniques in Appendix \ref{sec:Pisztora_giant meets_the_boundary} and \ref{sec:Catch} do prove that the giant supercritical cluster is robust enough that it can touch an arbitrary boundary set in many points, even under adverse boundary conditions, they fall short of what should be a plausible result, and which would complement Pisztora's original work \cite{Pis96} much better:
\begin{question}
For $d\geq 3$, $p>p_c$, and a box $\Lambda_n,$ let $A$ denote the set of points on $\partial_v\Lambda_n$ connected to the giant $\mathtt{Giant}_{\Lambda_n}$ (cf. \Cref{sec:Pisztora_giant meets_the_boundary}). Does $|A|$ satisfy a large deviation principle? 
\end{question}
The question is, on purpose, slightly vague, as boundary conditions might come into play - for instance, it would not at all be surprising that the free and wired measures would have different typical sizes of $|A|,$ even if they both agree that the giant should have size roughly $\theta|\Lambda_n|$ (which, among other things, is the content of Pisztora's Theorem). One apparent avenue of tackling the question would be to prove an analogous statement for half-space measures for the random-cluster model with constant boundary conditions.
\section*{Acknowledgments}
We would like to thank Lorca Heeney for several insightful remarks in casual discussion. We would also like to thank Romain Panis for encouraging us to write \Cref{sec:how_to_handle_sources}. Towards the completion of this work, we became aware of a related effort announced by Gunaratnam, Panagiotis, Panis and Severo \cite{GPPS} containing results similar to those presented here. However, their methods are very different, revolving around the geometry of the supercritical double random current model rather than that of the supercritical random-cluster model. One should therefore expect the techniques to generalise differently.  No data were used for this study and the authors have no relevant conflicts of interest.   FRK was supported by the Carlsberg Foundation, grant CF24-0466.  This research was funded in part by the Austrian Science Fund (FWF) 10.55776/P34713.

\section{Setup, notation, and necessary basic properties}
We start by fixing graph-theoretic notation. For a graph $G$, denote by $V(G)$, respectively $E(G)$, the set of vertices, respectively edges, of $G$. We will be particularly concerned with the finite subgraphs $\Lambda_n=[-n,n]^d\cap \mathbb{Z}^d$ of $\mathbb{Z}^d.$ For $G\subseteq \mathbb{Z}^d,$ we denote by $\partial_v G$ the vertex boundary of $G$, i.e. the set of vertices in $G$ with at least one neighbour outside of $G$. For percolation configurations $\omega\in \{0,1\}^E,$ we generally identify $\omega$ with the graph $(V,\omega^{-1}(\{1\}))$ and equivalently with its set of open edges $\omega^{-1}(\{1\}).$ We write $\mathcal{C}_v$ for the connected component (or cluster) of the vertex $v$ and $\{A\cc B\}$ for the event that there are $v\in A,w\in B$ with $\mathcal{C}_v=\mathcal{C}_w.$ In case of ambiguity, we let $\{A\overset{\omega}{\longleftrightarrow} B\}$ denote the event that $\omega\in \{A\cc B\}.$ We also remark at the outset that for a probability measure $\nu$ and a measurable function $f$, $\nu[f]$ denotes the expectation of $f$ under $\nu$. 

 We follow the standard notation set in e.g., \cite{DC17} and let $\Prb_{G,p}$ denote Bernoulli percolation on the finite graph $G\subseteq \mathbb{Z}^d$, $\phi^\xi_{G,p}$ be the FK-Ising model, also known as the random-cluster model with cluster weight $2$, with parameter $p\in (0,1)$ and boundary condition $\xi\in \{0,1\}^{E(\mathbb{Z}^d)}$ defined by assigning a probability to every edge configuration $\omega \in \{0,1\}^{E(G)}$,
$$
\phi^{\xi}_{G,p}[\omega] \propto 2^{\kappa^{\xi}(\omega)} \left( \frac{p}{1-p} \right)^{\abs{\omega}}, 
$$
where $\abs{\omega}=\sum_{e\in E}\omega_e$ is the number of open edges  in $\omega$ and $\kappa^{\xi}(\omega)$ is the number of connected components intersecting $G$ in $(V(\mathbb{Z}^d),\omega^{\xi}),$ where
$$
\omega^{\xi}_e=\begin{cases}
    \omega_e & e\in E(G) \\ \xi_e & e\in E(\mathbb{Z}^d)\setminus E(G).
\end{cases}
$$
The cases $\xi\equiv 0$ and $\xi\equiv 1$ are of particular interest and referred to as the free and wired measures, respectively.

Similarly, for $x\in (0,1)$ and $A\subseteq V(G)$ with $|A|$ even, $\ell_{G,x}$ denotes the loop O($1$) model defined on $\{0,1\}^E$ by assigning probabilities
$$
\ell^A_{G,x}[\eta]\propto x^{|\eta|}\id_{\partial \eta=A},
$$
where $\partial\eta=\{v\in V(G)\mid \sum_{w:vw\in E(G)} \eta_{vw} \;\mathrm{odd}\}$ denotes the set of sources of $\eta$. In general, we denote\footnote{Since any finite graph must have an even number of vertices with odd degree, this set is empty if $|A|$ is odd.} $\Omega_A(G)=\{\eta\in \{0,1\}^E\mid \partial \eta=A\}.$

Both models satisfy spatial Markov properties: For $H\subseteq G,$
\begin{equation*} \label{SMP}
\phi^{\xi}_{G,p}[\omega|_H \mid \omega|_{G\setminus H}]=\phi_{H,p}^{(\omega|_{G\setminus H})^{\xi}}[\omega|_H] \qquad \ell^{A}_{G,x}[\eta|_H\mid \eta|_{G\setminus H}]=\ell_{H,x}^{A\triangle\partial (\eta|_{G\setminus H})}[\eta|_H], \tag{SMP}
\end{equation*}
both of which may be checked manually\footnote{Note that our definition of $\phi^{\xi}_G$ only depends on $\xi|_{\mathbb{Z}^d\setminus G}$.}.

The Markov property gives rise to the following definition of Gibbs measures: Say that a measure $\ell$ is a Gibbs measure of the loop O(1) model if $\ell$-almost surely
\begin{align}\label{eq:definition_of_Gibbs_measure_for_loop_O(1)}
    \ell_{G,x}^{\partial (\eta\vert_{\Z^d \setminus G})}[\;\cdot\;] = \ell[\; \cdot \mid \eta\vert_{\Z^d \setminus G}].
\end{align}

While the main results of this paper are related to the loop O(1) model, most of the paper is concerned with proving technical results on the side of the random-cluster model. 
In \Cref{sec:proofofmain}, it is shown how the unique crossing events of the random-cluster model can be cashed out as uniqueness of measure for the loop O(1) model using the loop-cluster coupling. 
This coupling was introduced in \cite{evertz2002new, grimmett2007random} and generalised in \cite{aizenman2019emergent,hansen2025general,zhang2020loop}.

Define $\mathcal{F}_A=\{\omega\in \{0,1\}^E\mid \exists \eta \in \Omega_A(G):\eta\subseteq \omega\}.$ Equivalently, $\mathcal{F}_A$ is the event that for each connected component $\mathcal{C}$ of $\omega,$ $|\mathcal{C}\cap A|$ is even. 

Furthermore, we denote $\operatorname{UG}^A_G=\ell^A_{G,1},$ i.e. the uniform measure on $\Omega_A(G)$. For $A=\emptyset,$ $\operatorname{UG}^{\emptyset}$ is the uniform even subgraph, which we denote $\operatorname{UEG}$.
The measures $\phi^0_{G,p}[\;\cdot\mid \mathcal{F}_A]$ and $\ell^A_{G,x}[\;\cdot\;]$ are coupled through $\operatorname{UG}^A$ and Bernoulli sprinkling: 
\begin{coupling}[\cite{aizenman2019emergent,evertz2002new, grimmett2007random}] \label{Thm:The_Coupling}
Let $G=(V,E)$ be a finite graph, $x\in (0,1),$ $A\subseteq V$ with $|A|$ even and let $(\omega,\eta)\in \{0,1\}^E\times \Omega_A(G)$ be a random element with distribution
$$
\mathscr{P}[(\omega,\eta)]\propto \mathbb{P}_{G,x}[\omega]\id_{\eta\subseteq \omega}.
$$
Set $p = \frac{2x}{1+x}$. The marginals are $\mathscr{P}[\omega\in\cdot\;]=\phi^0_{G,p}[\;\cdot\mid \mathcal{F}_A]$ and $\mathscr{P}[\eta\in \cdot\;]=\ell^A_{G,x}[\;\cdot\;]$, while the conditional measures are $\mathscr{P}[\eta\mid \omega]=\operatorname{UG}^A_{\omega}[\eta]$ and $\mathscr{P}[\omega\mid \eta]=\delta_{\eta}\cup \mathbb{P}_{G,x}[\omega].$ 
\end{coupling}
Here, for two percolation measures $\mu$ and $\nu$, we denote by $\mu\cup \nu$ the pushforward of $\mu\otimes \nu$ under the union map. Equivalently, it is the distribution of the union of independent samples of $\mu$ and $\nu$. Thus, the last item says that 
\begin{equation} \label{FK_is_sprinkled_O(1)}
\hspace{4cm}\phi^0_{G,p}[\;\cdot \mid \mathcal{F}_A]=\ell^A_{G,x}\cup \mathbb{P}_{G,x} \tag{FK is sprinkled O($1$)}[\;\cdot\;].
\end{equation}
Similarly, the conclusion of the other conditional identity is that  $\ell_{G,x}^A$ can be sampled as the uniform subgraph with sources $A$ of $\phi^0_{G,p}[\;\cdot \mid \mathcal{F}_A]$, i.e.,
\begin{equation} \label{O(1)_as_ueg}
\ell^A_{G,x}[\;\cdot\;]=\phi^0_{G,p}[\operatorname{UG}^A_{\omega}[\;\cdot\;]\mid \mathcal{F}_A]. 
\end{equation}

\Cref{Thm:The_Coupling} is one of the main tools for extracting information about the loop O($1$) model, as the model itself lacks many desirable properties, such as positive association, finite energy and monotonicity \cite{klausen2021monotonicity}.

Monotonicity properties, in turn, play an immense role in the study of the FK-Ising model.
There is a natural partial order $\preceq$ on $\{0,1\}^E$ given by pointwise comparison (i.e. inclusion). We say that an event $A\subseteq \{0,1\}^E$ is increasing if whenever $\omega\preceq \omega'$ and $\omega\in A,$ then $\omega'\in A$. For two percolation measures $\nu,\mu$ on $\{0,1\}^E,$ we say that $\mu$ \emph{stochastically dominates} $\nu$, written $\nu\preceq \mu$, if $\mu[A]\geq \nu[A]$ for every increasing event $A$. By Strassen's Theorem, this is equivalent to the existence of a coupling $(\omega_1,\omega_2)$ of $\omega_1\sim \mu$ and $\omega_2\sim \nu$ such that $\omega_2\preceq \omega_1$ almost surely. Such a coupling is called \emph{increasing}. Two classical instances of stochastic monotonicity for the FK-Ising model are the comparison between boundary conditions (CBC) and the FKG inequality:
\begin{align*}
\phi^{\xi}_{G,p}\succeq \phi^{\xi'}_{G,p} &\qquad \mathrm{for} \; \xi\succeq \xi'\label{eq_CBC} \tag{CBC}\\
\phi^{\xi}_{G,p}[\;\cdot\mid A]\succeq \phi^{\xi}_{G,p} &\qquad \mathrm{for}\; A\; \mathrm{increasing} \label{FKG} \tag{FKG}
\end{align*}
We refer to \cite{DC17,Gri06} for the proofs and more thorough introductions to the random-cluster model. 

In this paper, we shall also be concerned with infinite volume versions of these measures, defined as weak limits of the form $\lim_{G_n\nearrow \mathbb{Z}^d} \phi_{G_n,p}^{\xi_n}$, $\lim_{G_n\nearrow \mathbb{Z}^d} \ell_{G_n,x}^{A_n}$. Previous work has already established the existence of a unique infinite volume measure $\phi_{\mathbb{Z}^d,p}=\lim_{G_n\nearrow \mathbb{Z}^d} \phi_{G_n,p}^{\xi_n}$ for all choices of exhaustion $(G_n)_{n\in\mathbb{N}},$ boundary conditions $(\xi_n)_{n\in \mathbb{N}}$ and $p\in (0,1)$ - we refer to \cite{raoufi2020translation} for the most unified treatment of this fact. The existence of a measure $\ell_{\mathbb{Z}^d,x}=\lim_{G_n\nearrow \mathbb{Z}^d} \ell^{\emptyset}_{G_n,x}$ may be justified purely via considerations of correlations in the Ising model along similar lines as \cite{aizenman2015random}, whereas we will  proceed via an infinite volume generalisation of the coupling from \Cref{Thm:The_Coupling}. That is, in \cite{hansen2023uniform}, a construction was given of a random pair $(\omega,\eta)$ with $\omega\sim \phi_{\mathbb{Z}^d,p}$ and $\eta$ selected as an even subgraph of $\omega$ chosen uniformly at random - and this pair is the weak limit of the coupling in \Cref{Thm:The_Coupling} with $A=\emptyset$. In particular, the marginal $\eta$ is the weak limit of $\ell^{\emptyset}_{G_n,x}$ for $G_n\nearrow \mathbb{Z}^d.$ We refer to \cite[Section 3]{hansen2023uniform} for a general treatment of uniform even subgraphs in infinite volume with some highlights repeated in this paper in our treatment of the $q$-flow model in \Cref{sec:qflow}, and some marginally new input appearing in Appendix \ref{sec:A_priori_inf}. Therefore, similarly to \eqref{O(1)_as_ueg}, we define
\begin{equation} \label{O(1)_as_inf_ueg}
\ell_{\mathbb{Z}^d,x}[\;\cdot\;]=\phi_{\mathbb{Z}^d,p}[\operatorname{UEG}_{\omega}[\;\cdot\;]].
\end{equation}
We note that \eqref{SMP} transfers to the infinite volume limit automatically for the loop O($1$) model, and holds for the FK model due to uniqueness of the infinite cluster \cite{BurtonKeane}.

In this paper, we primarily study the supercritical phase. By techniques dating back to Peierls \cite{peierls1936ising}, when $d\geq 2,$ there exists $p_c\in (0,1)$ such that $\phi_{\mathbb{Z}^d,p}[0\cc\infty]=0$ for $p<p_c$ and $\phi_{\mathbb{Z}^d,p}[0\cc\infty]>0$ for $p>p_c,$ where $\{0\cc \infty\}$ denotes the event that the origin lies in an infinite connected component. With \Cref{Thm:The_Coupling} in hand, we will attach to this parameter the corresponding critical parameter $x_c=\frac{p_c}{2-p_c}$ for the loop O($1$) model.

Finally, we will discuss the relationship between the random current and loop O($1$) model. A current on a graph $G=(V,E)$ is a function $\mathbf{n}\in\mathbb{N}_0^E$.  To each current, we naturally associate two percolation configurations, which we will occasionally call the traced current and the odd part of the current. They are defined through 
$$
\hat{\mathbf{n}}_e=\id_{\mathbf{n}_e\geq 1} \qquad \mathrm{and}\qquad \mathbf{n}_{e}^{\mathtt{odd}}=\id_{\mathbf{n}_e\;\mathrm{odd}}.
$$
The random current with source set $A$ at inverse temperature $\beta>0$ on the finite graph $G=(V,E)$ is the measure on $\mathbb{N}_0^E$ given by
$$
\mathbf{P}^A_{G,\beta}[\mathbf{n}]\propto \id_{\partial \mathbf{n}^{\mathtt{odd}}=A}\prod_{e\in E} \frac{\beta^{\mathbf{n}_e}}{\mathbf{n}_e!}.
$$
One elementarily checks that 
\begin{align*}\mathbf{P}^A_{G,\beta}[\hat{\mathbf{n}},\mathbf{n}^{\mathtt{odd}}]&\propto\id_{\hat{\mathbf{n}}\supseteq\mathbf{n}^{\mathtt{odd}},\partial \mathbf{n}^{\mathtt{odd}}=A} (\cosh(\beta)-1)^{|\hat{\mathbf{n}}\setminus\mathbf{n}^{\mathtt{odd}}|}\sinh(\beta)^{|\mathbf{n}^{\mathtt{odd}}|} \\
&\propto \id_{\hat{\mathbf{n}}\supseteq\mathbf{n}^{\mathtt{odd}},\partial \mathbf{n}^{\mathtt{odd}}=A} \left(\frac{1}{\cosh(\beta)}\right)^{|E\setminus \hat{\nn}|}\left(1-\frac{1}{\cosh(\beta)}\right)^{|\hat{\mathbf{n}}\setminus\mathbf{n}^{\mathtt{odd}}|}\tanh(\beta)^{|\mathbf{n}^{\mathtt{odd}}|},
\end{align*}
from which one immediately gets that the odd part of the single current has the law of the loop O(1) model,
\begin{equation}\label{dd_extraction}
\mathbf{P}^A_{G,\beta}[\mathbf{n}^{\mathtt{odd}}\in \cdot \;]=\ell^A_{G,x}[\;\cdot\;], \qquad x=\tanh(\beta)
\end{equation}
and that one can sprinkle loop O(1) to get the traced current, 
\begin{equation} \label{Sprinkle_to_RC} \mathbf{P}^A_{G,\beta}[\hat{\mathbf{n}}\in\cdot\;]=\ell^A_{G,x}\cup \mathbb{P}_{G,1-\frac{1}{\cosh(\beta)}}[\;\cdot\;]=\ell^A_{G,x}\cup \mathbb{P}_{G,1-\sqrt{1-x^2}}[\;\cdot\;].
\end{equation}

Since the only interactions in $\mathbf{n}$ are given by the constraint $\mathbf{n}^{\mathtt{odd}}\in \Omega_A,$ one sees that the conditional distribution 
\begin{align} \label{Resampling density}\mathbf{P}^A_{G,\beta}[\mathbf{n}\mid \mathbf{n}^{\mathtt{odd}}]\propto \prod_{e\in \mathbf{n}^{\mathtt{odd}}} \id_{\mathbf{n}_e\in 2\mathbb{N}_0+1}\frac{\beta^{\mathbf{n}_e}}{\mathbf{n}_e!}\prod_{e\in E\setminus \mathbf{n}^{\mathtt{odd}}} \id_{\mathbf{n}_e\in 2\mathbb{N}_0} \frac{\beta^{\mathbf{n}_e}}{\mathbf{n}_e!}
\end{align}
is a product measure.
In fact, the conditional probabilities tell us how to sample a current from a loop O($1$) configuration in a robust way. This is standard, see e.g., \cite[Equation 4.5]{Hutchcroft_continuity_2023}. In the following, for fixed $\beta>0,$ let $f_{\mathtt{even}},f_{\mathtt{odd}}:[0,1]\to \mathbb{N}_0$ be given by 
\begin{align*}
    f_{\mathtt{even}}(u) &=f_{\mathtt{even}}(u,\beta)=\min\{n \in 2\mathbb{N}_0\mid \sum_{k=0}^{n/2} \frac{\beta^{2k}}{(2k)!}\geq u\cosh(\beta) \} \\
    f_{\mathtt{odd}}(u) &=f_{\mathtt{odd}}(u,\beta)=\min\{n \in 2\mathbb{N}_0+1\mid \sum_{k=0}^{(n-1)/2} \frac{\beta^{2k+1}}{(2k+1)!}\geq u\sinh(\beta) \}.
\end{align*}
It is immediate that if $U$ is a uniform variable on $[0,1],$ then $f_{\mathtt{even}}(U),$ respectively $f_\mathtt{odd}(U),$ is supported on $2\mathbb{N}_0,$ respectively $2\mathbb{N}_0+1,$ and for any $n\in \mathbb{N}_0,$ 
\begin{align} \label{Parity weights}
\mathbb{P}[f_{\mathtt{even}}(U)=2n]&\propto \frac{\beta^{2n}}{(2n)!} \qquad \mathrm{and}\qquad
\mathbb{P}[f_{\mathtt{odd}}(U)=2n+1]\propto \frac{\beta^{2n+1}}{(2n+1)!}.
\end{align}

In particular, \eqref{dd_extraction} and \eqref{Resampling density} yield the following. Recall that $x=\tanh(\beta)$.
\begin{coupling} \label{Loop_to_curren_sampling}
Let $G$ be a finite graph,   $(U_e)_{e\in E(G)}$ be i.i.d. uniforms in $[0,1]$ and $\eta$ be an independent sample of $\ell^A_{G,x}$ for $A\subseteq V,$ $x\in [0,1]$ . Then,  the random variables
$$
\mathbf{n}_e=(1-\eta_e)f_{\mathtt{even}}(U_e)+\eta_e f_{\mathtt{odd}}(U_e)
$$
form a random current $\mathbf{n}\sim \mathbf{P}^A_{G,\beta}$.
\end{coupling}
For completeness, note that the random current measure also has a Spatial Markov Property,
\begin{align}\label{eq:random_current_SMP}
\Prbcur^{A}_{G,\beta}[\nn\vert_H\mid \nn\vert_{G\setminus H}]=\Prbcur_{H,\beta}^{A\triangle\partial (\nn\vert_{G\setminus H})}[\nn\vert_H]. \tag{SMP Current}
\end{align}
As in \eqref{eq:definition_of_Gibbs_measure_for_loop_O(1)}, the SMP can be used to define Gibbs measures of the single random current. 
\\
\\
Let  $\mathtt{UC}_n$ be the event that the annulus $\Lambda_n \setminus \Lambda_{n/2}$ has at most one crossing cluster from inner to outer boundary.
Generalisations of the following argument will be used throughout the paper.  
\begin{lemma}\label{lemma:FAevent_and_unique_crossings}
Let $\omega \in \{0,1\}^{E(\Lambda_n)}$ and $A\subseteq \partial_v \Lambda_n$, then
\begin{center}
$\omega \in \mathcal{F}_A \cap \mathtt{UC}_n$ if and only if $\omega\vert_{\Lambda_n \setminus \Lambda_{n/2}} \in \mathcal{F}_A \cap \mathtt{UC}_n$. 
\end{center}
\end{lemma}
\begin{proof} Note that the event $\mathtt{UC}_n$ is measurable with respect to the configuration in $\Lambda_{n}\setminus \Lambda_{n/2}$.
 Thus, the "if" part follows from the fact that $\mathcal{F}_A$ is increasing. 
 Suppose that $\omega \in \mathcal{F}_A \cap \mathtt{UC}_n$. For each cluster $\mathcal{C}$ of $\omega$, as $\abs{\mathcal{C} \cap A}$ is even, choose an arbitrary (possibly empty) pairing of its vertices. For each pair $\mathfrak{p}=(v,w)$, there is a path $\gamma_{\mathfrak{p}}$ connecting $v$ and $w$  inside $\omega\vert_{\Lambda_n \setminus \Lambda_{n/2}}$. 
Indeed, if $\mathcal{C}$ does not cross the annulus $\Lambda_n \setminus \Lambda_{n/2},$ this must necessarily be the case. Conversely, if $\mathcal{C}$ crosses the annulus, it follows, since $\omega \in \mathtt{UC}_n$. 
 Define $\eta=\Delta_{\mathfrak{p}} \gamma_{\mathfrak{p}}$. Then $\eta \subseteq \omega\vert_{\Lambda_n \setminus \Lambda_{n/2}}$ and $\partial \eta = \Delta_{\mathfrak{p}} (\partial \gamma_{\mathfrak{p}})= A$, so $\omega\vert_{\Lambda_n \setminus \Lambda_{n/2}} \in \mathcal{F}_A$.
\end{proof}
The following lemma is related to \cite[Proposition 3.6 \& Corollary 3.7]{hansen2023uniform} and \Cref{Separating_Surface_Prop} below. It shows how the unique crossing event decorrelates $ \operatorname{UG}^A$, the uniform subgraph with sources $A$. 
\begin{lemma}\label{lemma:UGA_to_UEG}
    For any $A \subset \partial_v \Lambda_n$, $\omega \in  \mathcal{F}_A \cap \mathtt{UC}_n$, the following marginals agree,
    $$
    \operatorname{UG}^A_{\omega}\mid_{\Lambda_{n/2}} = \operatorname{UEG}_{\omega}\mid_{\Lambda_{n/2}}.
    $$
Moreover, if $F \in \mathcal{A}_{n/2}$, then the random variable $\id_{\mathtt{UC}_n(\omega)} \UEG_\omega[F]$ is measurable with respect to $\omega\vert_{\Lambda_n}$.
\end{lemma}
\begin{proof}
    By \Cref{lemma:FAevent_and_unique_crossings}, there exists a subgraph $\eta_0 \subseteq \omega\vert_{\Lambda_n\setminus\Lambda_{n/2}}$ with $\partial \eta_0 =A$. The map $\eta \mapsto \eta \triangle \eta_0$ is a bijection between $\Omega_{\emptyset}(\omega)$ and $\Omega_A(\omega)$ and it satisfies that $\eta\vert_{\Lambda_{n/2}} = (\eta \triangle \eta_0)\vert_{\Lambda_{n/2}}$.

For the second part, note that by \eqref{SMP},
$\operatorname{UEG}_\omega[F \mid  \eta\vert_{\Lambda_n^c} = \eta_0] = \operatorname{UG}^{\partial \eta_0}_{\omega\vert_{\Lambda_n}}[F] = \operatorname{UEG}_{\omega\vert_{\Lambda_n}}[F].$ 
\end{proof}

\subsection{On the use of constants} In this paper, several constants will appear throughout proofs. Our general approach is to try to mark a constant $c$ changing by e.g., writing $c'$. After the change has happened, we will revert back to writing $c$ to prevent notational bloat. In cases where multiple constants are imported from different propositions, we will always uniformly pick one constant which satisfies both.

\section{Proofs of Main Theorems}\label{sec:proofs_of_main_theorems}
\subsection{From unique crossings to unique measures: Proof of \texorpdfstring{\Cref{thm:Unique_Loop}}{Theorem \ref*{thm:Unique_Loop}}} \label{sec:proofofmain}
Recall that $\UC_n$ denotes the event that the annulus $\Lambda_{n}\setminus\Lambda_{n/2}$ has at most one cluster crossing from inner to outer boundary. One of our main technical inputs is the following:

\begin{theorem}
\label{theorem:any_incerasing_event}
For $d\geq 2, p>p_c,$ there exists $c>0$ such that for any increasing event $\mathcal{F}$, and any boundary condition $\xi$,  
$$
\phi^\xi_{\Lambda_n,p}[\mathtt{UC}_n\mid \mathcal{F}]\geq 1-\exp(-cn).
$$ 
\end{theorem}

Let us deduce the main theorem from \Cref{theorem:any_incerasing_event}:
\begin{proof}[Proof of \Cref{thm:Unique_Loop}]

For $F\in  \mathcal{A}_{\Lambda_{n/8}},$  using first \eqref{O(1)_as_ueg} and \eqref{O(1)_as_inf_ueg}, and then \Cref{lemma:UGA_to_UEG},
\begin{align*}
|\ell^A_{\Lambda_n,x}[F]-\ell_{\mathbb{Z}^d,x}[F]| &=|\phi^0_{\Lambda_n,p}[\operatorname{UG}^A_{\omega}[F]\mid \mathcal{F}_A]-\phi_{\mathbb{Z}^d,p}[\operatorname{UEG}_{\omega}[F]]| \\
&\leq |\phi^0_{\Lambda_n,p}[\operatorname{UEG}_{\omega}[F]\mid \mathcal{F}_A,\mathtt{UC}_n]-\phi_{\mathbb{Z}^d,p}[\operatorname{UEG}_{\omega}[F]]|+1-\phi^0_{\Lambda_N,p}[\mathtt{UC}_n\mid \mathcal{F}_A]. \\
&\leq  \sup_{\xi} |\phi^{\xi}_{\Lambda_{n/2},p}[\operatorname{UEG}_{\omega^{\xi}}[F]]-\phi_{\mathbb{Z}^d,p}[\operatorname{UEG}_{\omega}[F]]|+\exp(-cn),
\end{align*}
where the second inequality is due to \Cref{theorem:any_incerasing_event} and \eqref{SMP} and \Cref{lemma:FAevent_and_unique_crossings}.
Next, we argue that $\sup_{\xi}|\phi^{\xi}_{\Lambda_{n/2},p}[\operatorname{UEG}_{\omega^{\xi}}[F]]-\phi_{\mathbb{Z}^d,p}[\operatorname{UEG}_{\omega}[F]]|\leq \exp(-c'n)$. This is verbatim the same as the proof of the weaker mixing result \cite[Theorem 1.3]{hansen2023uniform}, but we include the argument for completeness. The random variable $\id_{\mathtt{UC}_{n/4}(\omega)}\operatorname{UEG}_{\omega}[F]$ is positive, bounded above by $1,$ and by \Cref{lemma:UGA_to_UEG}, a measurable function of $\omega|_{\Lambda_{n/4}}.$ Therefore,
\begin{align*}
&\sup_{\xi} |\phi^{\xi}_{\Lambda_{n/2},p}[\operatorname{UEG}_{\omega^{\xi}}[F]]-\phi_{\mathbb{Z}^d,p}[\operatorname{UEG}_{\omega}[F]]|\\
\leq  &d_{TV}(\phi^{\xi}_{\Lambda_{n/2},p}[\omega|_{\Lambda_{n/4}}\in \cdot \;],\phi_{\mathbb{Z}^d,p}[\omega|_{\Lambda_{n/4}}\in \cdot \;])+2-\phi^{\xi}_{\Lambda_{n/2},p}[\mathtt{UC}_{n/4}]-\phi_{\mathbb{Z}^d,p}[\mathtt{UC}_{n/4}]\\
\leq & 3\exp(-c''n),
\end{align*}
where $d_{TV}$ denotes the total variation distance between probability measures, and we have used \Cref{theorem:any_incerasing_event} and weak mixing of the FK-Ising model \cite[Theorem 1.3]{duminil2020exponential}. Plugging back into our previous estimate, we get $[\ell^A_{\Lambda_n,x}[F]-\ell_{\mathbb{Z}^d,x}[F]|\leq \exp(-c'n)$ for a suitably adjusted constant.

Rearranging and applying \eqref{SMP} again, we get
\begin{align*}
&|\ell_{\mathbb{Z}^d,x}[F\cap \{\partial(\eta|_{\mathbb{Z}^d\setminus \Lambda_n})=A\}]-\ell_{\mathbb{Z}^d,x}[F]\ell_{\mathbb{Z}^d,x}[\partial(\eta|_{\mathbb{Z}^d\setminus \Lambda_n})=A]| \\
&=\ell_{\mathbb{Z}^d,x}[\partial(\eta|_{\mathbb{Z}^d\setminus \Lambda_n})=A]|\ell^A_{\Lambda_n,x}[F]-\ell_{\mathbb{Z}^d,x}[F]|
\leq \ell_{\mathbb{Z}^d,x}[\partial(\eta|_{\mathbb{Z}^d\setminus \Lambda_n})=A]\exp(-cn).
\end{align*}
The Markov property yields weak mixing, i.e., for general $F'\in \mathcal{A}_{\mathbb{Z}^d\setminus \Lambda_n}$
$$
|\ell_{\mathbb{Z}^d,x}[F\cap F']-\ell_{\mathbb{Z}^d,x}[F]\ell_{\mathbb{Z}^d,x}[F']|\leq \exp(-cn) \ell_{\mathbb{Z}^d,x}[F']. 
$$
 Ratio weak mixing now follows by a classical result of Alexander \cite[Theorem 3.3]{alexander1998weak} (note that the exponentially bounded controlling regions property is automatically satisfied by $\ell_{\mathbb{Z}^d,x}$ since its interaction is finite range). 
\end{proof}

\subsection{Consequences for random currents: Proof of \texorpdfstring{\Cref{thm:Free-mix}}{Theorem \ref*{thm:Free-mix}}}\label{sec:Inf_Rel}

 Let us first note the following elementary fact, which is useful for playing around with ratio mixing:

\begin{lemma} \label{Convexity_Lemma}
For any measure $\rho$, any at most countable index sets $I$ and $J$ and functions $(f_i)_{i\in I},$ $(g_j)_{j\in J}$ satisfying $0\leq f_i,g_j,$  that $0<\rho[ f_i],\rho[g_j]$ for all $i\in I,j\in J$ and that $\rho[\sum_i f_i],\rho[\sum_jg_j]<\infty,$ 
$$
\left|\frac{\rho [\sum_{i,j} f_ig_j]}{\rho[\sum_l f_l]\rho[\sum_kg_k]}-1\right|\leq \sup_{i,j}\left|\frac{\rho [f_ig_j]}{\rho[f_i]\rho[g_j]}-1\right|.
$$
\end{lemma}
\begin{proof}
Write
$$
\frac{\rho [\sum_{i,j} f_ig_j]}{\rho[\sum_l f_l]\rho[\sum_kg_k]}=\sum_{i,j} \frac{\rho[f_ig_j]}{\rho[f_i]\rho[g_j]} \frac{\rho[f_i]\rho[g_j]}{\rho[\sum_lf_l]\rho[\sum_kg_k]}.
$$
Since all terms of the sum are positive and $\sum_{i,j} \frac{\rho[f_i]\rho[g_j]}{\rho[{\sum_l f_l}]\rho[\sum_kg_k]}=1$, we conclude that
$$
\inf_{i,j}\frac{\rho[f_ig_j]}{\rho[f_i]\rho[g_j]}\leq \frac{\rho [\sum_{i,j} f_ig_j]}{\rho[\sum_l f_l]\rho[\sum_kg_k]}\leq \sup_{i,j} \frac{\rho[f_ig_j]}{\rho[f_i]\rho[g_j]},
$$
which yields the desired.
\end{proof}

Note that our definition of ratio weak mixing \eqref{eq:RWM} readily generalises to state spaces beyond $\{0,1\}.$ 

\begin{corollary} \label{Stable_RWM} Let $S$ and $T$ be measurable spaces and consider two probability measures $\mu$ and $\nu$ on $S^{E(\mathbb{Z}^d)}$ respectively $T^{E(\mathbb{Z}^d)}.$ If $\mu$ and $\nu$ are exponentially ratio weak mixing, then so is $\mu\otimes \nu$ as a measure on $(S\times T)^{E(\mathbb{Z}^d)}$.
\end{corollary}
\begin{proof}
For readability, we give the proof in case $S=T.$ The general proof is analogous.
For general events $F\in \mathcal{A}_{\Lambda_n}\otimes\mathcal{A}_{\Lambda_n}$ and $F'\in \mathcal{A}_{\mathbb{Z}^d\setminus \Lambda_{kn}}\otimes \mathcal{A}_{\mathbb{Z}^d\setminus \Lambda_{kn}}$ of positive probability and given $0<\varepsilon\leq \frac{1}{2}\min \{\mu\otimes \nu[F],\mu\otimes \nu[F']\} $, use that sets of the form $H\times H'$ form an intersection stable generating set of the product $\sigma$-algebra to write
$$
\mu\otimes \nu\left[\left|\id_F-\sum_{i=1}^{\infty} \id_{F_i\times H_i}\right|\right]\leq  \varepsilon \qquad  \mathrm{and}\qquad  \mu\otimes \nu \left[\left|\id_{F'}-\sum_{j=1}^{\infty} \id_{F'_j\times H'_j}\right|\right]\leq \varepsilon
$$
for $F_i,H_i\in \mathcal{A}_{\Lambda_n}$ and $F'_j,H'_j\in \mathcal{A}_{\mathbb{Z}^d\setminus \Lambda_{kn}}$. By \Cref{Convexity_Lemma} with $\rho=\mu\otimes \nu,$ $f_i=\id_{F_i\times H_i}$ and $g_j=\id_{F'_j\times H'_j}$ (for those indices where those functions are not $0$ $\mu\otimes \nu$-a.s.), 
\begin{align*}
    &\left|\frac{\mu\otimes \nu[F\cap F']}{\mu\otimes \nu[F]\mu\otimes \nu[F']}-1\right|\\
    \leq & \left|\frac{\mu\otimes\nu[\sum_{i,j} f_ig_j]}{\mu\otimes \nu[\sum_l f_l]\mu\otimes \nu[\sum_k g_k]}-1\right| + \left|\frac{\mu\otimes\nu[\sum_{i,j} f_ig_j]}{\mu\otimes \nu[\sum_l f_l]\mu\otimes \nu[\sum_k g_k]}-\frac{\mu\otimes\nu[F\cap F']}{\mu\otimes \nu[ F]\mu\otimes \nu[F']}\right| \\
    \leq & \sup_{i,j} \left|\frac{\mu[F_i\cap F'_j]\nu[H_i\cap H'_j]}{\mu[F_i]\mu[F'_j]\nu[H_i]\nu[H'_j]} -1\right|+\frac{((\mu\otimes \nu[F]+\varepsilon)(\mu\otimes \nu[F']+\varepsilon)+\varepsilon)}{(\mu\otimes\nu[F]-\varepsilon)^2(\mu\otimes \nu[F']-\varepsilon)^2}2\varepsilon \\
    \leq& C\exp(-c n)+\frac{5}{(\mu\otimes\nu[F]-\varepsilon)^2(\mu\otimes \nu[F']-\varepsilon)^2}2\varepsilon.
\end{align*}
Taking $\varepsilon$ to $0$ yields the desired.

\end{proof}

With stability of mixing and \Cref{thm:Unique_Loop} in hand, we are ready to deduce consequences for the random current. For a current $\mathbf{n},$ we write $\widehat{\mathbf{P}}^A_{G,\beta}=\ell^{A}_{G,x}\cup\mathbb{P}_{G,1-\sqrt{1-x^2}}$ for the distribution of $\hat{\mathbf{n}}$ (see \eqref{dd_extraction}).
\begin{proof}[Proof of \Cref{thm:Free-mix}.]
By \Cref{thm:Unique_Loop}, for any weak limit $\widehat{\mathbf{P}}_{\mathbb{Z}^d,\beta}$ of traced random current models, $$\widehat{\mathbf{P}}_{\mathbb{Z}^d,\beta}=\lim_{G_n\nearrow \mathbb{Z}^d} \widehat{\mathbf{P}}^{A_n}_{G_n,\beta}=\lim_{G_n\nearrow \mathbb{Z}^d} \ell^{A_n}_{G_n,x}\cup \mathbb{P}_{G_n,1 - \sqrt{1-x^2}} = \ell_{\mathbb{Z}^d,x}\cup \mathbb{P}_{\mathbb{Z}^d,1 - \sqrt{1-x^2}},$$
since the union is a continuous map $\cup:\{0,1\}^{E(\mathbb{Z}^d)}\times \{0,1\}^{E(\mathbb{Z}^d)}\to \{0,1\}^{E(\mathbb{Z}^d)}.$

Ratio weak mixing follows from the ratio weak mixing of $\ell_{\mathbb{Z}^d,x}\otimes \mathbb{P}_{\mathbb{Z}^d,1 - \sqrt{1-x^2}},$ which follows from \Cref{Stable_RWM}. Similarly, we get ratio weak mixing for $\widehat{\mathbf{P}}_{\mathbb{Z}^d,\beta}\otimes \widehat{\mathbf{P}}_{\mathbb{Z}^d,\beta}$. Deducing mixing for the full currents follows from ratio weak mixing of $\ell_{\mathbb{Z}^d,x}\otimes \mathtt{Unif}([0,1])^{\otimes E(\mathbb{Z}^d)}$ (which is another instance of \Cref{Stable_RWM}) and  \Cref{Loop_to_curren_sampling}.
\end{proof}

Note that the same proof also shows that exponential ratio weak mixing of the loop O(1) model implies exponential ratio weak mixing of the FK-Ising model cf. \eqref{FK_is_sprinkled_O(1)}. We did however use exponential ratio weak mixing of the FK-Ising in the proof of \Cref{thm:Unique_Loop}.

\subsection{Proof of \Cref{theorem:any_incerasing_event}}
This subsection is dedicated to the proof of \Cref{theorem:any_incerasing_event}. 
We note that the proof strategy is more or less that of \cite[Theorem 3.1]{Pis96}. 

The following lemma due to Pisztora \cite[Lemma 3.3]{Pis96} is central in the proof. Recall the slab of height $h$, denoted by $S_h = \mathbb{Z}^2\times \{0,\dots,h\}^{d-2}$.  
\begin{lemma}[\cite{Pis96}] \label{Long-range-slab}
 For $d\geq 3,$ $p>p_c$, there exists $h\in \mathbb{N}$ and a constant $c>0$ such that for all $R\in \mathbb{N}$ and all  $v,w\in S_h\cap \Lambda_R,$ we have $\phi^0_{S_h\cap \Lambda_R,p}[v\cc w]>c$.
\end{lemma}
From this, we immediately get:
\begin{corollary} \label{Long-range thin annulus}
For $d\geq 3, p>p_c,$ there exists $h\in \mathbb{N}$ and a constant $c>0$ such that for all $R\in \mathbb{N}$ and all $v,w\in \Lambda_{R+h}\setminus \Lambda_R,$ we have $\phi^0_{\Lambda_{R+h}\setminus \Lambda_R,p}[v\cc w]>c.$
\end{corollary}
\begin{proof}
This follows from \Cref{Long-range-slab}, repeated application of \eqref{FKG} and \eqref{SMP}.
\end{proof}

The positive connection probabilities in the rings of constant size $h$ can be iterated to get the following exponential decay estimate. 

\begin{lemma} \label{Pt-to-pt-decay}
For $d\geq 3,$ $p>p_c,$ there exists $c>0$ such that for any $n\in \mathbb{N}$ and any $v,w\in \partial_v \Lambda_n,$ any boundary condition $\xi$ and any increasing event $\mathcal{F},$  we have
$$
\phi^{\xi}_{\Lambda_{n},p}[v\cc \Lambda_{n/2},w\cc \Lambda_{n/2},\omega|_{\Lambda_{n}\setminus \Lambda_{n/2}}\in\{v \centernot\cc w \}\mid \mathcal{F}]\leq \exp(-cn).
$$
\end{lemma}
\begin{remark}
As $\{v\cc \Lambda_{n/2},w\cc \Lambda_{n/2},v\centernot\cc w\}$ is not increasing, the above lemma does not follow automatically from the unconditioned statement and, indeed, by Ornstein-Zernike theory (e.g., \cite[Theorem C]{OZFK}), it is false in the subcritical regime e.g. for $v=ne_1,$ $w=-ne_1$ and $\mathcal{F}=\{v\cc \Lambda_{n/2},w\cc \Lambda_{n/2}\}$.
\end{remark}

\begin{proof}
The lemma is first proven for $d \geq 3$, after which the case $d=2$ is handled separately. 
Let $h\in \mathbb{N}$ be as in \Cref{Long-range thin annulus} and let $\xi_k=\omega^{\xi}|_{\mathbb{Z}^d \setminus \Lambda_{n-kh}}$, for $0 \leq k \leq  \lfloor \frac{n}{2h} \rfloor$. 

Define the event $\mathcal{E}_k = \left \{ v \overset{\xi_{k}}\longleftrightarrow \Lambda_{n-kh},w\overset{\xi_{k}}\longleftrightarrow \Lambda_{n-kh}, \xi_{k}\in \{v\centernot\cc w\} \right \}$. 
Notice that as $\xi_k \subseteq \xi_{k+1}$, it holds that 
$\mathcal{E}_{k+1} \subseteq \mathcal{E}_{k}$. 
By \eqref{SMP}, \eqref{FKG} and \Cref{Long-range thin annulus}, 
\begin{align*}
\phi^{\xi}_{\Lambda_{n},p}[\mathcal{E}_{k} \mid \xi_{k-1},\mathcal{F}]
&=\phi^{\xi_{k-1}}_{\Lambda_{n-h(k-1)},p}[v\overset{\xi_{k}}{\longleftrightarrow} \Lambda_{n-kh},w \overset{\xi_{k}}{\longleftrightarrow} \Lambda_{n-kh}, v \overset{\xi_{k}}{\centernot\longleftrightarrow} w \mid \omega^{\xi_{k-1}}\in \mathcal{F}]\\
& \leq   \id[v\overset{\xi_{k-1}}{\cc} \Lambda_{n-(k-1)h},w\overset{\xi_{k-1}}{\cc} \Lambda_{n-(k-1)h}] \phi^{\xi_{k-1}}_{\Lambda_{n-h(k-1)},p}[ v \overset{\xi_{k}}{\centernot\longleftrightarrow} w \mid \omega^{\xi_{k-1}}\in \mathcal{F}] \leq  (1-c).
\end{align*}

Since, $\mathcal{E}_{k-1}$ is $\xi_{k-1}$-measurable, it follows that $\phi^{\xi}_{\Lambda_{n},p}[\mathcal{E}_{k} \mid F,\mathcal{E}_{k-1}]\leq 1-c$ and hence,

$$
\phi^{\xi}_{\Lambda_{n},p}[\mathcal{E}_k \mid \mathcal{F}]=
\phi^{\xi}_{\Lambda_{n},p}[\mathcal{E}_k \mid \mathcal{F},\mathcal{E}_{k-1}]\phi^{\xi}_{\Lambda_n,p}[\mathcal{E}_{k-1}\mid \mathcal{F}]\leq (1-c)\phi^{\xi}_{\Lambda_n,p}[\mathcal{E}_{k-1}\mid \mathcal{F}].
$$
Iterating $\lfloor \frac{n}{2h} \rfloor$ times yields the desired exponential decay. 
\end{proof}

\begin{proof}[Proof of \Cref{theorem:any_incerasing_event}]
Let us first do the proof for $d\geq 3$. If $\mathtt{UC}_n$ fails, there exist $v,w\in \partial_v\Lambda_n$ satisfying $v\cc \Lambda_{n/2},$ $w\cc\Lambda_{n/2},$ but which belong to distinct clusters. By a union bound and \Cref{Pt-to-pt-decay}, we get
$$
\phi^\xi_{\Lambda_n,p}[\mathtt{UC}_n\mid \mathcal{F}]\geq 1- Cn^{2(d-1)}\exp(-cn)\geq 1-\exp(-c'n)
$$
for an adjusted constant when $n$ is sufficiently large. Adjusting the constant to accommodate lower values of $n$ establishes the desired.

For $d= 2$, let $\mathtt{Circ}_n$ denote the event that there is a circuit of open edges in $\Lambda_n\setminus \Lambda_{n/2}.$ We note that
$$
\phi^\xi_{\Lambda_n,p}[\mathtt{UC}_n\mid \mathcal{F}]\geq \phi^\xi_{\Lambda_n,p}[\mathtt{Circ}_n\mid \mathcal{F}]\geq \phi^0_{\Lambda_n,p}[\mathtt{Circ}_n] =1-\phi^1_{\Lambda_n^*,p^*}[ \Lambda_{n/2}\cc \partial_v\Lambda_{n}]\geq 1-\exp(-Cn)
$$
where the second inequality is FKG, the equality is planar duality \cite[Proposition 2.17]{DC17}, and the third inequality is sharpness of the phase transition for the subcritical FK-model \cite{DCsharpness}.
\end{proof}

\subsection{Mixing, uniqueness of Gibbs measures and uniqueness of weak limits}
The statistical mechanics literature offers several classical approaches to infinite volume systems. The most naïve is to start with a model $\mu^{\xi}_G$ indexed by finite graphs $G$ and boundary conditions $\xi$ and try to take a weak limit $\lim_{G_n\nearrow \mathbb{Z}^d}\mu^{\xi_n}_{G_n}.$ 
However, often in statistical mechanics, one is concerned with models having some sort of Markov property.  
This, in turn, gives rise to a notion of Gibbs measures for the models, which are infinite volume measures sharing the same Markov property. The Backwards Martingale Convergence Theorem implies that any tail-trivial Gibbs measure must also be a weak limit in the previous sense\footnote{See e.g., the proof of \cite[Theorem 6.63]{friedli2018}.}, whereas the non tail-trivial ones are generally only weak limits if one allows the boundary conditions $\xi_n$ to be random. Conversely, if the interactions of the model are local, then the Markov property necessarily survives in the weak limit, and any weak limit is just a Gibbs measure. However, for models with non-local interactions, the two notions start being a priori different and this plays a role e.g., for the random-cluster model on non-amenable graphs \cite{Haggstrom1996}, and attempts have been made to remedy that \cite{halberstam2023uniqueness}. The general moral stands that uniqueness of weak limits is the stronger of the two.

Similarly, mixing statements for an infinite volume Gibbs measure $\mu$, with \eqref{eq:RWM} being among the strongest one might hope for, imply a certain indifference to boundary conditions for the finite volume measures $\mu^{\xi_n}_{G_n}.$ More precisely, if $\mu$ has the finite energy property\footnote{Meaning that the state of any given edge $e$ has full support on its state space conditionally on the state of all other edges}, and the interactions of the model are not too long range\footnote{We will allow ourselves to be vague as to exactly what counts, but the moral is hopefully clear.}, one may impose arbitrary boundary conditions $\xi_n$ on a finite graph $G_n$ under $\mu$ by hand. Thus, in this case, mixing will imply uniqueness of weak limits. As such, morally, mixing should be thought of as the strongest property discussed in this paper.

\section{The main theorems with bulk sources} \label{sec:how_to_handle_sources}

In this section, we show how our techniques can be extended to handling sources in the bulk. For any finite graph $G \subseteq \Z^d$ and $A \subseteq V(G)$ with $|A|$ odd,  let $\ell^{A,\delta}_{G,x}$ denote Bernoulli percolation at edge probability $\frac{x}{1+x}$ conditioned on the vertices of odd degree being equal to $A$ and an odd number of vertices in $\partial_v G$. For $A$ with $|A|$ even, we denote $\ell^{A,\delta}_{G,x}=\ell_{G,x}^A$. It is worth noting that for $\abs{A}$ odd, $\ell^{A,\delta}_{G,x}$ fits into \Cref{Thm:The_Coupling} as the uniform subgraph of $\omega\sim \phi^1_{G,p}$ with sources $A\cup \{\partial_v G\}$ (the latter is counted as a single vertex in $\kappa^1(\omega)$). The goal of this section is to prove the following theorem:
\begin{theorem} \label{thm:Unique_loop_with_sources}
For any $d\geq 2,$ $x>x_c$ and finite subset $A\subseteq V(\mathbb{Z}^d),$ the weak limit 

\noindent $\ell^A_{\mathbb{Z}^d,x}=\lim_{G_n\nearrow \mathbb{Z}^d} \ell_{G_n,x}^{A,\delta}$ exists, and for any $A_n\subseteq \partial _vG_n$ with $|A\cup A_n|$ even, 
$$
\ell^A_{\mathbb{Z}^d,x}=\lim_{G_n\nearrow \mathbb{Z}^d}\ell^{A\cup A_n}_{G_n,x}.
$$
Furthermore, $\ell^A_{\mathbb{Z}^d,x}$ is exponentially ratio weak mixing.
\end{theorem}

One may, just like before, introduce Gibbs measures for the loop O($1$) model with sources. In particular, \Cref{thm:Unique_loop_with_sources} says that this Gibbs measure is unique for any finite set of sources.

\Cref{thm:Unique_loop_with_sources} again yields the corresponding statement for the random current. For a finite  graph $G\subseteq \Z^d$ and $A\subseteq V(G)$ such that $\abs{A}$ is odd, define $\mathbf{P}^{A,\delta}_{G,\beta}$ to be an i.i.d. family $\nn_e$ of $\mathbf{Poi}(\beta)$ variables such that $\sum_{w: wv\in E(G)} \nn_{wv}$ is odd for $v \in A$ and an odd number of vertices on $\partial_v G$. For $\abs{A}$ even, we denote $\mathbf{P}^{A,\delta}_{G,\beta}=\mathbf{P}^{A}_{G,\beta}$. The version of \Cref{thm:Unique_loop_with_sources} for the random current is then:
\begin{theorem} \label{thm:Unique_cur_with_sources}
For any $d\geq 2,$ $\beta>\beta_c$  and finite subset $A\subseteq V(\mathbb{Z}^d),$ the weak limit

\noindent $\mathbf{P}^A_{\mathbb{Z}^d,\beta}=\lim_{G_n\nearrow \mathbb{Z}^d} \mathbf{P}^{A,\delta}_{G_n,\beta}$ exists and for any $A_n\subseteq \partial _vG_n$ with $|A\cup A_n|$ even, 
$$
\mathbf{P}^A_{\mathbb{Z}^d,\beta}=\lim_{G_n\nearrow \mathbb{Z}^d}\mathbf{P}^{A\cup A_n}_{G_n,\beta}.
$$
Furthermore, $\mathbf{P}^A_{\mathbb{Z}^d,\beta}$ is exponentially ratio weak mixing, and so is $\mathbf{P}^A_{\mathbb{Z}^d,\beta}\otimes \mathbf{P}^B_{\mathbb{Z}^d,\beta}$ for (possibly distinct) finite sets $A$ and $B$.
\end{theorem}
 For $A$ with $|A|$ odd, let $\mathcal{G}^n_{A}$ denote the event that $\omega|_{\Lambda_n}$ has a subgraph $\eta$ with $\partial \eta=A\cup B$ for some $B\subseteq \partial_v \Lambda_n$ with $|B|$ odd. Let $\mathcal{G}_{A}$ denote the infinite volume event that $\omega$ has a subgraph $\eta$ with $\partial \eta=A$. Note that for $\abs{A}$ odd, $\mathcal{G}^k_A\supseteq \mathcal{G}^{k+1}_A$ for every $k$ and that $\cap_{k=1}^{\infty} \mathcal{G}^k_A=\mathcal{G}_A.$ For $|A|$ even, let $\mathcal{G}^n_{A}= \{\omega\vert_{\Lambda_n} \in \mathcal{F}_A\}$ and note that 
 $\mathcal{G}_A=\cup_{n=1}^{\infty} \mathcal{G}^{n}_A.$ 

\begin{remark}
The two equivalent definitions of the event $\mathcal{F}_A$ in finite volume give rise to two distinct definitions in infinite volume. For instance, $\mathcal{G}_{\{v\}}$ can be satisfied by an infinite path starting at $v$, whereas it is impossible for $|\mathcal{C}_v\cap \{v\}|$ to be even. The notation $\mathcal{G}_A$ is chosen to avoid confusion, but \Cref{Limit_with_sources} carries for the other infinite volume definition of $\mathcal{F}_A$ as well (where every cluster intersects $A$ an even number of times) when $|A|$ is even. 
\end{remark}

A central input in the proof is the surface order large deviations originally due to Pisztora. 

\begin{definition}\label{def:Pisztora_event}
For a box $\Lambda_n$ (or a translate of it), Pisztora considered the event $\mathtt{Pis}_{\Lambda_n}(\varepsilon, \theta, L_0)$ defined as follows:
\begin{enumerate}
    \item[$i)$] There is a cluster $\Giant_{\Lambda_n} $ touching all faces of $\partial_v \Lambda_n$ and the box $\Lambda_{n/2}$. 
    \item[$ii)$] The cluster has a density: $\abs{\Giant_{\Lambda_n}}\geq (\theta- \varepsilon)|\Lambda_n|$.
    \item[$iii)$] Most other clusters have size bounded by $L_0:$ $\abs{ \left \{v \in \Lambda_n \setminus \mathtt{Giant}_{\Lambda_n} \mid \abs{\mathcal{C}_v} \geq L_0 \right \} } \leq \varepsilon |\Lambda_n|.$
\end{enumerate}
\end{definition}

When convenient, we will also refer to $\Giant_B$ as the \textbf{local giant}.
Pisztora's result  (combined  with Bodineau's) gives the following:
\begin{theorem}[\cite{Pis96}]\label{thm:pisztora}
Fix $d \geq 3, p > p_c(d),$ $\theta=\phi_{\mathbb{Z}^d,p}[0\cc \infty]$ and $\varepsilon = 2^{-d}\theta$. There exist  $L_0, N_0, c >0$ such that for any $n \geq N_0$ and any boundary condition $\xi$,
$$
\phi_{{\Lambda_n},p}^{\xi}[\mathtt{Pis}_{\Lambda_n}(\varepsilon, \theta, L_0) ] \geq 1 - e^{-c n^{d-1}}.
$$
\end{theorem}
Henceforth, we will fix $\theta=\phi_{\mathbb{Z}^d,p}[0\cc\infty]$ and abbreviate $\mathtt{Pis}_{\Lambda_n}=\mathtt{Pis}_{\Lambda_n}(2^{-d}\theta,\theta, L_0)$. One may note that the event $\mathtt{Pis}$ is, regrettably, non-increasing. This often causes technical difficulties when working with supercritical percolation models in higher dimension.

Let $\mathtt{Pis}_{r,R}$ denote the event that $\omega|_{B}\in \mathtt{Pis}_B$ for every translate $B$ of $\Lambda_{\frac{R-r}{2}}$ contained in $\Lambda_R\setminus \Lambda_r$ (there are $O(R^d)$ many such).  On $\mathtt{Pis}_{r,R},$ there is a unique giant cluster in $\Lambda_R\setminus \Lambda_r,$ which we refer to as $\mathtt{Giant}_{r,R}.$
We also let $\mathtt{UC}_{r,R}$ denote the event that there is a unique crossing cluster in the annulus $\Lambda_R \setminus \Lambda_r$. The following identity extends \Cref{lemma:FAevent_and_unique_crossings} to the case where $\abs{A}$ is odd.
\begin{lemma}\label{lemma:G_A_decoupling_lemma}
    For $A \subseteq \Lambda_r$, with $\abs{A}$ odd, the following localization identity holds
    \begin{align*}
        \mathcal{G}_A \cap \mathtt{UC}_{r,R} \cap \Pis_{r,R} \cap \{ \Giant_{r,R} \cc \infty \} =
        \mathcal{G}^{R}_A \cap \mathtt{UC}_{r,R} \cap \Pis_{r,R} \cap \{ \Giant_{r,R} \cc \infty \}. 
    \end{align*}
\end{lemma}
\begin{proof}
  One inclusion follows from  $\mathcal{G}_A \subseteq \mathcal{G}_A^{R}$. 
  For the other inclusion, let $\omega$ be contained in the right hand side, and suppose that $\eta_0 \subseteq \omega\vert_{\Lambda_{R}}$, is such that $\partial \eta_0 = A \cup B$, with $B \subset \partial \Lambda_{R}$  (with $\abs{B}$ odd). All the $b \in B$ which are not connected to the unique crossing cluster inside $\omega\vert_{\Lambda_{R}\setminus \Lambda_r}$ can be paired (using paths inside $\omega\vert_{\Lambda_{R}\setminus \Lambda_r}$), denote the symmetric difference of these paths $\gamma_0$. All the points in $B \cap \Giant_{r,R}$ except for one (call it $b_0$) can also be paired. Call the symmetric difference of all these paths $\gamma_1$.  
  Since $b_0 \in \Giant_{r,R}$ and $\Giant_{r,R} \cc \infty$, one can take an open self-avoiding path $\gamma_2$ from $b_0 $ to $\infty$, then $\partial \gamma_2 = \{b_0\}$.
  Then $\eta =\eta_0 \triangle \gamma_0 \triangle \gamma_1 \triangle \gamma_2$ is open in $\omega$ and satisfies that $\partial \eta = A$, hence $\omega \in \mathcal{G}_A$. 
\end{proof}

Now, we are in position to prove exponential ratio weak mixing. The strategy is as follows: First, we establish that $\mathcal{G}_A$ has positive probability uniformly  in possible boundary conditions. This, again, allows us to inherit all the high probability events from the unconditioned measure. Using this, we argue that, due to the presence of giant clusters, uniformly in boundary conditions, the marginal on a big annulus is insensitive to the conditioning on $\mathcal{G}_A$ and thus, looks like the unconditioned infinite volume measure. Once this has been established, another appeal to the giant cluster allows one to decouple a box around the origin from boundary conditions at macroscopic distance.  The strategy is illustrated in \Cref{fig:Mixing_strat}.

\begin{figure}
    \centering
\begingroup%
  \makeatletter%
  \providecommand\color[2][]{%
    \errmessage{(Inkscape) Color is used for the text in Inkscape, but the package 'color.sty' is not loaded}%
    \renewcommand\color[2][]{}%
  }%
  \providecommand\transparent[1]{%
    \errmessage{(Inkscape) Transparency is used (non-zero) for the text in Inkscape, but the package 'transparent.sty' is not loaded}%
    \renewcommand\transparent[1]{}%
  }%
  \providecommand\rotatebox[2]{#2}%
  \newcommand*\fsize{\dimexpr\f@size pt\relax}%
  \newcommand*\lineheight[1]{\fontsize{\fsize}{#1\fsize}\selectfont}%
  \ifx\svgwidth\undefined%
    \setlength{\unitlength}{530.07874016bp}%
    \ifx\svgscale\undefined%
      \relax%
    \else%
      \setlength{\unitlength}{\unitlength * \real{\svgscale}}%
    \fi%
  \else%
    \setlength{\unitlength}{\svgwidth}%
  \fi%
  \global\let\svgwidth\undefined%
  \global\let\svgscale\undefined%
  \makeatother%
  \begin{picture}(1,0.34759358)%
    \lineheight{1}%
    \setlength\tabcolsep{0pt}%
    \put(0,0){\includegraphics[width=\unitlength,page=1]{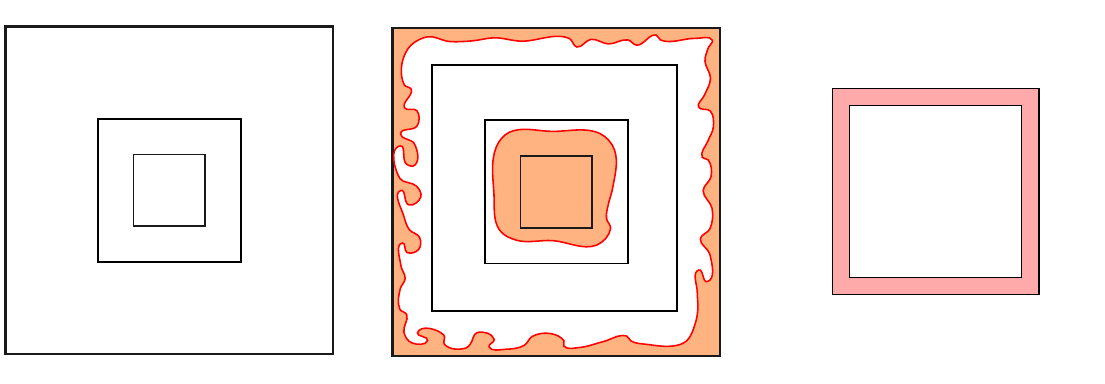}}%
    \put(0.02720205,0.06206929){\makebox(0,0)[lt]{\lineheight{1.25}\smash{\begin{tabular}[t]{l}$\Lambda_{8N}$\end{tabular}}}}%
    \put(0.17546408,0.12031663){\makebox(0,0)[lt]{\lineheight{1.25}\smash{\begin{tabular}[t]{l}$\Lambda_k$\end{tabular}}}}%
    \put(0.14492314,0.17361987){\makebox(0,0)[lt]{\lineheight{1.25}\smash{\begin{tabular}[t]{l}$A$\end{tabular}}}}%
    \put(0,0){\includegraphics[width=\unitlength,page=2]{Decoupling_pic.pdf}}%
    \put(0.11627545,0.33041318){\makebox(0,0)[lt]{\lineheight{1.25}\smash{\begin{tabular}[t]{l}$\xi$\end{tabular}}}}%
    \put(0,0){\includegraphics[width=\unitlength,page=3]{Decoupling_pic.pdf}}%
    \put(0.83895695,0.17258278){\makebox(0,0)[lt]{\lineheight{1.25}\smash{\begin{tabular}[t]{l}$A$\end{tabular}}}}%
    \put(0,0){\includegraphics[width=\unitlength,page=4]{Decoupling_pic.pdf}}%
    \put(0.81030926,0.3293761){\makebox(0,0)[lt]{\lineheight{1.25}\smash{\begin{tabular}[t]{l}$\xi$\end{tabular}}}}%
  \end{picture}%
\endgroup%

    \caption{The rough strategy for proving mixing conditioned on $\mathcal{G}_A$. Since, uniformly in the boundary condition $\xi,$ giants exist and connect to each other with good probability (depicted in the centre), the marginal of the conditioned measure on the smaller pink annulus, depicted right, looks like the marginal from the unconditioned measure, which, in turn, looks like the marginal of the infinite volume, unconditioned measure.}
    \label{fig:Mixing_strat}
\end{figure}

\begin{theorem} \label{Limit_with_sources}
For $d\geq 2,$ $p>p_c,$ the measure $\phi_{\mathbb{Z}^d,p}[\;\cdot\mid  \mathcal{G}_{A}]$ is exponentially ratio weak mixing.
\end{theorem}
\begin{proof}
We give the proof in the case $\abs{A}$ odd. For $\abs{A}$ even similar considerations yield a simplified proof. 

We start by arguing that there exists $c>0$ such that for any $N\geq 2k$ and any $\xi\in (\partial \Lambda_{8N}\cc \infty),$ we have that
\begin{align} \label{eq:GAprob_bound}
\phi^{\xi}_{\Lambda_{8N},p}[\omega^{\xi}\in\mathcal{G}_A]\geq c^{|A|}.
\end{align} 
This follows from \eqref{FKG} and \eqref{SMP},  since $\phi^0_{\Lambda_{8N}\setminus \Lambda_k,p}[v\cc w]\geq c'$ uniformly in $n \geq k$ and $v,w\in \Lambda_{2n}\setminus \Lambda_k$. For $d\geq 3,$ this may be obtained by repeated application of percolation in slabs \Cref{Long-range-slab}, \eqref{eq_CBC} and \eqref{FKG}. If $d=2,$ one gets it by the FKG inequality and repeated application of the fact that $1-\phi^0_{\Lambda_N,p}[\mathtt{Circ}_N]\leq \exp(-CN)$, as in the proof of \Cref{Pt-to-pt-decay}.

For $d\geq 3$, we will first argue that the marginal on a large subannulus of $\Lambda_{8N}\setminus \Lambda_N$ is insensitive to both the conditioning on $\mathcal{G}_A$ and on the exterior configuration $\xi.$ To this end, let $\boldsymbol{P}$ denote the following coupling: Sample $\xi_0$ according to marginal of $\phi^{\xi}_{\Lambda_{8N},p}[\;\cdot\mid \omega^{\xi}\in \mathcal{G}_A]$ on $(\Lambda_{2N}\setminus \Lambda_N)\cup(\Lambda_{8N}\setminus \Lambda_{6N})$ and given $\xi_0,$ sample $\psi\sim \phi_{\Lambda_{6N}\setminus \Lambda_{2N},p}[\omega|_{\Lambda_{5N}\setminus\Lambda_{3N}}\in \cdot\; \mid \xi_0],$ where we abuse notation and write $\xiA$ for the event $\{ \omega\vert_{\Lambda_{2N}\setminus \Lambda_{N} \cup \Lambda_{8N}\setminus \Lambda_{6N}}=\xiA\}.$  Note that this conditional distribution is equal to $\phi^{\xi_0}_{\Lambda_{6N}\setminus \Lambda_{2N},p}[\omega|_{\Lambda_{5N}\setminus\Lambda_{3N}}\in \cdot \;]$ for $\xi_0\in \mathtt{UC}_{6N,8N}\cap \mathtt{UC}_{N,2N}$. Given $\xi_0$ and $\psi,$ sample $$
\psi_A\sim\phi^{\xi}_{\Lambda_{8N},p}[\omega|_{\Lambda_{5N}\setminus \Lambda_{3N}}\in  \cdot\; \mid \xi_0,\omega^{\xi}\in \mathcal{G}_A] \qquad \mathrm{and} \qquad \psi_{\infty}\sim \phi_{\mathbb{Z}^d,p}[\omega|_{\Lambda_{5N}\setminus\Lambda_{3N}}\in \cdot\;]
$$ conditionally independently subject to maximising the probabilities that $\psi_A=\psi$ and $\psi_{\infty}=\psi$. Note that by the Law of Total Probability, $\psi_{\infty}\sim \phi_{\mathbb{Z}^d,p}[\omega|_{\Lambda_{5N}\setminus \Lambda_{3N}}\in \cdot\;] $ and $\psi_A\sim \phi^{\xi}_{\Lambda_{8N},p}[\omega|_{\Lambda_{5N}\setminus \Lambda_{3N}}\in \cdot\mid \omega^{\xi}\in \mathcal{G}_A]$.

Let $\mathtt{Bad}$ denote the event that either $\xi_0\not\in \mathtt{Pis}_{6N, 8N}\cap \mathtt{UC}_{6N,8N}\cap \mathtt{Pis}_{N,2N}\cap \mathtt{UC}_{N,2N}$ or that the giants fail to satisfy $\phi_{\Lambda_{6N}\setminus \Lambda_{2N},p}^{\xiA}[\mathtt{Giant}_{N,2N}(\xiA) \cc \mathtt{Giant}_{6N,8N}(\xiA)]\geq 1-2\exp(-C'N)$ for a constant $C'>0$ to be determined. We will argue the existence of a constant $C>0$ satisfying the two following inequalities:
\begin{align}
\boldsymbol{P}[\xi_0\in \mathtt{Bad}] &\leq \exp(-CN) \label{eq:Bad_is_Unlikely}\\
\boldsymbol{P}[\psi_A=\psi=\psi_{\infty}\mid \xi_0] &\geq 1-\exp(-CN) \qquad \forall \xi_0\not\in\mathtt{Bad}, \label{eq:Not_Bad_is_Good}
\end{align}
from which it follows by a union bound that
\begin{align}
&\frac{1}{2}d_{TV}(\phi^{\xi}_{\Lambda_{8N},p}[ \omega|_{\Lambda_{5N}\setminus \Lambda_{3N}}\in \cdot\mid \omega ^{\xi}\in \mathcal{G}_A], \phi_{\mathbb{Z}^d,p}[\omega|_{\Lambda_{5N}\setminus \Lambda_{3N}} \in \cdot \;])
\leq  \boldsymbol{P}[\psi_A\neq\psi_{\infty}]\leq 2\exp(-CN).\label{eq:annulus_marginals_close_to_unconditioned}  
\end{align}

To estimate the probability of $\mathtt{Bad},$ first note that
\begin{align}\label{eq:high_probability_events_proof}
\phi^{\xi}_{\Lambda_{8N},p}[\mathtt{Pis}_{6N, 8N}, \mathtt{UC}_{6N,8N}, \mathtt{Pis}_{N,2N},\mathtt{UC}_{N,2N}\mid \omega^{\xi}\in\mathcal{G}_A]\geq 1- \exp(-C'N)
\end{align}
uniformly in $\xi$ by \eqref{eq:GAprob_bound}.

Two applications of \Cref{lemma:G_A_decoupling_lemma} show that the crossing giants localise the infinite volume event $\mathcal{G}_A$: 
\begin{align*}
    \mathcal{G}_A \cap \mathtt{UC}_{N,2N} \cap \Pis_{N,2N} \cap \mathtt{UC}_{6N,8N}\cap \Pis_{6N,8N} \cap \{\Giant_{N,2N} \cc \Giant_{6N,8N}\} \cap \{\Giant_{6N,8N} \cc \infty \}\\
    =   \mathcal{G}^{2N}_A \cap \mathtt{UC}_{N,2N} \cap \Pis_{N,2N} \cap \mathtt{UC}_{6N,8N}\cap \Pis_{6N,8N} \cap \{\Giant_{N,2N} \cc \Giant_{6N,8N}\} \cap \{\Giant_{6N,8N}\cc \infty \}.
\end{align*}

Thus, if $\xiA$ is a configuration in the annuli $\Lambda_{2N}\setminus \Lambda_{N}$ and $\Lambda_{8N}\setminus \Lambda_{6N}$, such that $\xiA\in\mathtt{Pis}_{6N, 8N}\cap \mathtt{UC}_{6N,8N}\cap \mathtt{Pis}_{N,2N} \cap\mathtt{UC}_{N,2N}$, then 
\begin{align*}
\xiA \cap \{ \omega^\xi \in \mathcal{G}_A\} =   \xiA \cap \mathcal{G}_A^{2N} \cap \{\mathtt{Giant}_{N,2N}(\xiA)\cc \mathtt{Giant}_{6N,8N}(\xiA) \} \cap \{\Giant_{6N,8N}(\xiA) \overset{\xi}{\cc} \infty \}.
\end{align*}

So for each $\xi_0 \in\mathtt{Pis}_{6N, 8N}\cap \mathtt{UC}_{6N,8N}\cap \mathtt{Pis}_{N,2N} \cap\mathtt{UC}_{N,2N}$, 
\begin{align}
\label{eq:G_Aisconditionedonconnected_giants}
\phi^{\xi}_{\Lambda_{8N},p}[\; \omega\vert_{\Lambda_{6N}\setminus \Lambda_{2N}} \in \cdot \mid \xiA,\omega^{\xi}\in \mathcal{G}_A]=\phi_{\Lambda_{6N}\setminus \Lambda_{2N},p}^{\xiA}[\; \cdot \mid \mathtt{Giant}_{N,2N}(\xiA)\cc \mathtt{Giant}_{6N,8N}(\xiA)].
\end{align}

Since, by the law of total probability,
\begin{align*}
 1-\exp(-CN)&\leq 
\phi^{\xi}_{\Lambda_{8N},p}[\mathtt{Giant}_{N,2N}\cc \mathtt{Giant}_{6N,8N} ]\\
&=\sum_{\xiA} \phi_{\Lambda_{6N}\setminus \Lambda_{2N},p}^{\xiA}[\mathtt{Giant}_{N,2N}(\xiA) \cc \mathtt{Giant}_{6N,8N}(\xiA)] \phi^{\xi}_{\Lambda_{8N},p}[\xiA]
\end{align*}
and $\phi_{\Lambda_{6N}\setminus \Lambda_{2N},p}^{\xiA}[ \mathtt{Giant}_{N,2N}(\xiA) \cc \mathtt{Giant}_{6N,8N}(\xiA)]\leq 1$, the set 
\begin{align*} 
   \left \{ \xiA  \mid \phi_{\Lambda_{6N}\setminus \Lambda_{2N},p}^{\xiA}[\mathtt{Giant}_{N,2N}(\xiA) \cc \mathtt{Giant}_{6N,8N}(\xiA)]\geq 1-2\exp(-CN) \right \}
\end{align*}must have probability at least $1-2\exp(-CN)$ under $\phi^{\xi}_{\Lambda_{8N,p}}$ (and this determines the constant in the definition of $\mathtt{Bad}$).
It follows by \eqref{eq:GAprob_bound} that the probability is at least $1-2c^{-|A|}\exp(-CN)$ under $\phi ^{\xi}_{\Lambda_{8N,p}}[\;\cdot\mid \omega^{\xi}\in\mathcal{G}_A].$  Combining this with \eqref{eq:high_probability_events_proof} and \eqref{eq:G_Aisconditionedonconnected_giants}  yields \eqref{eq:Bad_is_Unlikely}.

Furthermore, \eqref{eq:G_Aisconditionedonconnected_giants} gives us that for $\xiA\not \in \mathtt{Bad},$ we have
\begin{align}\label{eq:set_of_good_xi}
d_{TV}(\phi_{\Lambda_{6N}\setminus \Lambda_{2N},p}^{\xiA}[\; \cdot \mid  \mathtt{Giant}_{N,2N}(\xiA) \cc \mathtt{Giant}_{6N,8N}(\xiA)],\phi_{\Lambda_{6N}\setminus \Lambda_{2N},p}^{\xiA}[\; \cdot\;])\leq 4\exp(-CN).
\end{align}

Exponential ratio mixing of the supercritical random-cluster model gives  
\begin{align}\label{eq:exp_ratio_boundary_proof}
    d_{TV}(\phi_{\Lambda_{6N}\setminus \Lambda_{2N},p}^{\xiA}[\omega|_{\Lambda_{5N}\setminus \Lambda_{3N}}\in \cdot \;],\phi_{\mathbb{Z}^d,p}[\omega|_{\Lambda_{5N}\setminus \Lambda_{3N}}\in \cdot \;])\leq \exp(-C'N), 
\end{align}
which, combined with \eqref{eq:G_Aisconditionedonconnected_giants} and \eqref{eq:set_of_good_xi}, yields \eqref{eq:Not_Bad_is_Good}.

Since $\xi$ was arbitrary, for $\xi,\xi'\in (\partial \Lambda_{8N}\cc \infty),$ \eqref{eq:annulus_marginals_close_to_unconditioned} implies the existence of a coupling $\mathscr{P}$ of $\omega_{\xi}\sim \phi^{\xi}_{\Lambda_{8N},p}[\;\cdot\mid \omega^{\xi}\in \mathcal{G}_A],$ $\omega_{\xi'}\sim\phi^{\xi'}_{\Lambda_{8N},p}[\;\cdot\mid \omega^{\xi'}\in \mathcal{G}_A]$ and $\omega\sim\phi_{\mathbb{Z}^d,p}$ such that 
$$
\mathscr{P}[\omega_{\xi}|_{\Lambda_{5N}\setminus \Lambda_{3N}}=\omega|_{\Lambda_{5N}\setminus \Lambda_{3N}}=\omega_{\xi'}|_{\Lambda_{5N}\setminus \Lambda_{3N}}]\geq 1-2\exp(-CN).
$$
Furthermore, if $\omega|_{\Lambda_{5N}\setminus \Lambda_{3N}}\in \mathtt{UC}_{3N,5N},$ then the conditional marginal on $\Lambda_{3N}$ is the same under $\phi^{\xi}_{\Lambda_{8N}}[\;\cdot\mid \omega^{\xi}\in \mathcal{G}_A]$ and $\phi^{\xi'}_{\Lambda_{8N}}[\;\cdot\mid \omega^{\xi'}\in \mathcal{G}_A]$. That is, $\mathscr{P}$ may be chosen such that
$$
\mathscr{P}[\omega_{\xi}|_{\Lambda_{5N}}=\omega_{\xi'}|_{\Lambda_{5N}}]\geq \mathscr{P}[\omega_{\xi}|_{\Lambda_{5N}\setminus \Lambda_{3N}}=\omega_{\xi'}|_{\Lambda_{5N}\setminus \Lambda_{3N}}=\omega|_{\Lambda_{5N}\setminus \Lambda_{3N}}, \omega\in \mathtt{UC}_{3N,5N} ]\geq 1-3\exp(-CN),
$$
where the last inequality is a union bound.

This establishes weak mixing. However, as we have also established Alexander's exponentially bounded controlling regions property (given by the unique crossings), we deduce ratio weak mixing from \cite[Theorem 3.3]{alexander1998weak}.

Finishing the proof for $d= 2$ is similar, except we replace the Pisztora events by events of the form $\mathtt{Circ}_{r,R}$ and define $\mathtt{Giant}_{r,R}$ to be the cluster of the outermost circuit in $\Lambda_R\setminus \Lambda_r$ enclosing $\Lambda_r.$
\end{proof}

With \Cref{Limit_with_sources} in hand, we are ready to prove \Cref{thm:Unique_loop_with_sources}.
In the following, for $x>x_c,$ and $A\subset V(\Z^d)$ finite, we define 
$$
\ell^A_{\mathbb{Z}^d,x}[\;\cdot\;]=\phi_{\mathbb{Z}^d,p}[\operatorname{UG}^A_{\omega}[\;\cdot\;]\mid \mathcal{G}_A].
$$
Our goal is then to prove that this \textit{is}, indeed, equal to the right weak limits.

\begin{proof}[Proof of \Cref{thm:Unique_loop_with_sources}]

For any source set $A_N\subseteq \partial_v \Lambda_N$ with $|A\cup A_N|$ even and any percolation configuration $\xi\in \{0,1\}^{E(\Lambda_N)}\cap \mathtt{UC}_N\cap \mathcal{F}_{A\cup A_N}$,
$$
\{\omega \in \{0,1\}^{E(\Lambda_{N/2})}\mid \omega^{\xi}\in \mathcal{F}_{A\cup A_N}\}=\{\omega \in \{0,1\}^{E(\Lambda_{N/2})}\mid \omega^{\xi}\in \mathcal{G}^N_A\}.
$$
An adaptation of \Cref{lemma:UGA_to_UEG}, shows that for $\omega \in \mathtt{UC}_N$ the marginal of $\operatorname{UG}^{A\cup A_N}_{\omega}$ on $\Lambda_{N/2}$ is equal to the marginal of $\operatorname{UG}^{A,\delta}_{\omega}$ on $\Lambda_{N/2},$ where $\operatorname{UG}_{\omega}^{A,\delta}$ denotes the uniform measure on subgraphs $\eta$ of $\omega$ with $\partial \eta= A\cup B$ for some $B \subseteq \partial_v \Lambda_N$.

 Accordingly, for $F\in  \mathcal{A}_{\Lambda_{N/8}},$  by \eqref{O(1)_as_ueg} and our definition of $\ell^A_{\mathbb{Z}^d,x}$,
\begin{align*}
|\ell^{A\cup A_N}_{\Lambda_N,x}[F]-\ell^A_{\mathbb{Z}^d,x}[F]| &=|\phi^0_{\Lambda_N,p}[\operatorname{UG}^{A\cup A_N}_{\omega}[F]\mid \mathcal{F}_{A\cup A_N}]-\phi_{\mathbb{Z}^d,p}[\operatorname{UG}^A_{\omega}[F]\mid \mathcal{G}_A]| \\
&\leq |\phi^0_{\Lambda_N,p}[\operatorname{UG}^{A\cup A_N}_{\omega}[F]\mid \mathcal{F}_{A\cup A_N} ,\mathtt{UC}_N]-\phi_{\mathbb{Z}^d,p}[\operatorname{UG}^A_{\omega}[F]\mid \mathcal{G}_A]|+1-\phi^0_{\Lambda_N,p}[\mathtt{UC}_N\mid \mathcal{F}_{A\cup A_N}] \\
&\leq  \sup_{\xi\in \mathtt{UC}_N} |\phi^{\xi}_{\Lambda_{N/2},p}[\operatorname{UG}^{A,\delta}_{\omega}[F] \mid \omega^{\xi}\in \mathcal{G}^{N}_A]-\phi_{\mathbb{Z}^d,p}[\operatorname{UG}^A_{\omega}[F]\mid \mathcal{G}_A]| +\exp(-cN),
\end{align*}
where the second inequality is due to \Cref{theorem:any_incerasing_event} and \eqref{SMP}. 

To finish, we once again argue as in the proof of \cite[Theorem 1.3]{hansen2023uniform}. The random variable $\id_{\operatorname{UC}_{N/4}}(\omega) \operatorname{UG}^{A,\delta}_{\omega}[F]$ is positive, bounded from above by 1, and is measurable with respect to $\omega|_{\Lambda_{N/4}}$ (as a random variable on the state space $\mathcal{G}^N_A)$. Furthermore, it is equal to $\id_{\operatorname{UC}_{N/4}}(\omega)\operatorname{UG}^{A}_{\omega}[F]$ supposing that $\omega\in \mathcal{G}_A$.
Hence, applying ratio weak mixing of $\phi_{\mathbb{Z}^d,p}[\;\cdot \mid \mathcal{G}_A],$ we get
\begin{align*}
&\sup_{\xi\in \mathtt{UC}_N} |\phi^{\xi}_{\Lambda_{N/2},p}[\operatorname{UG}^{A,\delta}_{\omega}[F] \mid \omega^{\xi}\in \mathcal{G}^{N}_A]-\phi_{\mathbb{Z}^d,p}[\operatorname{UG}^A_{\omega}[F]\mid \mathcal{G}_A]|\\
\leq& \sup_{\xi \in \mathtt{UC}_N} |\phi^{\xi}_{\Lambda_{N/2},p}[\id_{\operatorname{UC}_{N/4}}\operatorname{UG}^{A,\delta}_{\omega}[F] \mid \omega^{\xi}\in \mathcal{G}^{N}_A]-\phi_{\mathbb{Z}^d,p}[\id_{\operatorname{UC}_{N/4}}\operatorname{UG}^A_{\omega}[F]\mid \mathcal{G}_A]|+C\exp(-cN)\\
\leq &\sup_{\xi\in \mathtt{UC}_N} |\phi^{\xi}_{\Lambda_{N/2},p}[\id_{\operatorname{UC}_{N/4}}\operatorname{UG}^{A,\delta}_{\omega}[F] \mid \omega^{\xi}\in \mathcal{G}^{N}_A]
-\phi_{\mathbb{Z}^d,p}[\id_{\operatorname{UC}_{N/4}}\operatorname{UG}^A_{\omega}[F]\mid \omega|_{\Lambda_N\setminus \Lambda_{N/2}}=\xi|_{\Lambda_N\setminus\Lambda_{N/2}},\mathcal{G}_A]|
\\&+2C\exp(-cN)\\
=&\sup_{\xi\in \mathtt{UC}_N} |\phi^{\xi}_{\Lambda_{N/2},p}[\id_{\operatorname{UC}_{N/4}}\operatorname{UG}^{A,\delta}_{\omega}[F] \mid \omega^{\xi}\in \mathcal{G}^{N}_A]-\phi^{\xi}_{\Lambda_{N/2},p}[\id_{\operatorname{UC}_{N/4}}\operatorname{UG}^{A,\delta}_{\omega}[F]\mid \omega^{\xi}\in\mathcal{G}^N_A]|+2C\exp(-cN) \\
=&2C\exp(-cN).
\end{align*}
The rest of the proof follows from the Markov property, similarly to the proof of \Cref{thm:Unique_Loop}. Again, the exponentially bounded controlling regions property of Alexander's Theorem 3.3 is automatically satisfied.
\end{proof}
The proof of \Cref{thm:Unique_cur_with_sources} given \Cref{thm:Unique_loop_with_sources} is verbatim that of \Cref{thm:Free-mix} and we omit it here.

\section{Adaptation of the results to the \texorpdfstring{$q$}{q}-flow model}\label{sec:qflow}
In this section, we discuss how the arguments presented above apply to the $q$-flow model, which is the analogue of the loop O($1$) model for the random-cluster model with integer cluster weight ($q \neq 2$). The random-cluster model $\phi^{\xi}_{G,p,q}$ with cluster weight $q>0$ and boundary condition $\xi$,  on a finite graph $G=(V,E),$ as the percolation model with weights
$$
\phi^{\xi}_{G,p,q}[\omega]\propto q^{\kappa^{\xi}(\omega)}\left( \frac{p}{1-p}\right)^{|\omega|}.
$$
Note that $\phi^{\xi}_{G,p,1}=\mathbb{P}_{G,p}$ and $\phi^{\xi}_{G,p,2}=\phi^{\xi}_{G,p}$. Similar to before, there exist infinite volume limits $\phi^1_{\mathbb{Z}^d,p,q}=\lim_{G_n\nearrow \mathbb{Z}^d} \phi^1_{G_n,p,q}$ and $\phi^0_{\mathbb{Z}^d,p,q}=\lim_{G_n\nearrow \mathbb{Z}^d} \phi^0_{G_n,p,q},$ but it remains open whether these limits are equal in general\footnote{And, indeed, Pirogov-Sinai theory \cite{Pirogov-Sinai-FK} gives that they will differ at $p_c$ when $q$ is large enough.}. By \cite{Convexity_of_pressure}, it is known\footnote{And elementary arguments give equality when $\phi_{\Z^d,p,q}^1$ has no infinite cluster.} that the two differ for at most countably many values of $p$. For $q \geq 1$, the model satisfies \eqref{eq_CBC}, and so, when these two measures coincide, so do any other weak limits and hence, the definition of $p_c$ is unambiguous.

 However, in previous sections, we used Bodineau's result \cite{Bod05} that the slab and percolation thresholds agree for $q=2,$ which is not known for general $q$. Recall that a slab is a graph of the form $S_h=\mathbb{Z}^2\times \{0,\dots,h\}^{d-2}$.  Define the slab percolation threshold 
 $$
 p_{slab}=\inf\{p\in [0,1]\mid \sup_h\phi^0_{S_h,p,q}[0\cc \infty]>0\}.
 $$
 Conjecturally, $p_c=p_{slab}$ \cite{Pis96}. For convenience, we define $p_{slab}(\Z^2,q) = p_c(\Z^2,q)$.
\subsection{Generalities on uniform cycles}

In this section, we discuss the application of our arguments to the $q$-flow representation of the random-cluster model for integer $q$. For a simple graph $G=(V,E),$ let $\mathcal{O}(E)$ denote its set of oriented edges and denote by $\Omega^q(G)$ the set of $\mathbb{Z}/q\mathbb{Z}$-valued 1-forms on $E.$ These are maps $\eta:\mathcal{O}(E)\to \mathbb{Z}/q\mathbb{Z}$ such that $\eta_{(v,w)}=-\eta_{(w,v)}$ for all $(v,w)\in \mathcal{O}(E)$. There is a linear divergence map $\partial^G:\Omega^q(G)\to (\mathbb{Z}/q\mathbb{Z})^V$ given by $(\partial^G\eta)_v=\sum_{w\sim v} \eta_{(v,w)}.$ In the $q=2$ case, one gets $\eta_{(v,w)}=-\eta_{(v,w)}$ and one may therefore simply regard a $1$-form as a function on the edges.

It is a computation (see e.g., \cite[Lemma 4.1]{hansen2025general}) that for a percolation configuration $\omega$ on a finite graph $G$, $|\ker\partial^{\omega}|=q^{\abs{\omega}+\kappa(\omega)-|V|}$ and it follows that $\left(\frac{p}{1-p}\right)^{\abs{\omega}}q^{\kappa(\omega)}\propto \left(\frac{x}{1-x}\right)^{\abs{\omega}}|\ker \partial^{\omega}|$ for $x=\frac{p}{p+q(1-p)}$. This motivates the $q$-flow measure $\ell_{G,x}$ on $\ker \partial^{G}$ given by $\ell_{G,x}[\eta]\propto x^{|\hat{\eta}|},$ where, $\hat{\eta}_{vw}=\id_{\eta_{(v,w)}\neq 0},$ is called the trace of $\eta$. The $q$-flow model couples to $\omega\sim \phi^0_{G,p,q}$ as the uniformly random divergence-free form on $\omega$ - such forms are also called cycles. In this section, we will also need the version with sources.
\begin{definition}
For $q\in \mathbb{N}_{\geq 2},$ $x\in (0,1)$, a finite graph $G=(V,E)$ and $A\in(\mathbb{Z}/q\mathbb{Z})^V$ with $\sum_{v\in V} A_v=0,$  define the $q$-flow model with sources $A$ to be the measure on $\Omega^q(G)$ with 
$$
\ell_{G,x}^{q,A}[\eta]\propto \id_{\partial \eta=A} x^{|\hat{\eta}|}.
$$
Similarly, let $\hat{\ell}^{q,A}_{G,x}[\;\cdot\;]=\ell^{q,A}_{G,x}[\hat{\eta}\in\cdot\;].$ 
\end{definition}
Analogously to \eqref{SMP}, the $q$-flow measure has a Markov property,  
\begin{align}\label{eq:q_flow_Markov}
     \ell^{q,A}_{G,x}[\eta|_H\mid \eta|_{G\setminus H}]=\ell_{H,x}^{q,A-\partial (\eta|_{G\setminus H})}[\eta|_H].
\end{align}
We will briefly discuss the uniform measure on the cycle space, which we will denote $\operatorname{UD}_G=\operatorname{UD}^q_G=\ell^{q,0}_{G,1}$, as well as the uniform measure on its cosets $\{\partial \eta=A\}$ for $A\in (\mathbb{Z}/q\mathbb{Z})^V$ with $\sum_vA_v=0$, denoted $\operatorname{UD}^A_G=\operatorname{UD}^{q,A}_G$. 
In the proof of \Cref{thm:Unique_Loop}, we used that unique crossing events decouple the uniform even subgraph (cf. \cite[Proposition 3.6]{hansen2023uniform}). This is not specific to the $\mathbb{Z}/2\mathbb{Z}$ case and lifts straightforwardly to cycles with coefficients in $\mathbb{Z}/q\mathbb{Z}$ (or any other compact, Abelian group for that matter). First, we note that if $G$ is a graph and $H\subseteq G$ is a subgraph, then the restriction map $\pi_H:\ker(\partial^{G})\to (\mathbb{Z}/q\mathbb{Z})^{\mathcal{O}(E(H))}$ is a group homomorphism and since $\operatorname{UD}_G$ is the Haar measure on $\ker(\partial^G)$, it follows that the marginal of $\operatorname{UD}_G$ on $H$ is simply the uniform measure on the image of $\pi_H$. This furthermore allows one to deduce conditional independence from unique crossings. The proof is mutatis mutandis the same as in \cite[Proposition 3.6]{hansen2023uniform} and is sketched for completeness. For an edge set $E,$ we denote by $G(E)$ the induced graph, i.e. the graph with edge set $E$ and vertex set equal to the set of end-points of elements of $E$.

\begin{proposition} \label{Separating_Surface_Prop}
Let $q\in\mathbb{N}_{\geq 2},$ $G=(V,E)$ be a graph and let $E_1,E_3\subseteq E$ be disjoint. Denote $E_2:=E\setminus (E_1\cup E_3).$ Suppose that
\begin{enumerate}[label=(\roman*)]
    \item $E_1$ and $E_2$ are finite.
    \item The induced graph $G(E_2)$ is connected.
    \item Any path $(\gamma_j)_{1\leq j\leq n}$ in $G$ with $\gamma_1\in G(E_1)$ and $\gamma_n\in G(E_3)$ must have a $1<j<n$ with $\gamma_j\in G(E_2).$
\end{enumerate}
Then, $\pi_{G(E_1)}(\ker(\partial^G))=\pi_{G(E_1)}(\ker(\partial^{G(E_1\cup E_2)})).$ In particular, for any $A\in (\mathbb{Z}/q\mathbb{Z})^{V(G(E_3))}$ with $\sum_v A_v=0$,
$$\operatorname{UD}^A_G[\eta|_{\mathcal{O}(E_1)} \in \cdot\;]=\operatorname{UD}_{G}[\eta|_{\mathcal{O}(E_1)} \in \cdot\;]=\operatorname{UD}_{G(E_1\cup E_2)}[\eta|_{\mathcal{O}(E_1)}\in\cdot\;].$$ Furthermore, for $\eta\sim \operatorname{UD}^A_G,$ the random variables $\eta|_{\mathcal{O}(E_1)}$ and $\eta|_{\mathcal{O}(E_3)}$ are independent. 
\end{proposition}
\begin{proof}[Sketch of proof.] We start with the cycle case $A\equiv 0.$
Since $\ker(\partial^G)$ is spanned by paths which are either simple loops or bi-infinite, it suffices to argue that for any such path $\gamma,$ there is a $\tilde{\gamma}\subseteq G(E_1\cup E_2)$ such that $\gamma\cap E_1=\tilde{\gamma}\cap E_1.$ This is achieved by using $(iii)$ to cut $\gamma$ according to the times when it hits $G(E_2)$ and then using the connectedness of $G(E_2)$ to form loops.

When $A \neq 0,$ we may take a representative $\eta_0$ with $\partial\eta=A$ and get a bijection $\psi:\{\partial \eta=A\}\to\ker \partial,$ given by $\eta\mapsto \eta-\eta_0.$ Applying $(ii)$ and $(iii)$ again, one may pick $\eta_0$ to have support in $E_2\cup E_3.$ Then, if $\eta\sim \operatorname{UD}^A_G,$ then $\eta-\eta_0\sim \operatorname{UD}_{G},$ and deterministically, $\eta|_{E_1}=(\eta-\eta_0)|_{E_1}$.

The last conclusion follows from the equality of the marginals together with the Markov property \eqref{eq:q_flow_Markov}.
\end{proof}
 Similarly, one gets the natural generalisation of the determination of the Gibbs measures of the uniform even graph \cite[Theorem 3.14]{hansen2023uniform}. We say that a probability measure $\mu$ on $\ker \partial^\mathbb{G}$ is \textit{Gibbs} for the uniform cycle if for any finite $\Lambda\subseteq \mathbb{G},$ we have that $\mu[\;\cdot\mid \eta|_{\mathbb{G}\setminus \Lambda}]$ is $\mu$-a.s. uniform on $\{\eta' \in \Omega^q(\Lambda)\mid \partial \eta'=-\partial \eta|_{\mathbb{G}\setminus \Lambda} \}.$ Furthermore, we denote by $(\ker \partial^\mathbb{G})^{<\infty}$ the set of finitely supported cycles and $(\ker \partial^\mathbb{G})^{0}=\overline{(\ker \partial^\mathbb{G})^{<\infty}}$ with the closure taken in the topology of pointwise convergence. We refer to $(\ker \partial^\mathbb{G})^{0}$ as the set of \emph{free} cycles. Since $\ker(\partial^\mathbb{G})$ is spanned by finite loops and bi-infinite paths, we see that $(\ker \partial^\mathbb{G})^{0}=\ker \partial^\mathbb{G}$ if and only if $\mathbb{G}$ is one-ended by arguments as those in \cite{angel2021uniform,hansen2023uniform}. Again, the proof lifts from \cite{hansen2023uniform} mutatis mutandis and we include it for completeness.
 \begin{theorem} \label{Gibbs for cycles}
For any infinite, locally finite graph $\mathbb{G}$, the set of extremal Gibbs measures of the uniform cycle for $q\geq 2$ is in 1-1 correspondence with $\ker \partial^\mathbb{G}/(\ker \partial^\mathbb{G})^0$.
\end{theorem}
\begin{proof} By definition, a measure on $\ker \partial^\mathbb{G}$ is Gibbs for the uniform cycle if and only if it is invariant under the natural action of $(\ker \partial^\mathbb{G})^{<\infty}$. Since this action is continuous, it extends to an invariance under all of $(\ker \partial^\mathbb{G})^{0}.$ By uniqueness of the Haar probability measure, there is only one  $(\ker \partial^\mathbb{G})^{0}$-invariant measure on each co-set in $\ker \partial^\mathbb{G}/(\ker \partial^\mathbb{G})^0,$ which gives that all of these must be extremal Gibbs measures for the uniform cycle. On the other hand, any $(\ker \partial^\mathbb{G})^0$-invariant measure $\mu$ on $\ker \partial^\mathbb{G}$ for which there is $H\subseteq \ker \partial^\mathbb{G}/(\ker \partial^\mathbb{G})^0$ with $\mu[\cup_{h\in H} h]\in (0,1),$ we of course have
$$
\mu[\;\cdot\;]=\mu[\cup_{h\in H} h]\cdot \mu[\;\cdot \mid \cup_{h\in H} h]+(1-\mu[\cup_{h\in H}h])\cdot \mu[\;\cdot \mid \ker\partial^\mathbb{G}\setminus (\cup_{h\in H}h)],
$$
and since both measures on the right-hand side are $(\ker\partial^\mathbb{G})^0$-invariant and distinct, $\mu$ cannot be extremal. Since $\ker \partial^\mathbb{G}/(\ker(\partial ^{\mathbb{G}}))^0$ is a compact metric space, the only $\{0,1\}$-valued measures on $\ker \partial^\mathbb{G}/(\ker(\partial ^{\mathbb{G}}))^0$ are Dirac masses. Accordingly, for any $\ker(\partial^{\mathbb{G}})^0$-invariant measure $\mu$ on $\ker(\partial^{\mathbb{G}})$ which is not supported on a single co-set, there must exist $H\subseteq \ker \partial^\mathbb{G}/(\ker \partial^\mathbb{G})^0$ with $\mu[\cup_{h\in H} h]\in (0,1).$ This finishes the proof.
\end{proof}

Furthermore, we have the following natural marriage of \cite[Theorem 3.2]{aizenman2019emergent} and \cite[Proposition 4.3]{hansen2025general} (originally due to \cite{zhang2020loop}). In the following, for $A\in (\mathbb{Z}/q\mathbb{Z})^V$ with $\sum_v A_v=0$, we let $\mathcal{F}_{A}$ denote the set of graphs $\omega\in \{0,1\}^E$ such that there exists $\eta\in (\mathbb{Z}/q\mathbb{Z})^{\mathcal{O}(\omega)}$ with $\partial \eta=A.$ We note that this event is not simply equal to $\mathcal{F}_{\widehat{A}}$ as in the $q=2$ case - for instance, if $q\geq 3$ and $A=\id_{v}+\id_{v'}-\id_{w}-\id_{w'},$ then 
$$
\mathcal{F}_{A}=\left(\{v\cc w\}\cap \{v'\cc w'\} \right)\cup \left(\{v\cc w'\}\cap \{v'\cc w\}\right).
$$
Nonetheless, it remains true that the events $\mathcal{F}_{A}$ are increasing, and that $\mathcal{F}_A \cap \mathtt{UC}_n$ is measurable with respect to $\omega\vert_{\Lambda_n \setminus \Lambda_{n/2}}$ for any $A \in (\Z/q\Z)^{\partial_v \Lambda_{n/2}}$, making them amenable to the analysis from the rest of the paper. 
\begin{coupling} \label{thm:The_q_Coupling}
Let $G$ be a finite graph, $x\in (0,1)$, $A\in (\mathbb{Z}/q\mathbb{Z})^{V}$ such that $\sum_{v\in V} A_v=0,$  and $(\omega,\eta)$ be a random element of $\{0,1\}^E\times \{\partial \eta=A\}$ with distribution
$$
\mathscr{P}[(\omega,\eta)]\propto \mathbb{P}_{G,x}[\omega]\id_{\hat{\eta}\subseteq \omega}.
$$
Then, $\mathscr{P}[\omega\in\cdot \;]=\phi^0_{G,p,q}[\;\cdot\; \mid \mathcal{F}_{A}]$ satisfying $x=\frac{p}{p+q(1-p)},$ $\mathscr{P}[\eta\in \cdot \;]=\ell^{q,A}_{G,x}[\;\cdot \;]$, $\mathscr{P}[\eta\mid \omega]\propto \id_{\hat{\eta}\subseteq \omega}$ is uniform, and $\mathscr{P}[\omega\mid \eta]\propto (\delta_{\hat{\eta}}\cup \mathbb{P}_{x})[\omega]$.
\end{coupling}

\begin{proof}
By \cite[Theorem 1.1]{hansen2025general}, 
\begin{align*}
&\mathscr{P}[\omega=\omega_0]\propto \left(\frac{x}{1-x}\right)^{|\omega_0|}|\{\eta : \partial\eta=A,\hat{\eta}\subseteq \omega_0\}|=\id_{\mathcal{F}_{A}}(\omega_0)\left(\frac{x}{1-x}\right)^{|\omega_0|}|\ker(\partial^{\omega})|\propto \phi^0_{G,p,q}[\omega_0\mid \mathcal{F}_{A}], \\
&\mathscr{P}[\eta=\eta_0]\propto \mathbb{P}_{G,x}[\widehat{\eta_0}\;\mathrm{ open}]=x^{|\hat{\eta_0}|}.\\
&\mathscr{P}[\eta\mid \omega]\propto \id_{\hat{\eta}\subseteq \omega}\\
&\mathscr{P}[\omega\mid \eta] \propto \mathbb{P}_{G,x}[\omega\mid \hat{\eta}\;\mathrm{open}]=\delta_{\hat{\eta}}\cup \mathbb{P}_{G,x},
\end{align*}
which was what we wanted.
\end{proof}

The following lemma mirrors the input we used in the $q=2$ case in the proof of \Cref{thm:Unique_Loop}. The conclusion is weaker because we do not know uniqueness of infinite volume measures for the random-cluster model for general $q$.
\begin{lemma} \label{Source relaxation}
Let $q\in \mathbb{N}_{\geq 2}$, $d\geq 2$ and $p>p_{slab},$ and suppose that $\phi$ is a weak limit of the form
$$
\phi[\;\cdot\;]=\lim_{G_n\nearrow \mathbb{Z}^d}\phi^0_{G_n,p,q}[\;\cdot\mid \mathcal{F}_{A_n} ]
$$
for $A_n\in (\mathbb{Z}/q\mathbb{Z})^{\partial_vG_n}$ with $\sum_{v\in \partial_vG_n} A_{n,v}=0.$ Then, there exist $(\xi_n)_{n\in\mathbb{N}}\subseteq  \{0,1\}^{E(\mathbb{Z}^d)}$ (possibly random), such that
$$
\phi=\lim_{n\to\infty} \phi^{\xi_n}_{\Lambda_{n},p,q}.
$$
\end{lemma}
\begin{proof}
Without loss of generality, we may suppose that $\Lambda_{2n}\subseteq G_n.$ Note that, similarly to the $q=2$ case,
$$
\phi^0_{G_n,p,q}[\omega|_{\Lambda_n}\in\cdot  \mid \mathcal{F}_{A_n},\mathtt{UC}_{2n}]=\sum_{\xi} \phi^{\xi}_{\Lambda_n,p,q}[\;\cdot\;]\phi^0_{G_n,p,q}[\omega|_{\Lambda_{2n}\setminus \Lambda_n}=\xi \mid \mathcal{F}_{A_n},\mathtt{UC}_{2n}],
$$
where we have used the fact that $\xi\in \mathtt{UC}_{2n}$ to conclude that the value of $\omega|_{G_n\setminus \Lambda_{2n}}$ has no influence on the conditional distribution.

By \Cref{theorem:any_incerasing_event} (the proof of which works for any  $q \geq 1$ and any $p>p_{slab}$),  $d_{TV}(\phi^0_{G_n,p,q}[\omega|_{\Lambda_{2n}}\in\cdot  \mid \mathcal{F}_{A_n},\mathtt{UC}_{2n}],\phi^0_{G_n,p,q}[\omega|_{\Lambda_{2n}}\in\cdot  \mid \mathcal{F}_{A_n}])\leq \exp(-cn),$ where $d_{TV}$ denotes total variation distance. Thus,
$$
\lim_{n\to\infty}\phi^0_{G_n,p,q}[\omega|_{\Lambda_{n}}\in\cdot  \mid \mathcal{F}_{A_n}]=\lim_{n\to\infty} \phi^0_{G_n,p,q}[\omega|_{\Lambda_n}\in\cdot  \mid \mathcal{F}_{A_n},\mathtt{UC}_{2n}]=\lim_{n\to\infty}\phi^{\xi_n}_{\Lambda_n,p,q},
$$
where $\xi_n$ is random and chosen according to $\phi^0_{G_n,p,q}[\omega|_{\Lambda_{2n}\setminus \Lambda_n} \in \cdot \mid \mathcal{F}_{A_n},\mathtt{UC}_{2n}]$.
\end{proof}

Say that a measure $\ell$ on $\Omega^q(\mathbb{Z}^d)$ is \textit{Gibbs} for the $q$-flow model if for any finite 
$\Lambda\subseteq \mathbb{Z}^d,$ $$\ell[\;\cdot\mid \eta|_{\mathbb{Z}^d\setminus \Lambda}]= \ell^{q,-\partial (\eta|_{\mathbb{Z}^d\setminus \Lambda})}_{\Lambda,x}[\;\cdot\;] \hspace{2cm} \text{        $\ell$-a.s. }$$  From \Cref{theorem:any_incerasing_event}, we may conclude the following: Note that one-endedness of the infinite cluster\footnote{Which follows from deletion tolerance and uniqueness of the infinite cluster.} under $\phi_{\mathbb{Z}^d,p,q}$ implies uniqueness of its uniform cycle by \Cref{Gibbs for cycles}.

\begin{theorem}  \label{thm:Unique_loop_qversion}
Let $q\in \mathbb{N}_{\geq 2}$, $d\geq 2$ and $x>x_{slab}$. For any Gibbs measure $\ell^q_{\mathbb{Z}^d,x}$ of the $q$-flow model, there exists a weak limit $\phi_{\mathbb{Z}^d,p,q}$ of finite-volume random-cluster measures such that 
$$
\ell^q_{\mathbb{Z}^d,x}[\;\cdot\;]=\phi_{\mathbb{Z}^d,p,q}[\operatorname{UD}_{\omega}[\;\cdot\;]].
$$
In particular, whenever $x>x_{slab},$ there is a unique Gibbs measure for the $q$-flow model if and only if $\phi^1_{\mathbb{Z}^d,p,q}=\phi^0_{\mathbb{Z}^d,p,q}$.
\end{theorem}
\begin{proof}
By \Cref{theorem:any_incerasing_event}, \Cref{thm:The_q_Coupling}, \Cref{Separating_Surface_Prop} and the Gibbs property, we have that $\ell^q_{\mathbb{Z}^d,x}=\mu[\operatorname{UD}_{\omega}],$ where $\mu=\ell^q_{\mathbb{Z}^d,x}\cup \mathbb{P}_{\mathbb{Z}^d,x}.$ By \Cref{Source relaxation}, we have that $\mu$ is a weak limit of random-cluster measures. Finally, note that the relation $\ell \cup \Prb_{x} = \ell' \cup \Prb_{x}$ implies that $\ell = \ell'$ (cf. \cite[Claim B.5.]{hansen2025general}).
\end{proof}

\section{Consequences for Codimension 1 Lattice Gauge Theories} \label{sec:gauge}
In this section, we discuss consequences of our work for dual spin representations. These so-called lattice gauge models have interactions over high-dimensional cells. A $k$-cell in $\mathbb{Z}^d$ is an embedded copy of the hypercube $\{0,1\}^{k}$. $1$-cells are edges and $2$-cells are (two-dimensional) plaquettes. For a finite subset $\Lambda\subseteq \mathbb{Z}^d,$ denote by $\Lambda_{k}$ its set of $k$-cells and $\mathcal{O}(\Lambda_{k})$ its set of oriented $k$-cells\footnote{Recall that an orientation of $\mathbb{R}^k$ is a choice of orthonormal basis up to the action of $SO(k)$.}. Each $\mathfrak{c}\in \mathcal{O}(\Lambda_k)$ has a boundary $\partial \mathfrak{c}$ with orientations inherited from $\mathfrak{c}.$

Similarly to the $q$-flow models, a $\mathbb{Z}/q\mathbb{Z}$-valued $k$-chain on $\Lambda$ is an anti-symmetric function $\sigma:\mathcal{O}(\Lambda_{k})\to \mathbb{Z}/q\mathbb{Z}$ (i.e. $\sigma_{\mathfrak{c}}=-\sigma_{-\mathfrak{c}}$ for all $\mathfrak{c}\in \mathcal{O}(\Lambda_{k}),$ where $-\mathfrak{c}$ denotes the same $k$-cell with the opposite orientation). Let $C_{k}=C_{k}(\Lambda)$ denote this set of chains. For $\sigma\in C_{d-2},$ we get a gradient $d\sigma\in C_{d-1}$ given by $d\sigma_{\mathfrak{c}}=\sum_{\mathfrak{c'}\in \partial \mathfrak{c}} \sigma_{\mathfrak{c'}}.$ We say that $\mathfrak{c}\in \mathcal{O}(\Lambda_{d-1})$ is satisfied if $d\sigma_{\mathfrak{c}}=0$. We note that since $\sigma$ is a chain, $\mathfrak{c}$ is satisfied if and only if $-\mathfrak{c}$ is and so, we may talk about the corresponding unoriented cell being satisfied or not. We denote by $S(\sigma)$ the number of satisfied \emph{unoriented} cells of the chain $\sigma$.

For $q\in \mathbb{N}_{\geq 2}$ and $d\geq 3$ and $\beta>0,$ the Potts lattice gauge model (interacting over codimension 1 plaquettes) is the measure on $C_{d-2}$ defined by
$$
\mu_{\Lambda,\beta}[\sigma]\propto \exp(\beta |S(\sigma)|).
$$
The reason for introducing this measure here is that it is dual to the $q$-flow model \cite{WegDual}. This generalises the planar case, where the loop O($1$) model has the distribution of the cluster interfaces of the Ising model. That is, to each edge $e\in E(\mathbb{Z}^d)$ corresponds a dual $(d-1)$-cell $e^*$ of the dual lattice $\mathbb{Z}^d+(1/2,1/2,...,1/2)$. This can be made consistent with the orientations. Then, if $\eta\sim \ell^q_{\Lambda,x'}$ and $\eta^*_{(v,w)^*}=\eta_{(v,w)},$ then $\eta^*$ exactly has the distribution of the gradient $d\sigma$ for $x'=1-\exp(-\beta)$ (see e.g., \cite[Section D.1]{hansen2025general}). 
 It is known that the phase diagram of the random-cluster model corresponds to different topological phases for the gauge theories with $p<p_c$ corresponding to the so-called perimeter law regime and $p>p_{slab}$ corresponding to the so-called area law regime. We denote the dual parameter $\beta(p_{slab})=\beta_{\mathtt{area}}.$ We refer to \cite{AllTheAuthors, duncanPRCM2} for further reading on these connections.

As such, our work has the following corollaries: We say that an infinite volume measure $\mu^{\mathtt{grad}}_{\mathbb{Z}^d,\beta}$ on $C_{d-1}(\mathbb{Z}^d)$ is a \emph{gradient Gibbs measure} for $\mu_{\beta}$ if it is dual to a Gibbs measure of $\ell^{q}_{\mathbb{Z}^d,x}$, see \cite{duncanPRCM2,lebowitz1981surface} and \cite[Theorem 9.1]{aizenman2025geometric} and references therein, where area law and perimeter law of lattice gauge theories is also defined. In the Ising case ($q=2$), uniqueness of Gibbs measures for the loop O(1) model as proven in \Cref{thm:Unique_Loop} implies that:
\begin{theorem}
For $q=2,$ $d\geq 3$ and any $\beta\in(0,\beta_{\mathtt{area}})$, $\mu_{\Z^d,\beta}$ has a unique gradient Gibbs measure, which, furthermore, is ratio weak mixing.
\end{theorem}
We believe one could get a perturbative version of this theorem via classical techniques such as the cluster expansion, but we do not know of any other proof of this statement which works throughout the subcritical phase of a gauge theory.

For the other values of $q,$  \Cref{thm:Unique_loop_qversion} has the following consequence:
\begin{theorem}
    For $q\in \mathbb{N}_{\geq 3},$ $d\geq 3$ and any $\beta\in (0,\beta_{\mathtt{area}}),$ $\mu_{\Z^d,\beta}$ has a unique gradient Gibbs measure if and only if $\phi^1_{\mathbb{Z}^d,p,q}=\phi^0_{\mathbb{Z}^d,p,q}$ for the corresponding dual random-cluster measures.
\end{theorem}
Note that this statement is not a priori true. For the ordinary Ising model (that is, with codimension $d-1$), despite the fact that $\phi_{\Z^d}^1=\phi_{\Z^d}^0$, it is not true that there is a unique gradient Gibbs measure at all temperatures on $\mathbb{Z}^d$, as evidenced by the existence of Dobrushin states \cite{DobrushinStates}. 

\begin{appendix} 
\section{Miscellaneous A Priori Infinite Volume Identities} \label{sec:A_priori_inf}
In this appendix, we permit ourselves to jot down some basic, a priori facts about the zoo of graphical representations of the Ising model that are known to some in the community but hard to come by in writing. We claim no originality to \Cref{Prop:Folklore} below, although the presentation is likely idiosyncratic.

For any finite graph $G=(V,E)\subseteq \Z^d$ consider the wired graph $G^1=\mathbb{Z}^d/(\mathbb{Z}^d\setminus G)\cong G\cup\{\delta\}$. One checks that $\phi^1_{G,p}=\phi_{(G^{\circ})^1,p}$ under the natural identification of edges, where $G^{\circ}=G\setminus \partial_vG.$ We will permit ourselves the slight abuse of notation of writing $\phi^1_{G,p}=\phi_{G^1,p}$ instead in this section.\footnote{Since we are taking a limit $G\nearrow \mathbb{G}$ anyway, the choice of convention does not matter.}
As a step towards proving continuity of the Ising phase transition, Aizenman, Duminil-Copin and Sidoravicius  \cite{aizenman2015random} proved that the limit $\Prbcur_{\Z^d,\beta}^+=\lim_{G\nearrow \Z^d}\Prbcur_{G^1,\beta}$ exists. Similarly, one may introduce $\ell^+_{\Z^d,x}=\lim_{G\nearrow \Z^d} \ell_{G^1,x},$ which exists on similar grounds. Many natural questions around uniqueness in random current measures go towards a priori understanding whether $\Prbcur_{\Z^d,\beta}^+=\Prbcur_{\Z^d,\beta}^{\emptyset}$  or, equivalently, whether $\ell^+_{\Z^d,x}=\ell^0_{\mathbb{Z}^d,x}$ (see, e.g., \cite{hansen2026enddoublerandomcurrent} for how such a statement can be used). A posteriori, however, this is known:
\begin{proposition}[Folklore] \label{Prop:Folklore}
If $\mathbb{G}$ is transitive, amenable, and $\phi_\mathbb{G}^1=\phi_\mathbb{G}^0$, then $\ell_\mathbb{G}^+=\ell_\mathbb{G}^0$ and $\mathbf{P}_\mathbb{G}^+=\mathbf{P}_\mathbb{G}.$
\end{proposition}

Note that  $\ell_\mathbb{G}^+=\ell_\mathbb{G}^0\Longleftrightarrow \mathbf{P}_\mathbb{G}^+=\mathbf{P}_\mathbb{G}$ by considering odd parts and $\ell^+_{\mathbb{G}}=\ell^0_{\mathbb{G}}\Longrightarrow \phi^1_{\mathbb{G}}=\phi^0_{\mathbb{G}}$ follows from \Cref{Thm:The_Coupling} by continuity of taking unions. That is, we have highlighted the least trivial implication.  We note that we believe that neither the assumption of transitivity nor that of amenability should be necessary, but technical hurdles arise in cases where there are multiple infinite clusters. We believe that this paper is neither the time nor the place to address this issue.
The following two propositions enable a new proof of \Cref{Prop:Folklore}, which will be given at the end of this section.
For a general graph $\mathbb{G},$ we denote by $\Lambda_n(v)$ the ball of radius $n$ for the graph distance around $v$.
\begin{proposition}\label{prop:general_loop_O(1)_thermodynamic_limit}
For any infinite, locally finite, countable graph $\mathbb{G}$ and any $x\in [0,1],$ the limit 
$
\ell^0_{\mathbb{G},x}=\lim_{G\nearrow \mathbb{G}}\ell_{G,x}
$ exists and satisfies
$$
\ell^0_{\mathbb{G},x}[\;\cdot\,]=\phi_{\mathbb{G},p}^0[\operatorname{UEG}^0_\omega[\;\cdot\;]].
$$
Furthermore, $\phi^0_{\mathbb{G},p}=\ell^0_{\mathbb{G},x}\cup \mathbb{P}_{\mathbb{G},x}$.
\end{proposition}
\begin{proof}
For finite $G\subseteq \mathbb{G},$ we identify $\phi^0_{G,p}$ with a measure on $\{0,1\}^{E(\mathbb{G})}$ for which edges outside of $G$ are deterministically closed. Pick an increasing coupling of finite graphs $(\omega_G)_{G\subseteq \mathbb{G}}$ of $\phi^0_{G,p}$ in the sense that if $G\subseteq G',$ then $\omega_G\preceq \omega_{G'}$ almost surely. Denote by $\mathcal{P}$ the joint distribution. In particular, $\omega_{\mathbb{G}}=\lim_{G\nearrow \mathbb{G}}\omega_G$ exists pointwise almost surely and $\omega_{\mathbb{G}}\sim \phi^0_{\mathbb{G},p}.$ By \cite[Theorem 3.10]{hansen2023uniform}, we get that, almost surely, $\lim_{G\nearrow \mathbb{G}}\operatorname{UEG}_{\omega_G}=\operatorname{UEG}^0_{\omega_{\mathbb{G}}}$ in the sense of weak limits. 

By the Dominated Convergence Theorem\footnote{There is, of course, no version of the Dominated Convergence Theorem for general nets, but note that there are only countably many finite subgraphs of $\mathbb{G}$ and so the limit can be understood purely in terms of sequences.}, for any event $A$ depending only on finitely many edges,
\begin{align*}
  \lim_{G\nearrow \mathbb{G}}\ell_{G,x}[A] &=\lim_{G\nearrow \mathbb{G}}\phi^0_{G,p}[\operatorname{UEG}_{\omega}[A]] =\lim_{G\nearrow \mathbb{G}} \mathcal{P}[\operatorname{UEG}_{\omega_G}[A]] \\
  &=\mathcal{P}[\lim_{G\nearrow \mathbb{G}} \operatorname{UEG}_{\omega_G}[A]]=\mathcal{P}[\operatorname{UEG}^0_{\omega_{\mathbb{G}}}[A]]=\phi^0_{\mathbb{G},p}[\operatorname{UEG}^0_{\omega}[A]].  
\end{align*}
Since the finitely supported events form an intersection stable generating set for the Borel $\sigma$-algebra on $\{0,1\}^{E(\mathbb{G})},$ this establishes the desired. The second statement follows from the finite volume version (cf. \Cref{Thm:The_Coupling}), since taking unions is a continuous operation.
\end{proof}

\begin{proposition}\label{prop:general_loop_O(1)_thermodynamic_limit_wired}
For any infinite, locally finite, countable graph $\mathbb{G}$ and any $x\in [0,1],$ the limit 
$
\ell^+_{\mathbb{G},x}=\lim_{G\nearrow \mathbb{G}}\ell_{G^1,x}
$ exists and satisfies
$$
\ell^+_{\mathbb{G},x}[\;\cdot\,]=\phi_{\mathbb{G},p}^1[\operatorname{UEG}_\omega[\;\cdot\;]].
$$

Furthermore, $\phi^1_{\mathbb{G},p}=\ell^+_{\mathbb{G},x}\cup \mathbb{P}_{\mathbb{G},x}.$
\end{proposition}
\begin{proof}
This will be similar to the above, except that we do not have access to \cite[Theorem 3.10]{hansen2023uniform}. However, uniqueness of the Haar measure saves the day. For finite $G\subseteq \mathbb{G},$ we identify $\phi^1_{G,p}$ with a measure on $\{0,1\}^{E(G^1)}\cong \{0,1\}^{E(\mathbb{G})}$ for which edges with no end-point\footnote{The reason for not throwing away the  edges outside $G$ is that keeping them admits direct comparisons of the respective spaces of even subgraphs, which previously proved an obstruction to getting an explicit a priori description of \Cref{Thm:The_Coupling} in infinite volume in the wired case (say, when $\phi_{\mathbb{G}}^0\neq \phi_{\mathbb{G}}^1$) - contrast with \cite[Theorem 3.11]{hansen2023uniform}. The reason we choose Bernoulli percolation with parameter $p$ is that this is the measure obtained upon identifying all vertices in the complement of $G$.}  in $G$ are sampled according to a Bernoulli percolation with parameter $p$ - with components counted in $G^1$. Pick an increasing coupling of finite graphs $(\omega_G)_{G\subseteq \mathbb{G}}$ of $\phi^1_{G,p}$ in the sense that if $G\subseteq G',$ then $\omega_G\succeq \omega_{G'}$ almost surely. Denote by $\mathcal{P}$ the joint distribution. In particular, $\omega_{\mathbb{G}}=\lim_{G\nearrow \mathbb{G}}\omega_G$ exists pointwise almost surely and $\omega_{\mathbb{G}}\sim \phi^1_{\mathbb{G},p}.$ We want to argue that $\lim_{G\nearrow \mathbb{G}} \operatorname{UEG}_{\omega_G}=\operatorname{UEG}_{\omega_\mathbb{G}}$ almost surely. By compactness of the space of probability measures on $\{0,1\}^{E(\mathbb{G})},$ it suffices to uniquely characterise any accumulation point. As $\operatorname{supp}(\operatorname{UEG}_{\omega_G})=\Omega_{\emptyset}(\omega_G),$ any accumulation point must have support on $\cap_{G\subseteq \mathbb{G}, G \text{ finite}}\Omega_{\emptyset}(\omega_G)=\Omega_{\emptyset}(\omega_{\mathbb{G}}).$ Furthermore, for any finite graph $G\subseteq \mathbb{G},$ $\Omega_{\emptyset}(\omega_{\mathbb{G}})\subseteq \Omega_{\emptyset}(\omega_{G})$ and hence, $\operatorname{UEG}_{\omega_G}$ is invariant under the action of $\Omega_{\emptyset}(\omega_{\mathbb{G}}).$ The conclusion is that any accumulation point $\operatorname{UEG}_{\omega_G}$ is supported on $\Omega_{\emptyset}(\omega_{\mathbb{G}})$ and invariant under the action of $\Omega_{\emptyset}(\omega_{\mathbb{G}})$. By uniqueness of the Haar measure, the only such measure is $\operatorname{UEG}_{\omega_{\mathbb{G}}}$ and hence, $\lim_{G\nearrow \mathbb{G}}\operatorname{UEG}_{\omega_G}=\operatorname{UEG}_{\omega_{\mathbb{G}}}$ almost surely.

In particular, by the Dominated Convergence Theorem, for any event $A$ depending only on finitely many edges,
\begin{align*}
\lim_{G\nearrow \mathbb{G}}\ell_{G^1,x}[A]
&=\lim_{G\nearrow \mathbb{G}}\phi^1_{G,p}[\operatorname{UEG}_{\omega}[A]] =\lim_{G\nearrow \mathbb{G}} \mathcal{P}[\operatorname{UEG}_{\omega_G}[A]]\\
&=\mathcal{P}[\lim_{G\nearrow \mathbb{G}} \operatorname{UEG}_{\omega_G}[A]]=\mathcal{P}[\operatorname{UEG}_{\omega_{\mathbb{G}}}[A]]=\phi^1_{\mathbb{G},p}[\operatorname{UEG}_{\omega}[A]].
\end{align*}
Since the finitely supported events form an intersection stable generating set for the Borel $\sigma$-algebra on $\{0,1\}^{E(\mathbb{G})},$ this establishes the desired. The second statement follows from the finite volume version (cf. \Cref{Thm:The_Coupling}), since taking unions is a continuous operation.
\end{proof}
And so, we are in position:

\begin{proof}[Proof of \Cref{Prop:Folklore}]
By amenability and transitivity, $\phi_{\mathbb{G},p}$ is almost surely one-ended if it percolates for some $p<1$. Accordingly, $\operatorname{UEG}_{\omega}=\operatorname{UEG}_{\omega}^0$ almost surely by \cite[Lemma 3.9]{angel2021uniform}. Thus, \Cref{prop:general_loop_O(1)_thermodynamic_limit} and \Cref{prop:general_loop_O(1)_thermodynamic_limit_wired}  show that $\ell_\mathbb{G}^+=\ell_\mathbb{G}^0.$ This immediately implies $\Prbcur_\mathbb{G}^+=\Prbcur_\mathbb{G}^{\emptyset}$ since they have the same odd parts. 
\end{proof}

An alternative, and plausibly more standard, proof would be to plug $\phi_\mathbb{G}^1=\phi_\mathbb{G}^0$ into the Edwards-Sokal coupling to deduce that the gradient Gibbs measures for the Ising model with free respectively $+$ boundary conditions agree. Then, one may retrace the proof of the existence of $\Prbcur_\mathbb{G}^+$ and $\Prbcur_\mathbb{G}^{\emptyset}$ (cf. the proof of Theorem 2.3 in \cite{aizenman2015random}) and realise that the only Ising correlations used are functions of the gradient. We also note that this strategy should work out without the assumptions of transitivity and amenability.

A complementary perspective is the following analogue of a well-known result (cf. the proof of \cite[Theorem 5.33a)]{Gri06}) for the random-cluster model:
\begin{proposition} \label{Unique Non-perc}
For any infinite, locally finite, countable graph $\mathbb{G}$ and any $x\in [0,1],$ any Gibbs measure $\ell_{\mathbb{G},x}$ of the loop O($1$) model which does not percolate is necessarily $\ell^0_{\mathbb{G},x}.$
\end{proposition}
\begin{remark}
For $\mathbb{G}=\mathbb{H}$ the hexagonal lattice, $\phi_{\mathbb{H},p}[\operatorname{UEG}_{\omega}[\;\cdot\;]]$ never percolates for any $p\in[0,1]$ and so, the statement has nontrivial content even in the regime $x>x_c.$
\end{remark}
\begin{proof}
Let $v\in V(\mathbb{G}),$ $k\in \mathbb{N}$ and $A\in \mathcal{A}_{\Lambda_k}$ be given. For $N>k$, denote by $\mathbf{C}_N$ the union of the clusters of $\partial_v \Lambda_N(v)$ in $\Lambda_N(v)$. If $\mathbf{C}_N$ does not intersect $\Lambda_k(v),$ by the Gibbs property, \eqref{eq:definition_of_Gibbs_measure_for_loop_O(1)}, then as random variables,
$$
\ell_{\mathbb{G},x}[A\mid \mathbf{C}_N,\eta|_{\Lambda_N^c(v)} ]=\ell_{\mathbb{G},x}[A\mid \mathbf{C}_N ]=\ell_{\Lambda_N(v)\setminus \mathbf{C}_N,x}[A].
$$
Since $\sigma((\mathbf{C}_N,\eta|_{\Lambda_N(v)^c}))$ is decreasing in $N,$ we can apply the Backwards Martingale Convergence Theorem and get
$$
\ell_{\mathbb{G},x}[A]= \lim_{N\to\infty}\mathbb{E}[\ell_{\mathbb{G},x}[A\mid \mathbf{C}_N,\eta|_{\Lambda_N(v)^c} ]]=\mathbb{E}[\lim_{N\to\infty}\ell_{\mathbb{G},x}[A\mid \mathbf{C}_N,\eta|_{\Lambda_N(v)^c} ]],
$$
where the second equality is due to the Dominated Convergence Theorem. Since $\ell_{\mathbb{G},x}$ does not percolate, we get that $\Lambda_N(v)\setminus \mathbf{C}_N\nearrow \mathbb{G}$ almost surely and so, by \Cref{prop:general_loop_O(1)_thermodynamic_limit},
$$
\mathbb{E}[\lim_{N\to\infty}\ell_{\mathbb{G},x}[A\mid \mathbf{C}_N,\eta|_{\Lambda_N(v)^c} ]]=  \mathbb{E}[\lim_{N\to\infty}\ell_{\Lambda_N(v)\setminus \mathbf{C}_N,x}[A]]=\mathbb{E}[\ell^0_{\mathbb{G},x}[A]]=\ell^0_{\mathbb{G},x}[A].
$$
\end{proof}

\section{Pisztora's giant meets the boundary}

\label{sec:Pisztora_giant meets_the_boundary}
The geometry of the random-cluster model beyond the so-called slab percolation threshold has been well-understood since the seminal work of Pisztora \cite{Pis96}: In a finite box, the infinite cluster manifests as a single giant cluster, and all other clusters are small.

The goal of this appendix is to prove \Cref{prop:giant_touches_face_density}, showing that for any fixed subset of the boundary, Pisztora's giant cluster will touch some proportion, even under adverse boundary conditions. In \Cref{sec:Catch}, we will use \Cref{prop:giant_touches_face_density} to prove the following proposition (a special case of \Cref{theorem:any_incerasing_event}) which is also sufficent to deduce the main theorems.

\begin{proposition}\label{proposition:unique_crossing_with_fa_old_version}
For any $d\geq 2$ and $p>p_c,$ there exists $C>0$ such that for any $N\in \mathbb{N}$ and any $A\subseteq \partial_v \Lambda_N$ with $|A|$ even,
$$
\phi^0_{\Lambda_N,p}[\mathtt{UC}_N\mid \mathcal{F}_A]\geq 1-\exp(-CN).
$$
\end{proposition}

\subsection{Exploration coupling}
In the following, we repeat a standard increasing coupling of FK-percolation measures and prove that it has certain properties that will suffice to study measures of the form $\phi[\;\cdot \mid \mathcal{F}_A]$.

We shall need a slightly stronger comparison than usual stochastic domination, which we call \textbf{strong} stochastic domination. We say that $\mu$ strongly stochastically dominates $\nu,$ written $\nu\overset{\mathbf{s}}{\preceq} \mu,$ if, for every $E'\subseteq E$ and every $\xi\preceq \xi'\in \{0,1\}^{E'}$ such that $\mu[\omega|_{E'}=\xi'],\nu[\omega|_{E'}=\xi]>0,$ we have $\nu[\;\cdot\mid \omega|_{E'}=\xi]\preceq \mu[\;\cdot \mid \omega|_{E'}=\xi'].$ It is worth noting that a lot of natural instances of stochastic domination in statistical mechanics are explicitly examples of strong stochastic domination. 
In fact, it is often used (see e.g., \cite[Proposition 2.6]{ScalingRelations}) to get explicit versions of Strassen's Theorem:

\vspace{0.3cm}
\begin{mdframed}[style=MyFrame]
\begin{center}
    \textbf{Coupling by Exploration:}
\end{center}

Let $E$ be a finite set of edges and fix a total ordering $e_1,e_2,...,e_{|E|}$. Let $(U_e)_{e\in E}$ be i.i.d. and uniformly distributed on $[0,1]$.

For a measure $\nu$ on $\{0,1\}^E,$ we define the $\{0,1\}^E$-valued random variable $\omega^{\nu}$ recursively by
\begin{align*}
\omega^{\nu}_{e_1}&=\id[U_{e_1}\leq \nu[e_1\text{ open}]]\\
\omega^{\nu}_{e_{j+1}}&=\id[U_{e_{j+1}}\leq \nu[e_{j+1}\text{ open}\mid \omega|_{\{e_1,...,e_j\}}=(\omega^{\nu}_{e_1},...,\omega^{\nu}_{e_j})]].
\end{align*}
Then, $\omega^{\nu}\sim \nu$. Furthermore, if $\nu$ and $\nu'$ are two measures on $\{0,1\}^E$ such that $\nu'\overset{\mathbf{s}}{\preceq} \nu$, then $\omega^{\nu'}\preceq \omega^{\nu}$ almost surely.
\end{mdframed}

\vspace{0.3cm} Checking the distribution of $\omega^{\nu}$ is a straightforward application of the Law of Total Probability. Checking that the coupling is increasing between strongly dominating measures is simply the fact that $\{e \text{ open}\}$ is an increasing event for every $e\in E$. One may note that in many applications, the ordering is actually taken to be random, with the choice of $e_{j+1}$ being a measurable function of $(U_{e_i})_{1\leq i\leq j}$. We omit this additional (but harmless) complication as it will not play a role in the current paper.
\begin{lemma} \label{lemma:FA_dom}
For any finite graph $G=(V,E),$ $p\in(0,1),$ $A\subseteq V$ with $|A|$ even, and any boundary condition $\xi,$ we have
$$\phi^\xi_{G,p} \overset{\mathbf{s}}{\preceq} \phi^\xi_{G,p}[\;\cdot \mid \mathcal{F}_A].$$
\end{lemma}
\begin{proof}
Since $\mathcal{F}_A$ is increasing, we get ordinary stochastic domination by \eqref{FKG}. Now, fix $E'\subseteq E$ and $\tilde \xi\preceq \xi'\in \{0,1\}^{E'}$ for which $\phi_{G,p}[\omega|_{E'}=\xi'\mid \mathcal{F}_A]>0$.  
Using \eqref{SMP}, if we let $\xi'_0=(\xi')^{\xi}$ and let $\tilde{\xi}_0=(\tilde{\xi})^{\xi},$
$$
\phi^\xi_{G,p}[\; \cdot \mid \mathcal{F}_A,\omega|_{E'}=\xi']=\phi_{G\setminus E',p}^{\xi'_0}[\;\cdot \mid \omega^{\xi'_0}\in \mathcal{F}_{A}]\succeq \phi^{\xi'_0}_{G\setminus E',p}\succeq \phi^{\tilde \xi_0}_{G\setminus E',p}=\phi^\xi_{G,p}[\;\cdot \mid \omega|_{E'}=\tilde \xi],
$$
where the first inequality is \eqref{FKG}, and the second is \eqref{eq_CBC}. 
\end{proof}

Using \eqref{eq_CBC} and \eqref{SMP}, one can similarly show that $\phi_{B',p}^0 \ssd \phi_{B,p}^0$ for $B'\subseteq B,$ where $\phi_{B',p}^0$ is identified with a measure on $B$ such that every edge outside $B'$ is deterministically closed. 

\begin{lemma}\label{lemma:coupling_many_FK_measures}
    Under any exploration coupling $\mathscr{P}$, the marginals $\omega_B \sim \phi_B^0$ are coupled such that if $B' \subseteq B,$ then $\omega_{B'} \preceq \omega_{B}$ and if $ E(\tilde B) \cap E(B) = \emptyset,
    $ then $\omega_B \independent \omega_{\tilde B}$. 
\end{lemma}
\begin{proof}
 Since $\phi_{B',p}^0 \ssd \phi_{B,p}^0$ whenever $B' \subseteq B,$ the first item follows. Since $\omega_B$ is $(U_e)_{e\in E(B)}$-measurable, the second item follows.
\end{proof}

\subsection{Pisztora's giants touch a density of points on the boundary}
In the following, we argue that the local giants will touch even free boundaries robustly, in the sense that for any designated set $A\subseteq \partial_v \Lambda_n,$ the giant will touch at least $\gamma|A|$ of its vertices with positive probability.

Our strategy for doing this is to couple a tree of several giants in increasing fashion and using the local geometry of each giant to deduce sufficient regularity of the biggest one. In the following, we let $\mathscr{T}_n^{\mathtt{d}}$ denote the rooted tree with $n$ generations and each node having $\mathtt{d}$ children. For $v\in V(\mathscr{T}_n^\mathtt{d})$, we denote by $\mathtt{gen}(v)$ the generation of $v$, i.e. the graph distance from the root $o$. We denote by $V_j=V_j(\mathscr{T}_n^\mathtt{d})$ the set of vertices in the $j$'th generation. The following technical lemma is the key to our proof of \Cref{prop:giant_touches_face_density} below.

\begin{lemma} \label{lemma:Big_Bad_Tree_lemma} Let $n\in \mathbb{N}$, $\mathtt{d}> 2$,$C>0, \lambda >1.$ There exists $C'>0$ with the following property: Suppose that $\alpha>0$ and $\nu$ is a percolation measure  on $\mathscr{T}_n^\mathtt{d}$ satisfying \begin{enumerate}
    \item[$i)$] For any $v\in V,$ 
    $
    \nu[v\;\mathrm{ closed}]\leq \alpha \exp\left(-C \lambda^{
    n-\mathtt{gen}(v)}\right).
    $
    \item[$ii)$] For any $W,W'\subseteq V(\mathscr{T}_n^\mathtt{d})$ with no common descendants, the $\sigma$-algebras $\mathcal{A}_{W}$ and $\mathcal{A}_{W'}$ are independent.
\end{enumerate}
Let $\mathcal{C}_o$ denote the cluster of the root.  Then, for every $A\subseteq V_n,$ 
$$
\nu[|\mathcal{C}_o \cap A|\leq (1-\alpha C'-|A|^{-1/3}) |A| ]\leq \alpha C'|A|^{-1/3}.
$$
\end{lemma}
\begin{proof}
This will essentially be a second moment computation.

For $A\subseteq V_j$, let $\Gamma^A_j$ be the set of simple paths from $o$ to $A$ and abbreviate $\Gamma^{V_j}_j = \Gamma_j$.  Say that each such path is \emph{open} if all the vertices it traverses are open, and \emph{not open} if at least one vertex along the path is closed. Let $\mathfrak{C}^{A}_j$ be the set of non-open paths in $\Gamma^{A}_j$. We write $\mathfrak{C}_j$ for the set of non-open paths from the root $o$ to the $j$'th generation 
and $\mathfrak{C}$ for the random set of all simple non-open paths (whether or not they contain the root). 

For the first moment, for $v\in V(\mathscr{T}_n^{\mathtt{d}}),$ let $\gamma_v$ denote the unique path from $v$ to the root. Then, $i)$ yields that
\begin{align}\label{eq:expected_connection}
\nu[|\mathfrak{C}_n^A|]=\sum_{v\in A} \nu[\gamma_v\mathrm{\; not\; open}]\leq  |A|\alpha \sum_{k=0}^n\exp\left(-C \lambda^{
    n-k}\right)\leq \alpha C' |A|. 
\end{align}

Turn now to the second moment. 
By $ii)$ and a union bound, for $\gamma,\gamma'\in \Gamma_j$,
$$
\nu[\gamma,\gamma'\in \mathfrak{C}]\leq \nu[\gamma\setminus \gamma'\in \mathfrak{C}]\nu[\gamma'\setminus \gamma \in \mathfrak{C}]+\nu[\gamma\cap \gamma'\in \mathfrak{C}].
$$
On generic grounds, 
$$
\nu[\gamma\in \mathfrak{C}_j]\nu[\gamma'\in \mathfrak{C}_j]\geq \nu[\gamma\setminus \gamma'\in \mathfrak{C}]\nu[\gamma'\setminus \gamma\in\mathfrak{C}],
$$
so, all in all,
$$
\operatorname{Cov}_{\nu}[\id_{\gamma\in \mathfrak{C}},\id_{\gamma'\in \mathfrak{C}} ]\leq \nu[\gamma\cap \gamma'\in \mathfrak{C}].
$$
Letting $D(\gamma)$ denote the descendants of the last vertex on $\gamma$, this yields
$$
\operatorname{Var}_{\nu}[|\mathfrak{C}^A_n|]\leq  \sum_{\gamma,\gamma'\in \Gamma^A_n}\nu[\gamma\cap \gamma'\in \mathfrak{C}]=\sum_{k=0}^n \sum_{\gamma\in \Gamma^A_k}\nu[\gamma\in \mathfrak{C}_k] {|A\cap D(\gamma)| \choose 2}\leq \sum_{k=0}^n \alpha e^{-C \lambda^{
    n-k}} \mathtt{d}^{(n-k)} |A|\leq \alpha C' |A|,
$$
where, in the middle equality, we summed over the value of $\gamma\cap \gamma'\in \Gamma_k$.  The second inequality used $i)$ together with the fact that $\gamma\in \Gamma_k$ has at most $\mathtt{d}^{n-k}$ descendants in $V_n$ and that each $v\in A$ is the descendant of exactly one $\gamma\in \Gamma_k^A$.

Now, by Chebyshev's Inequality,
$$
\nu[|\mathfrak{C}^A_n-\nu[|\mathfrak{C}^A_n|]|\geq |A|^{2/3}]\leq  \alpha C'|A|^{-1/3}.
$$
This yields the conclusion, since 
 $$
 \nu[|\mathcal{C}_o\cap A|\leq (1-\alpha C'-|A|^{-1/3})|A|]=\nu[|\mathfrak{C}^{A}_n|\geq (\alpha C'+|A|^{-\frac{1}{3}}) |A|]\leq \nu[||\mathfrak{C}^{A}_n|-\nu[|\mathfrak{C}^{A}_n|]|\geq |A|^{\frac{2}{3}}]\leq \alpha C'|A|^{-\frac{1}{3}}.
 $$
\end{proof}

\begin{figure}[ht]
\centering
\begin{tikzpicture}[
    scale=1,
    every node/.style={font=\small},
    shared/.style={draw=violet!70!black, fill=violet!30, thick, circle, inner sep=2.5pt},
    gammanode/.style={draw=red!70!black, fill=red!30, circle, inner sep=2pt},
    gammapnode/.style={draw=blue!70!black, fill=blue!30, circle, inner sep=2pt},
    ghostnode/.style={draw=black!25, fill=black!10, circle, inner sep=1.5pt},
    sharedline/.style={violet!70!black, line width=1.8pt},
    gammaline/.style={red!60!black, line width=1.4pt},
    gammaprline/.style={blue!60!black, line width=1.4pt},
    ghostline/.style={black!15, line width=0.5pt},
    genlabel/.style={font=\footnotesize\itshape, text=black!50},
]

\node[genlabel, anchor=east] at (-0.6, 0) {gen $0$};
\node[genlabel, anchor=east] at (-0.6, -1.8) {gen $k$};
\node[genlabel, anchor=east] at (-1.9, -6.2) {gen $n$};

\draw[black!20, densely dashed, thin] (-0.3,0) -- (0.2,0);
\draw[black!20, densely dashed, thin] (-0.3,-1.8) -- (0.2,-1.8);
\draw[black!20, densely dashed, thin] (-0.3,-6.2) -- (0.2,-6.2);

\draw[ghostline] (0,0) -- (-2.2,-1.8);
\draw[ghostline] (0,0) -- (2.2,-1.8);
\node[ghostnode] at (-2.2,-1.8) {};
\node[ghostnode] at (2.2,-1.8) {};
\draw[ghostline] (-2.2,-1.8) -- (-2.55,-2.6);
\draw[ghostline] (-2.2,-1.8) -- (-1.85,-2.6);
\draw[ghostline] (2.2,-1.8) -- (1.85,-2.6);
\draw[ghostline] (2.2,-1.8) -- (2.55,-2.6);

\draw[ghostline] (0,-1.8) -- (0,-2.8);
\node[ghostnode] at (0,-2.8) {};
\draw[ghostline] (0,-2.8) -- (-0.25,-3.4);
\draw[ghostline] (0,-2.8) -- (0.25,-3.4);

\draw[black!30, dashed, rounded corners=8pt, thin] 
    (-2.5,-2.2) rectangle (2.5,-6.6);
\node[font=\footnotesize, text=black!45] at (2.85,-2.5) {$D(\gamma\cap\gamma')$};

\draw[sharedline] (0,0) -- (0,-1.8);

\draw[gammaline] (0,-1.8) -- (-1.2,-3.0);
\draw[gammaline] (-1.2,-3.0) -- (-1.5,-4.3);
\draw[gammaline] (-1.5,-4.3) -- (-1.65,-6.2);
\draw[ghostline] (-1.2,-3.0) -- (-0.8,-3.7);
\draw[ghostline] (-1.5,-4.3) -- (-1.9,-5.0);
\draw[ghostline] (-1.5,-4.3) -- (-1.1,-5.0);
\node[ghostnode] at (-0.8,-3.7) {};
\node[ghostnode] at (-1.9,-5.0) {};
\node[ghostnode] at (-1.1,-5.0) {};

\draw[gammaprline] (0,-1.8) -- (1.2,-3.0);
\draw[gammaprline] (1.2,-3.0) -- (1.5,-4.3);
\draw[gammaprline] (1.5,-4.3) -- (1.65,-6.2);
\draw[ghostline] (1.2,-3.0) -- (0.8,-3.7);
\draw[ghostline] (1.5,-4.3) -- (1.1,-5.0);
\draw[ghostline] (1.5,-4.3) -- (1.9,-5.0);
\node[ghostnode] at (0.8,-3.7) {};
\node[ghostnode] at (1.1,-5.0) {};
\node[ghostnode] at (1.9,-5.0) {};

\node[shared, inner sep=3pt] (root) at (0,0) {};
\node[anchor=west, font=\small] at (0.3, 0.05) {$o$};
\node[shared, inner sep=3pt] (vk) at (0,-1.8) {};
\node[anchor=west, font=\small] at (0.3,-1.75) {$v_k$};
\node[gammanode] at (-1.2,-3.0) {};
\node[gammanode] at (-1.5,-4.3) {};
\node[gammanode, inner sep=2.8pt] at (-1.65,-6.2) {};
\node[font=\footnotesize, anchor=north] at (-1.65,-6.55) {$v\in A$};
\node[gammapnode] at (1.2,-3.0) {};
\node[gammapnode] at (1.5,-4.3) {};
\node[gammapnode, inner sep=2.8pt] at (1.65,-6.2) {};
\node[font=\footnotesize, anchor=north] at (1.65,-6.55) {$v'\in A$};

\draw[decorate, decoration={brace, amplitude=5pt, mirror}, violet!70!black]
    (-0.35,0) -- (-0.35,-1.8)
    node[midway, left=7pt, font=\footnotesize\itshape, text=violet!70!black] {$\gamma\!\cap\!\gamma'$};

\node[font=\footnotesize\itshape, text=red!60!black, anchor=east] at (-1.75,-3.5) {$\gamma \setminus \gamma'$};
\node[font=\footnotesize\itshape, text=blue!60!black, anchor=west] at (1.75,-3.5) {$\gamma' \setminus \gamma$};
\end{tikzpicture}
\caption{Two root-to-leaf paths $\gamma,\gamma'\in \Gamma_n^A$ sharing a common prefix down to $v_k$ at generation $k$. By condition $(ii)$, the non-open events on $\gamma \setminus \gamma'$ and $\gamma'\setminus \gamma$ are independent, so $\operatorname{Cov}[\mathbf{1}_{\gamma\in\mathfrak{C}},\mathbf{1}_{\gamma'\in\mathfrak{C}}]\leq \nu[\gamma\cap\gamma'\in\mathfrak{C}]$.} 
\label{fig:tree_paths}
\end{figure}

\begin{remark} An additional application of $ii)$ and $i)$ makes it possible to bootstrap the decay in probability to something stretch-exponential in $|A|$ rather than sublinear. However, for our purposes, it suffices to have some rate of decay.
\end{remark}

The following corollary follows from the machinery of Pisztora (recall \Cref{def:Pisztora_event}).

\begin{corollary} \label{Touch_Cluster}
For $d\geq 3$ and $p>p_c,$ there exists $c>0$ such that for any $n\in \mathbb{N}$ and any $v\in \partial_v \Lambda_n,$ 
$$
\phi^0_{\Lambda_n,p}[|\mathcal{C}_v| \geq \frac{2^d-1}{2^{d+1}}\theta|\Lambda_n|]\geq c.
$$
\end{corollary}
\begin{proof}
By insertion tolerance, it suffices to consider $n$ large.
Assume without loss of generality that the face of $\Lambda_n$ containing $v$ is contained in the hyperplane $\{w_1=n\}=\{(w_1,...,w_d)\in \mathbb{Z}^d\mid w_1=n\}$. Fix $h$ as in \Cref{Long-range-slab}. Let $\Lambda\subseteq \Lambda_n$ denote a translate of $\Lambda_{n-h}$ which has a face contained in $\{w_1=-n\}$. Note that whenever $\omega\in \mathtt{Pis}_{\Lambda},$ then there exists a cluster in $\omega|_{\Lambda_n\setminus\{w_1\geq n-h\}}$ which is larger than $\frac{2^{d}-1}{2^d}\theta |\Lambda_{n-h}|,$
which, in turn, is larger than $\frac{2^d-1}{2^{d+1}}\theta|\Lambda_n|$ for $n$ large. Let $F$ denote the event that $\omega|_{\Lambda_n\setminus \{w_1\geq n-h\}}$ contains a cluster with size at least $\frac{2^d-1}{2^{d+1}}\theta|\Lambda_n|$ and which touches the hyperplane $\{w_1=n-h\}.$

By applying \Cref{Long-range-slab} $d-2$ times along with \eqref{FKG} and \eqref{SMP}, there exists a constant $c'>0$ such that
$$
\phi^0_{\Lambda_n,p} \left[|\mathcal{C}_v |\geq\frac{2^d-1}{2^{d+1}}\theta|\Lambda_n| \mid \omega|_{\Lambda_n\setminus \{w_1\geq n-h\}}\right]\geq c'\id_{\omega\in F}.
$$
Integrating yields
$$
\phi^0_{\Lambda_n,p}\left[|\mathcal{C}_v|\geq \frac{2^d-1}{2^{d+1}}\theta|\Lambda_n|\right]\geq c'\phi^0_{\Lambda_n,p}[F]\geq c' \phi^0_{\Lambda_n,p}[\mathtt{Pis}_{\Lambda}]\geq c
$$
for some adjusted constant.
\end{proof}

The following proposition shows that Pisztora's giants touch a density of any set with constant probability. 
For convenience, in the sequel, we extend the definition of $\Giant_K \subseteq V_K$ so that it is the giant in the event that $\mathtt{Pis}_K$ occurs and equal to $\emptyset$ otherwise. 
\begin{proposition}[The giant touches a density of any subset of the boundary]\label{prop:giant_touches_face_density}
Fix $d \geq 3, p > p_c(d)$. There exist $\varepsilon, \gamma >0$ such that for each $n \in \N$ and any  $A \subseteq \partial_v \Lambda_n$, 
$$
\phi_{\Lambda_n,p}^0[\abs{ \mathtt{Giant}_{\Lambda_n} \cap A} \geq \gamma \abs{A}] \geq \varepsilon.
$$    
\end{proposition}
\begin{proof}
Let us give an outline of the proof. 
First, we prove the proposition in case $\abs{A}\leq M$ for some constant $M$ that will be fixed at the end of the proof. Then, we do a dyadic decomposition and show how the giants glue. This is used for the case $\abs{A}\geq M$ afterwards. We finish by combining those bounds. 

Suppose first $\abs{A}\leq M.$ Then, by \Cref{Touch_Cluster}, a union bound, and \eqref{FKG}, 
\begin{align*}
\phi^0_{\Lambda_n,p}[A\subseteq \mathtt{Giant}_{\Lambda_n}]&=\phi^0_{\Lambda_n,p}\left[\cap_{v\in A} \left(|\mathcal{C}_v|\geq \frac{2^d-1}{2^{d+1}}\theta|\Lambda_n|\right)\cap \mathtt{Pis}_{\Lambda_n}\right]\\
&\geq \phi^0_{\Lambda_n,p}\left[\cap_{v\in A} \left(|\mathcal{C}_v|\geq \frac{2^d-1}{2^{d+1}}\theta|\Lambda_n|\right)\right] -\exp(-cn^{d-1})\\
&\geq c^{|A|}-\exp(-cn^{d-1})\geq (c')^M
\end{align*}
for some adjusted constant.

Next, we will consider the following dyadic subdivision scheme: For fixed $L$, to be determined, we will split $\Lambda_{n-L}$ into its $2^d$ hyperoctants and remove from any hyperoctant the edges of the faces where one of the coordinates is minimal - this ensures that any two resulting boxes are edge-disjoint. We proceed in this way inductively with each of the resulting boxes, stopping the iteration after $\log(n)-\log(L)$ steps, which yields that the last scale is of side-length roughly $L$. By possibly enlarging $L,$ we may assume that the last scale is strictly smaller than $L$.

Among the last generation of boxes, those of side length roughly $L,$ at least $\abs{A}L^{-d}$ of them must be within distance $L$ of a vertex $a \in A$.  From this last generation, fix a (target) set $\mathfrak{B}$ of size $\abs{A}L^{-d+1}\leq \abs{\mathfrak{B}} \leq \abs{A}$, such that for every $a\in A,$ there is a box $K\in \mathfrak{B}$ within distance $L$ of $a$.

 Let $\mathscr{P}$ be an exploration coupling under an arbitrary ordering of the edges of $\Lambda_n.$ For $G\subseteq \Lambda_n,$ we denote by $\omega_G$ the corresponding sample of $\phi^0_{G,p}$. Furthermore, define 
 $$
\mathcal{K}= \{K \in \mathfrak{B}\mid \Giant_K(\omega_K) \overset{\omega_{\Lambda_{n-L}}}{\longleftrightarrow } \Giant_{\Lambda_{n-L}}(\omega_{\Lambda_{n-L}})\},
 $$
 which can be thought of as the set of "targets" in  $\mathfrak{B}$, that were "hit" by the giant $\Giant_{\Lambda_{n-L}}(\omega_{\Lambda_{n-L}})$, see \Cref{fig:skydeskive}. 
  
 A main technical step  is to use \Cref{lemma:Big_Bad_Tree_lemma} to argue the existence of $c,C'>0$ such that, uniformly in our choice of $L,$
 \begin{align} \label{eq:tree_lemma_output}
 \mathscr{P}\left[\abs{\mathcal{K}}\leq (1- e^{-cL^\frac{d}{2}} C' - \abs{\mathfrak{B}}^{-\frac{1}{3}})\abs{\mathfrak{B}}\right] \leq  e^{-cL^\frac{d}{2}} C'\abs{\mathfrak{B}}^{-\frac{1}{3}}. 
 \end{align}

We will then fix $L$ so large that $e^{-cL^{d/2}}C'<1/2.$

We consider each dyadic box $B$ as a vertex in a graph $\mathscr{T}$, with an edge between two vertices of the graph if one of the corresponding boxes is included in the other. Note that $\mathscr{T}$ is a $2^d$-regular rooted tree with root corresponding to $\Lambda_n$.

    On $\mathscr{T}$,  define a site percolation $\tau$ which declares a vertex $B$ open if $\omega_B \in \Pis_B$. Denote by $\nu$ its distribution.

    Whenever $L \geq \max\{L_0,N_0\}$ (as in \Cref{thm:pisztora}), it is ensured that $\Giant_{B_{j}}(\omega_{B_j}) \overset{\omega_{B_{j-1}}}\longleftrightarrow \Giant_{B_{j-1}}(\omega_{B_{j-1}})$ whenever $B_{j}$ is a hyperoctant of $B_{j-1}$ and both are open in $\tau$. Indeed, suppose $\omega_{B_{j-1}} \in \Pis_{B_{j-1}}$ and $\omega_{B_{j}} \in \Pis_{B_{j}}$. 
    By \Cref{lemma:coupling_many_FK_measures},
    $\omega_{B_{j}} \preceq \omega_{B_{j-1}}$  almost surely. 
     Therefore, for any vertex $v \in \Giant_{B_{j}},$ its enlarged cluster $\mathcal{C}_v(\omega_{B_{j-1}})$ has density at least $\theta \cdot 2^{-d}$. 
    However, in $\omega_{B_{j-1}},$ the local giant $\Giant_{B_{j-1}}(\omega_{B_{j-1}})$ is the only cluster with diameter at least $L$ and density larger than $\theta \cdot 2^{-d}$. Therefore, $v \in \Giant_{B_{j-1}}$. See \Cref{fig:Gluing_Giants}.
Iterating this argument shows that whenever $B_k \in \mathcal{T}_{2^d}$ is an open vertex and there is a path of open vertices from $B_k$ to the root $o$ then $\Giant_{B_k}(\omega_{B_k}) \overset{\omega_{\Lambda_{n-L}}}{\longleftrightarrow} \mathtt{Giant}_{\Lambda_{n-L}}(\omega_{\Lambda_{n-L}})$. 

\begin{figure} 
    \centering
    \includegraphics[width=\linewidth]{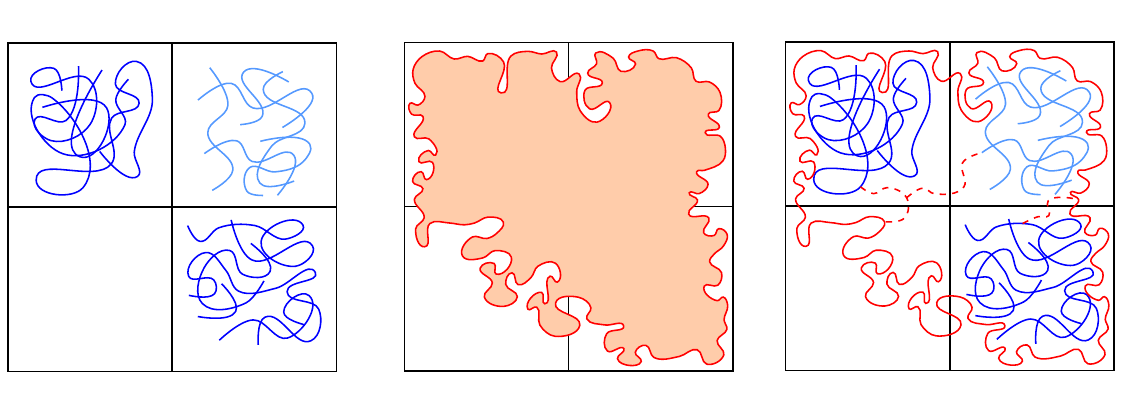}
    \caption{One step of the dyadic subdivision process. On the left: Giant clusters existing under free boundary conditions on each smaller box in various shades of blue. In the middle: The giant component under free boundary conditions on the larger box in red and orange. On the right: The blue giants are so large that, under the increasing coupling, they must be connected to the red giant via red edges. Due to planar limitations of the graphical presentation, we have elected not to indicate that all giants are connected to the boundary.}
    \label{fig:Gluing_Giants}
\end{figure}

Let us verify that the assumptions of \Cref{lemma:Big_Bad_Tree_lemma} are satisfied.  
Item $ii)$ of \Cref{lemma:Big_Bad_Tree_lemma} follows
 from \Cref{lemma:coupling_many_FK_measures} since, by construction, boxes from the same generation are edge-disjoint. 
For item $i),$ first note that by Pisztora's result (\Cref{thm:pisztora}) the probability that a vertex in $\mathtt{gen}(v)$ is closed is at most $e^{-c (2^{-\mathtt{gen}(v)}n)^{d-1}}$.
Thus, denoting the probability measure on the constructed tree by $\nu$,  since there are $ \log(n)-\log(L)$ generations in total,
$$
\nu[v \text{ closed}] \leq e^{-c2^{-\mathtt{gen}(v)(d-1)}n^{d-1}} = e^{-c2^{(\log(n)-\mathtt{gen}(v))(d-1)}} \leq \alpha \cdot e^{-c2^{(\log(n)-\mathtt{gen}(v))(\frac{d}{2}-1)}},$$
with $\alpha = e^{-cL^\frac{d}{2}}$. Since every box is split into $\mathtt{d}=2^d$ boxes, item $i)$ of \Cref{lemma:Big_Bad_Tree_lemma} is satisfied with $\lambda = 2^{\frac{\log_2(\mathtt{d})}{2}-1} >1$,  since $d\geq 3$. 
Inputting all boxes $K \in \mathfrak{B}$ into  \Cref{lemma:Big_Bad_Tree_lemma}, the considerations about connecting the giants at dyadic scales above yield \eqref{eq:tree_lemma_output}.\\

Now, by possibly thinning $\mathfrak{B}$ (but maintaining the lower bound $|\mathfrak{B}|\geq \frac{1}{\sqrt{d}}L^{-d}|A|$), for each $K\in \mathfrak{B}$ one may choose a translate $B^K$ of $\Lambda_L$ such that a face of $K$ is contained within a face of $B^K$, $A\subseteq \cup_{K\in \mathfrak{B}}B^K$ and $B^K$ and $B^{K'}$ only intersect at a face for $K\neq K'$. This is also illustrated on \Cref{fig:skydeskive}. We let $B^{K,\circ}$ denote the edges of $B^K$ which do not lie on any face. Again, this ensures disjointness of the respective edge sets. Let $A_K$ denote the event that every edge of $B^{K,\circ}$ is open in $\omega_{B^{K,\circ}}$.

\begin{figure}
    \centering
\begin{tikzpicture}

  \fill[gray!12] (-3,-3) rectangle (3,3);
  \fill[white]   (-2.5,-2.5) rectangle (2.5,2.5);
  \draw[thick]   (-3,-3) rectangle (3,3);

  \pgfmathsetmacro{\L}{0.42}
  \pgfmathsetmacro{\Lh}{\L/2}

  \begin{scope}
    \clip (-2.5,-2.5) rectangle (2.5,2.5);
    \node at (0,0) {\includegraphics[width=5cm]{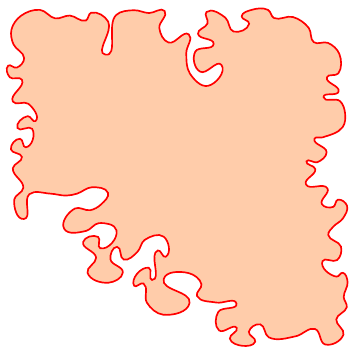}};
  \end{scope}
  \draw[thick] (-2.5,-2.5) rectangle (2.5,2.5);


  \fill[pattern=north east lines] ({0.9-\Lh},  {2.5-\L}) rectangle ({0.9+\Lh},  2.5);
  \draw                           ({0.9-\Lh},  {2.5-\L}) rectangle ({0.9+\Lh},  2.5);
  \fill[pattern=north east lines] ({-1.4-\Lh}, {2.5-\L}) rectangle ({-1.4+\Lh}, 2.5);
  \draw                           ({-1.4-\Lh}, {2.5-\L}) rectangle ({-1.4+\Lh}, 2.5);
  \fill[pattern=dots, pattern color=black] (-1.65, 2.5) rectangle (-1.15, 3.0);
  \draw (-1.65, 2.5) rectangle (-1.15, 3.0);
  \draw ({2.3-\Lh}, {2.5-\L}) rectangle ({2.3+\Lh}, 2.5);

  \fill[pattern=north east lines] ({2.5-\L}, {1.1-\Lh}) rectangle (2.5, {1.1+\Lh});
  \draw                           ({2.5-\L}, {1.1-\Lh}) rectangle (2.5, {1.1+\Lh});
  \fill[pattern=north east lines] ({2.5-\L}, {-1.3-\Lh}) rectangle (2.5, {-1.3+\Lh});
  \draw                           ({2.5-\L}, {-1.3-\Lh}) rectangle (2.5, {-1.3+\Lh});

  \draw ({-0.6-\Lh}, -2.5) rectangle ({-0.6+\Lh}, {-2.5+\L});
  \fill[pattern=north east lines] ({1.5-\Lh}, -2.5) rectangle ({1.5+\Lh}, {-2.5+\L});
  \draw                           ({1.5-\Lh}, -2.5) rectangle ({1.5+\Lh}, {-2.5+\L});
  \fill[pattern=dots, pattern color=black] (1.25, -3.0) rectangle (1.75, -2.5);
  \draw (1.25, -3.0) rectangle (1.75, -2.5);

  \draw (-2.5, {-1.2-\Lh}) rectangle ({-2.5+\L}, {-1.2+\Lh});
  \fill[pattern=north east lines] (-2.5, {1.6-\Lh}) rectangle ({-2.5+\L}, {1.6+\Lh});
  \draw                           (-2.5, {1.6-\Lh}) rectangle ({-2.5+\L}, {1.6+\Lh});

  \fill ( 0.9,  3)   circle (2pt);  \node[above, font=\small] at ( 0.9,  3)   {};
  \fill ( 1,  3)   circle (2pt);  \node[above, font=\small] at (1,  3)   {};
  \fill (-1.4,  3)   circle (2pt);  \node[above, font=\small] at (-1.4,  3)   {};
  \fill ( 3,    1.1) circle (2pt);  \node[right, font=\small] at ( 3,    1.1) {};
  \fill (-0.6, -3)   circle (2pt);  \node[below, font=\small] at (-0.6, -3)   {};
  \fill (-3,   -1.2) circle (2pt);  \node[left,  font=\small] at (-3,   -1.2) {};
  \fill ( 2.3,  3)   circle (2pt);  \node[above, font=\small] at ( 2.3,  3)   {};
  \fill ( 3,   -1.3) circle (2pt);  \node[right, font=\small] at ( 3,   -1.3) {};
    \fill ( 3,   -1.15) circle (2pt);  \node[right, font=\small] at ( 3,   -1.15) {};
  \fill ( 1.5, -3)   circle (2pt);  \node[below, font=\small] at ( 1.5, -3)   {};
  \fill (-3,    1.6) circle (2pt);  \node[left,  font=\small] at (-3,    1.6) {};

  \draw[<->, thin, gray!60] (3.4, -3) -- node[right, font=\scriptsize, gray]
        {$L$} (3.4, -2.5);
 
  \node at (1.9, -2.3) {\tiny{$K$}};
  
    \node at (2, -2.7) {\tiny{$B^K$}};
  \pgfmathsetmacro{\ly}{-3.78}      
  \pgfmathsetmacro{\yc}{\ly+\Lh}    
  \pgfmathsetmacro{\legendgap}{0.08}

  \pgfmathsetmacro{\kx}{-2.20}
  \pgfmathsetmacro{\bx}{-0.35}
  \pgfmathsetmacro{\akx}{2.55}

  \node[font=\small, anchor=east] at ({\kx-\legendgap}, {\yc}) {$\mathcal{K}$:};
  \fill[pattern=north east lines] (\kx, \ly) rectangle ({\kx+\L}, {\ly+\L});
  \draw                           (\kx, \ly) rectangle ({\kx+\L}, {\ly+\L});

  \node[font=\small, anchor=east] at ({\bx-\legendgap}, {\yc}) {$\mathfrak{B}$:};
  \draw (\bx, \ly) rectangle ({\bx+\L}, {\ly+\L});

  \node[font=\small, anchor=east] at ({\akx-\legendgap+1.5}, {\yc})
        {$B^K$ satisfying $A_K$:};
  \fill[pattern=dots, pattern color=black] (\akx+1.5, \ly) rectangle ({\akx+1.5+\L}, {\ly+\L});
  \draw                                    (\akx+1.5, \ly) rectangle ({\akx+1.5+\L}, {\ly+\L});
\end{tikzpicture}
    \caption{The box $\Lambda_n$ with marked points from $A$ on the boundary. In the box $\Lambda_{n-L}$, the boxes in $\mathfrak{B}$ are at the smallest dyadic scale, such that every point in $A$ is at most distance $L$ away from a box in $\mathfrak{B}$. $\mathcal{K}$ is the set of boxes in $\mathfrak{B}$ that connect to the giant. For each box $K\in \mathcal{K}$, a box of size $L$, denoted $B^{K}$, that contains a face of $K$ and touches $\partial \Lambda_n$ is chosen.  $A_K$ is the event that all internal edges in $B^{K}$ are open (shown with dots).} 
    \label{fig:skydeskive}
\end{figure}

 As the $\omega_{B^{K,\circ}}$ are mutually independent and independent of $\mathcal{K},$ by the Chernoff-Hoeffding inequality, there exists a $\delta >0$, which depends on $\min_K\mathscr{P}[A_K]\geq c^{L^d}$ by finite energy, such that 
$$
\mathscr{P}[ \sum_{K\in \mathcal{K}} \id_{A_K}\leq \delta |\mathcal{K}| \mid \mathcal{K}]\leq \exp(- \delta |\mathcal{K}|).
$$
Thus, for given $\gamma'>0,$
$$
\mathscr{P}\left[ \sum_{K\in \mathcal{K}} \id_{A_K}\leq \gamma' |A| \right]
\leq  \exp(-\gamma' |A|)\mathscr{P}\left[|\mathcal{K}|\geq \frac{\gamma'}{\delta}|A|\right]+\mathscr{P}\left[|\mathcal{K}|< \frac{\gamma'}{\delta}|A|\right].
$$
Applying \eqref{eq:tree_lemma_output} and the fact that $|\mathfrak{B}|\geq \frac{1}{\sqrt{d}}L^{-d}|A|$, we get for $\gamma = \delta \left(1- e^{-cL^\frac{d}{2}} C' - (\frac{L^{-d}M}{\sqrt{d}})^{-\frac{1}{3}} \right) \frac{1}{\sqrt{d}}L^{-d}$
$$
\mathscr{P}\left[ \sum_{K\in \mathcal{K}}\id_{A_K}\leq \gamma|A| \right]\leq \exp(-\gamma |A|)+e^{-cL^{d/2}}C'' L^{d/3} |A|^{-1/3}.
$$
Note that on the event $\{\mathtt{Giant}_K(\omega_K)\overset{\omega_{\Lambda_{n-L}}}{\longleftrightarrow}\mathtt{Giant}_{\Lambda_{n-L}}(\omega_{\Lambda_{n-L}})\}\cap A_K\cap \mathtt{Pis}_{\Lambda_n}(\omega_{\Lambda_n}),$ every vertex in $A\cap B^K$ is connected to the giant in $\omega_{\Lambda_n}$. Thus, on $\mathtt{Pis}_{\Lambda_n},$ $\sum_{K\in \mathcal{K}}\id_{A_K}\leq L^{-d}|\mathtt{Giant}_{\Lambda_{n}}\cap A|.$

All in all, 
by a union bound
\begin{align*}
\mathscr{P}[|A\cap \mathtt{Giant}_{\Lambda_{n}}|\leq \gamma |A|]&\leq
\mathscr{P}[ \sum_{K\in \mathcal{K}} \id_{A_K}\leq \gamma|A| ,\mathtt{Pis}_{\Lambda_n}] +\exp(-cn^{d-1})\\
&\leq \exp(-\gamma |A|)+e^{-cL^{d/2}}C' L^{d/3} |A|^{-1/3}+2\exp(-cn^{d-1}). 
\end{align*}
Since $|A|\leq c'' n^{d-1}$, we get the desired by choosing $M$  large enough that
$$
\exp(-\gamma M)+e^{-cL^{d/2}}C' L^{d/3} M^{-1/3}+2\exp(-c'''M)< 1
$$
and combining the upper bounds achieved.
\end{proof}

\section{Unique FK-crossings conditioned on \texorpdfstring{$\mathcal{F}_A$}{F\_A}: Proof of \texorpdfstring{\Cref{proposition:unique_crossing_with_fa_old_version}}{Proposition \ref*{proposition:unique_crossing_with_fa_old_version}}}
\label{sec:Catch}

In this section, we prove \Cref{proposition:unique_crossing_with_fa_old_version}. The strategy is, again, to use an exploration coupling to gradually relax the condition $\mathcal{F}_A$. Basically, we slice up the box $\Lambda_n$ into annuli and use \Cref{prop:giant_touches_face_density} to glue vertices in $A$ to giants under free boundary conditions in each annulus. This is the content of \Cref{sec:Catch}, which gives a useful bound until the number of clusters intersecting $A$ starts looking sublinear, at which point repeated application of \Cref{prop:giant_touches_face_density} is no longer strong enough to yield the right bound. However, at this point, realising the modified $\mathcal{F}_{A'}$ event given what was already explored has an a priori cost which is at most exponential with a rate we can control. Thus, we can get away with a union bound. 

\subsection{Catching via exploration}

The following two lemmata will be used to control a single exploration step in \Cref{The_Catching_Prop} below.  Define $N_k=N-k\log N$  and the annulus $\mathtt{Ann}^k=\Lambda_{N_{k-1}}\setminus \Lambda_{N_k}$
of width $\log(N)$. Similarly to the dyadic subdivision in \Cref{prop:giant_touches_face_density}, we tile $\ak$ by a maximal collection of edge disjoint boxes of side length $\log(N)$. Refer to one fixed such collection as $\mathfrak{B}_k$.  

We suggestively let $\mathtt{Pis}_{\ak}$ denote the event that $\omega|_{B}\in \mathtt{Pis}_{B}$ for all $B\in \mathfrak{B}_k$ and that $\mathtt{Giant}_B(\omega|_B)\overset{\omega_{\ak}}{\longleftrightarrow} \mathtt{Giant}_{B'}(\omega|_{B'})$ for $B,B'\in \mathfrak{B}_k.$ The following is a standard application of Pisztora and its proof is more or less the same as the beginning of the proof of \cite[Lemma 4.8]{hansen2023uniform} apart from the fact that the annulus under consideration is significantly thinner.
\begin{lemma}[All Giants Glue] \label{lemma:all_giants_glue}
Fix $d\geq 3$ and $p>p_c$. For any $N \in \N$, there exists $c>0$ such that for each $k < \lfloor \frac{N}{4\log N}\rfloor$,
$
\phi^0_{\ak,p}[\mathtt{Pis}_{\ak}]\geq 1- e^{-c\log(N)^{d-1}}.
$
\end{lemma}
\begin{proof}
Let $\mathscr{B}_k\supseteq \mathfrak{B}_k$ be an enlargement of $\mathfrak{B}_k$ with the following properties: 
\begin{enumerate}
    \item Each $B\in \mathscr{B}_k$ is either a translate of $\Lambda_{\log N}$ or a translate of $\Lambda_{\log(N)/2}$ and $B\subseteq \ak$. We will refer to the former as large boxes and to the latter as small boxes.
    \item For large boxes $B,B'\in \mathscr{B}_k,$ there exists an alternating sequence $B=B_0, b_{0,1},B_1,...,b_{n-1,n},B_n=B'$ with each $B_i$ large and each $b_{i,i+1}$ small and $b_{i,i+1}\subseteq B_i\cap B_{i+1}.$
    \item $|\mathscr{B}_k|\leq C N^{d-1}.$
\end{enumerate} 
For instance, $\mathscr{B}_k$ can be chosen to be the set of all large and small boxes centered at a point on $\partial_v \Lambda_{N_k+(\log N)/2}$. Note that for a path as in $(2),$ if $(\omega|  _{B_i},\omega|_{b_{i,i+1}},\omega|_{B_{i+1}})\in \mathtt{Pis}_{B_i}\times \mathtt{Pis}_{b_{i,i+1}}\times \mathtt{Pis}_{B_{i+1}}$, then since $|\omega|_{b_{i,i+1}}|\geq 2^{-d}\theta |B_i|=2^{-d}\theta|B_{i+1}|$, $\mathtt{Giant}_{b_{i,i+1}}(\omega|_{b_{i,i+1}})\subseteq 
\mathtt{Giant}_{B_i}(\omega|_{B_i})\cap \mathtt{Giant}_{B_{i+1}}(\omega|_{B_{i+1}})$. In particular,
\begin{align*}
   \phi^0_{\ak,p}[\mathtt{Pis}_{\ak}] &\geq \phi^0_{\ak,p}\left[\bigcap_{B\in \mathscr{B}_k} \mathtt{Pis}_{B}(\omega_B)\right] \\ 
   &\geq 1-|\mathscr{B}_k|\exp(-c(\log N/2)^{d-1})\geq 1-CN^{d-1}\exp(-c(\log N/2)^{d-1}),
\end{align*}
by a union bound. Now, since $\exp(-c(\log N/2)^{d-1})$ is superpolynomial, 
$$
CN^{d-1}\exp(-c(\log N/2)^{d-1})\leq \exp(-c'(\log N/2)^{d-1})
$$
for an adjusted constant $c'$ and $N$ large. Possibly adjusting the constant again to take care of smaller values of $N$ yields the lemma.
\end{proof}

Let $\partial_{out} \ak$ denote the outer boundary of the annulus $\ak$. Furthermore, on the event that $\omega|_{\ak}$ has a unique large cluster with density $2^{-d}\theta$ in each $B\in \mathfrak{B}_k,$ we denote by $\mathtt{Giant}_{\ak}$ the corresponding cluster. If no such cluster exists, we define it to be empty. 
\begin{lemma}\label{lemma:fraction_connects_to_supergiant}
  Fix $\delta>0$, and let $A \subseteq \partial_{out} \ak$ satisfy that $\abs{A} \geq \delta N$. There exists $\varepsilon, \alpha >0$ so that for all $N$ large enough,
    $$
    \phi_{\ak,p}^0[\abs{A \cap \Giant_{\ak}} \geq \varepsilon \abs{A}] \geq 1-N^{-\alpha}. 
    $$
\end{lemma}
\begin{proof}   
It is possible to choose the tiling $\mathfrak{B}_k$ to have $|\mathfrak{B}_k|\leq C (\frac{N}{\log N})^{d-1} $. Since $A \subseteq \partial_{out} \ak,$ and $\abs{A} \geq \delta N$ it suffices to prove that $\Giant_{\ak}$ touches a fraction of the covered vertices. 
By  \Cref{prop:giant_touches_face_density} for some $\varepsilon_0, \gamma >0$  (and $n=\log(N)$) for any $B$ in the tiling, 
$$\phi_{B,p}^{0}[\abs{ \mathtt{Giant}_B \cap A} \geq \gamma \abs{A \cap B}] > \varepsilon_0. $$

 Define $G_B =\abs{ \mathtt{Giant}_B  \cap A}$ and $G = \sum_{B \in \mathfrak{B}_k}G_B$. Then $\sum_{B \in \mathfrak{B}_k}\mathscr{\phi}^0_{B,p}[G_B] \geq \gamma \varepsilon_0 \abs{A} =2 \varepsilon\abs{A}$ and, deterministically, $0 \leq G_B \leq \abs{B} \leq C\log(N)^{d-1}.$ 
As the existence of the giant is not increasing, we will make a slight detour to get concentration bounds for $G$ under $\phi^0_{\ak,p}$ out of $\mathscr{\phi}^0_{B,p}[G_B]$. Introduce the random variable corresponding to the points in dense clusters, $H_B = \abs{\{a \in A\cap B \mid \abs{\mathcal{C}_a(\omega_B)}\geq 2^{-d} \theta \abs{B} \}}  $ and note that $G_B \id_{\Pis_B} =H_B \id_{\Pis_B}$. Therefore,
$$
\phi_{B,p}^0[H_B] \geq \phi_{B,p}^0[H_B \mid \Pis_B] - e^{-c \log(N)^{d-1}} =  \phi_{B,p}^0[G_B \mid \Pis_B] - e^{-c \log(N)^{d-1}}.
$$
Set $H= \sum_{B}H_B.$ Since $H$ is increasing, Hoeffding's inequality shows that
\begin{align*}
\phi_{\ak,p}^0[H < \varepsilon\abs{A}] &\leq (\otimes_{B} \phi_{B,p}^0)[H < \varepsilon\abs{A}] \leq   \exp \left( - \frac{2 (\frac{1}{2}(\otimes_{B} \phi_{B}^0)[H])^2}{\sum_{B} \abs{A\cap B}^2} \right) \\
&\leq  \exp\left( - \frac{2 \varepsilon^2\abs{A}^2}{\log(N)^{d-1}\sum_{B}\abs{A\cap B}} \right) 
= \exp\left( - \frac{2 \varepsilon^2}{\log(N)^{d-1}} \abs{A}\right).  
\end{align*}
All Pisztora's events are likely enough to transfer to $G$. By \Cref{lemma:all_giants_glue}, 
\begin{align*}
\phi_{\ak}^0[G<\varepsilon\abs{A}] 
&\leq \phi_{\ak}^0[G<\varepsilon\abs{A} \mid \cap_B\Pis_B]  +e^{-c \log(N)^{d-1}}\\ 
&\leq \phi_{\ak}^0[H<\varepsilon\abs{A} \mid \cap_B\Pis_B] +e^{-c \log(N)^{d-1}}
\leq \exp\left( - \frac{2 \varepsilon^2}{\log(N)^{d-1}} \abs{A}\right) + 2e^{-c \log(N)^{d-1}}.
\end{align*}

Thus, as long as $\abs{A} \geq \delta N$, a proportion of the vertices glue to their local giant with high probability.

By \Cref{lemma:all_giants_glue}, all local giants (in boxes of sizes $\log(N)$) exist simultaneously and glue together to $\Giant_{\ak}$  with high probability.  The lemma follows by a union bound. 
\end{proof}

 Furthermore, we will need the following elementary bound.
\begin{lemma} \label{Lemma:Bernoulli_LD}
Let $\alpha>0$, $k=N/\log(N)$ and $p=N^{-\alpha}$. Then, for any $C>0,$ 
$$
\operatorname{Bin}_{k,p}[Y\geq k - C \log(N)]\leq \exp(-N(\alpha+o(1))).
$$
\end{lemma}
\begin{proof}
    By the (exponential) Markov inequality, for $Y \sim \Bin_{k,p}$, 
    \begin{align*}
     \Bin_{k,p}[Y \geq k - c \log(N)] 
     &=   \Bin_{k,p}[\exp(\alpha \log(N)Y)\geq \exp(\alpha \log(N)(k - c \log(N)))] \\ 
     &\leq  \Bin_{k,p}[\exp(\alpha \log(N)Y)] \exp(- \alpha\log(N)(k - c \log(N))). 
    \end{align*}
    As $\Bin_{k,p}[\exp(\alpha \log(N)Y)] = (1-p+pe^{\alpha \log(N)})^k = (1-N^{-\alpha} +N^{-\alpha}N^{\alpha})^k \leq 2^k,$ plugging in $k=N/\log N$ and expanding, the resulting exponent is at most 
    $$
     - N ( \alpha - \log(2)\log(N)^{-1} - c \log^2(N)N^{-1}) = - N( \alpha + o(1)).
    $$
\end{proof}

\noindent For a percolation configuration $\omega$ and a set of vertices $A,$ let $\mathscr{C}_A(\omega)$ denote the clusters of $\omega$ intersecting $A$.

\begin{lemma} \label{The_Catching_Prop}
Fix $d\geq 3$, $p>p_c$. For any $\delta>0$, there exists $C>0$ such that for every $N\in \mathbb{N}$ and $A\subseteq \partial_v\Lambda_N$ with $|A|$ even,
$$
\phi^0_{\Lambda_N,p}[ |\mathscr{C}_A(\omega|_{\Lambda_N\setminus \Lambda_{3N/4}})| \geq   \delta N\mid \mathcal{F}_A]\leq C\exp(-C N).
$$
\end{lemma}
\begin{proof}
Define $N_k=N-k\log N$ and let $k_{\mathtt{fin}}=\lfloor \frac{N}{4\log N}\rfloor$.

We will consider $\mathtt{Ann}^k=\Lambda_{N_{k-1}}\setminus \Lambda_{N_k}$ and the corresponding free random-cluster measures $\phi^0_{\mathtt{Ann}^k,p}.$ 
 We will explore the configuration in $\omega^A\sim \phi^0_{\Lambda_N,p}[\;\cdot\mid \mathcal{F}_A]$ one scale at a time under an (increasing) exploration coupling with $(\omega^k )_{1\leq k\leq k_{\mathtt{fin}}},$ where $\omega^k \sim \phi^0_{\mathtt{Ann}^k,p}$. That is, fix a total order of the edges such that every edge in $\ak$ is smaller than every edge in $\mathtt{Ann}^{k+1}$ for every $k$ and denote the corresponding exploration coupling by $\mathscr{P}$. By \Cref{lemma:coupling_many_FK_measures}, the $\omega^k$ are independent. Furthermore, by \Cref{lemma:FA_dom},  $\omega^k\preceq \omega^A$ for every $k$, and by our choice of ordering, $\omega^k$ is independent of $\omega^A|_{\Lambda_N \setminus \Lambda_{N_{k-1}}}.$

Let $A^0=A$ and inductively, define $A^k$ as follows: For each cluster $\mathcal{C}$ of $\omega^A|_{\mathtt{Ann}^k}$ which intersects $A^{k-1},$ pick one vertex $v_\mathcal{C}\in (\partial_v \Lambda_{N_{k}}\cap \mathcal{C})$ (say, the ones first in the lexicographical ordering). Set $A^k$ equal to the union of the $v_{\mathcal{C}}$'s.  Note that $|A^k|$ is decreasing in $k$ and $|A^{k_{\mathtt{fin}}}|\geq |\mathscr{C}_A(\omega|_{\Lambda_N\setminus \Lambda_{3N/4}})|.$ 
See \Cref{fig:catching_figure} for an illustration.

Let $\tau=\inf\{k\mid |A^k|< \delta N\}$. 
By \Cref{lemma:fraction_connects_to_supergiant} there exist $\alpha,\varepsilon>0$ such that, in every step, as long as $1\leq k\leq \tau$, there is probability at least $1-N^{-\alpha}$ that an $\varepsilon$-fraction of the vertices in $A^{k-1}$ glue to a single cluster in $\omega^k$.  Since $\omega^k \preceq \omega^A$, 
\begin{equation} \label{Binomial_Induction}
\mathscr{P}[|A^k|\leq (1-\varepsilon) |A^{k-1}| \mid A^{k-1}, \omega^A|_{\Lambda_N\setminus \Lambda_{N_{k-1}}} ] \geq \phi_{\ak}^0[|\mathtt{Giant}_{\ak}\cap A^{k-1}|\geq  \varepsilon |A^{k-1}| ]\geq 1-N^{-\alpha}.
\end{equation}
Accordingly, the set of $k\leq \tau$ for which $[|A^k|> (1-\varepsilon) |A^{k-1}|$ is stochastically dominated by a Bernoulli percolation on $\{1,...,\tau\}$ with density $N^{-\alpha}.$ Thus, \Cref{Lemma:Bernoulli_LD} implies that:
\begin{align*}
\phi_{\Lambda_N,p}^0[|\mathscr{C}_A(\omega|_{\Lambda_N\setminus \Lambda_{3N/4}})|\geq \delta N\mid \mathcal{F}_A]
=\mathscr{P}[|\mathscr{C}_A(\omega^A|_{\Lambda_N\setminus \Lambda_{3N/4}})|\geq \delta N]
\leq \mathscr{P}[|A^{k_{\mathtt{fin}}}|\geq \delta N]
\leq\exp\left(-CN\right).
\end{align*}
\end{proof}

\begin{figure}
    \centering
\usetikzlibrary{calc, backgrounds, decorations.pathreplacing}

\begin{tikzpicture}[scale=0.8]

\def\rA{5.0}   
\def\rB{3.8}   
\def\rC{2.6}   
\def\rD{1.4}   

\begin{scope}[on background layer]
  \fill[orange!8] (-\rA,-\rA) rectangle (\rA,\rA);
  \fill[white] (-\rB,-\rB) rectangle (\rB,\rB);
  \fill[blue!6] (-\rB,-\rB) rectangle (\rB,\rB);
  \fill[white] (-\rC,-\rC) rectangle (\rC,\rC);
  \fill[orange!5] (-\rC,-\rC) rectangle (\rC,\rC);
  \fill[white] (-\rD,-\rD) rectangle (\rD,\rD);
\end{scope}

\draw[thick] (-\rA,-\rA) rectangle (\rA,\rA);
\draw[gray!60] (-\rB,-\rB) rectangle (\rB,\rB);
\draw[gray!60] (-\rC,-\rC) rectangle (\rC,\rC);
\draw[thick, dashed, gray!40] (-\rD,-\rD) rectangle (\rD,\rD);

\node[anchor=south east, font=\small] at (-\rA, \rA) {$\Lambda_N$};
\node[anchor=south east, gray!70, font=\small] at (-\rB, \rB) {$\Lambda_{N_1}$};
\node[anchor=south east, gray!70, font=\small] at (-\rC, \rC) {$\Lambda_{N_2}$};

\draw[decorate, decoration={brace, amplitude=5pt, raise=3pt}, gray!60, thick]
  (-\rA, \rA) -- (-\rA, \rB);
\node[gray!60, font=\footnotesize, anchor=east] at (-\rA-0.35, {0.5*(\rA+\rB)}) {$\log N$};

\node[orange!60!black, font=\footnotesize] at ({0.5*(\rA+\rB)}, -{0.5*(\rA+\rB)}) {$\mathtt{Ann}^1$};
\node[blue!50!black, font=\footnotesize] at ({0.5*(\rB+\rC)}, -{0.5*(\rB+\rC)}) {$\mathtt{Ann}^2$};


\begin{scope}
  \clip (-\rA,-\rA) rectangle (\rA,\rA);
  \fill[orange!18, rounded corners=3pt]
    (-4.8, 4.55) .. controls (-3.5, 4.1) and (-2.0, 4.65) ..
    (-0.5, 4.15) .. controls (1.0, 3.95) and (2.5, 4.5) ..
    (4.8, 4.1) -- (4.8, 3.9) .. controls (2.5, 4.2) and (1.0, 3.85) ..
    (-0.5, 4.0) .. controls (-2.0, 4.5) and (-3.5, 3.9) ..
    (-4.8, 4.35) -- cycle;
  \fill[orange!12, rounded corners=3pt]
    (4.55, -3.5) .. controls (4.1, -2.0) and (4.6, -0.5) ..
    (4.15, 1.0) .. controls (3.95, 2.5) and (4.5, 3.5) ..
    (4.2, 3.6) -- (3.95, 3.6) .. controls (4.2, 3.0) and (3.85, 2.0) ..
    (3.95, 1.0) .. controls (4.4, -0.5) and (3.9, -2.0) ..
    (4.35, -3.5) -- cycle;
  \fill[orange!12, rounded corners=3pt]
    (-4.8, -4.5) .. controls (-3.0, -4.1) and (-1.0, -4.6) ..
    (1.0, -4.15) .. controls (3.0, -3.95) and (4.0, -4.5) ..
    (4.8, -4.2) -- (4.8, -4.0) .. controls (4.0, -4.3) and (3.0, -3.8) ..
    (1.0, -3.95) .. controls (-1.0, -4.4) and (-3.0, -3.9) ..
    (-4.8, -4.3) -- cycle;
  \fill[orange!12, rounded corners=3pt]
    (-4.5, 3.5) .. controls (-4.1, 2.0) and (-4.6, 0.5) ..
    (-4.2, -1.0) .. controls (-3.95, -2.5) and (-4.5, -3.5) ..
    (-4.3, -3.6) -- (-4.05, -3.6) .. controls (-4.3, -3.0) and (-3.85, -2.0) ..
    (-4.0, -1.0) .. controls (-4.4, 0.5) and (-3.9, 2.0) ..
    (-4.25, 3.5) -- cycle;
\end{scope}

\begin{scope}
  \clip (-\rB,-\rB) rectangle (\rB,\rB);
  \fill[blue!15, rounded corners=3pt]
    (-3.6, 3.45) .. controls (-2.5, 3.0) and (-1.0, 3.5) ..
    (0.5, 3.1) .. controls (1.5, 2.9) and (2.5, 3.4) ..
    (3.6, 3.1) -- (3.6, 2.85) .. controls (2.5, 3.1) and (1.5, 2.7) ..
    (0.5, 2.9) .. controls (-1.0, 3.3) and (-2.5, 2.75) ..
    (-3.6, 3.2) -- cycle;
\end{scope}

\foreach \x in {-4.5, -3.8, -3.0, -2.2, -1.4, -0.5, 0.3, 1.1, 1.9, 2.7, 3.5, 4.3} {
    \fill[red!80!black] (\x, \rA) circle (1.8pt);
}
\foreach \y in {-3.5, -1.5, 0.5, 2.5} {
    \fill[red!70!black, opacity=0.6] (\rA, \y) circle (1.5pt);
}
\foreach \x in {-3.5, -1.0, 1.5, 4.0} {
    \fill[red!70!black, opacity=0.6] (\x, -\rA) circle (1.5pt);
}
\foreach \y in {-2.5, 0.0, 2.5} {
    \fill[red!70!black, opacity=0.6] (-\rA, \y) circle (1.5pt);
}

\draw[red!40!orange, semithick, opacity=0.5] (-4.5, \rA) -- (-3.5, \rB);
\draw[red!40!orange, semithick, opacity=0.5] (-3.8, \rA) -- (-3.5, \rB);
\draw[red!40!orange, semithick, opacity=0.5] (-3.0, \rA) -- (-3.5, \rB);
\draw[red!40!orange, semithick, opacity=0.5] (-2.2, \rA) -- (-1.0, \rB);
\draw[red!40!orange, semithick, opacity=0.5] (-1.4, \rA) -- (-1.0, \rB);
\draw[red!40!orange, semithick, opacity=0.5] (-0.5, \rA) -- (-1.0, \rB);
\draw[red!40!orange, semithick, opacity=0.5] (0.3, \rA) -- (1.5, \rB);
\draw[red!40!orange, semithick, opacity=0.5] (1.1, \rA) -- (1.5, \rB);
\draw[red!40!orange, semithick, opacity=0.5] (1.9, \rA) -- (1.5, \rB);
\draw[red!40!orange, semithick, opacity=0.5] (2.7, \rA) -- (3.5, \rB);
\draw[red!40!orange, semithick, opacity=0.5] (3.5, \rA) -- (3.5, \rB);
\draw[red!40!orange, semithick, opacity=0.5] (4.3, \rA) -- (3.5, \rB);

\fill[violet!80!black] (-3.5, \rB) circle (2.2pt);
\fill[violet!80!black] (-1.0, \rB) circle (2.2pt);
\fill[violet!80!black] (1.5, \rB) circle (2.2pt);
\fill[violet!80!black] (3.5, \rB) circle (2.2pt);

\foreach \y in {-1.5, 1.5} {
    \fill[violet!70!black, opacity=0.5] (\rB, \y) circle (1.7pt);
}
\foreach \x in {-1.5, 1.5} {
    \fill[violet!70!black, opacity=0.5] (\x, -\rB) circle (1.7pt);
}
\fill[violet!70!black, opacity=0.5] (-\rB, 0) circle (1.7pt);

\draw[blue!50!violet, semithick, opacity=0.6] (-3.5, \rB) -- (-1.0, \rC);
\draw[blue!50!violet, semithick, opacity=0.6] (-1.0, \rB) -- (-1.0, \rC);

\draw[blue!50!violet, semithick, opacity=0.6] (1.5, \rB) -- (2.2, \rC);
\draw[blue!50!violet, semithick, opacity=0.6] (3.5, \rB) -- (2.2, \rC);

\fill[blue!80!black] (-1.0, \rC) circle (2.5pt);
\fill[blue!80!black] (2.2, \rC) circle (2.5pt);
\fill[blue!80!black, opacity=0.5] (\rC, 0) circle (1.8pt);
\fill[blue!80!black, opacity=0.5] (-0.3, -\rC) circle (1.8pt);

\begin{scope}[shift={(6.2, 1.5)}]
  \fill[red!80!black] (0, 3.0) circle (1.8pt);
  \node[font=\small, anchor=west] at (0.3, 3.0) {$A^0 = A$};

  \fill[violet!80!black] (0, 2.3) circle (2.2pt);
  \node[font=\small, anchor=west] at (0.3, 2.3) {$A^1$};

  \fill[blue!80!black] (0, 1.6) circle (2.5pt);
  \node[font=\small, anchor=west] at (0.3, 1.6) {$A^2$};

  \fill[orange!22] (-0.15, 0.75) rectangle (0.15, 1.05);
  \node[font=\small, anchor=west] at (0.3, 0.9) {$\mathtt{Giant}_{\mathtt{Ann}^k}$};
\end{scope}

\end{tikzpicture}
\caption{Schematic of the strategy in the proof of \Cref{The_Catching_Prop}. Concentric annuli are explored from the outside and inwards. In every annulus of size $\log(N)$, there is a giant cluster in the annulus and with probability tending to $1,$ an $\varepsilon$-fraction of the points in $A^k$ connect to this giant cluster. }
    \label{fig:catching_figure}
\end{figure}

\subsection{Unique crossings conditioned on \texorpdfstring{$\mathcal{F}_A$}{F\_A}}
We are now in position to prove our main technical proposition, which generalises the corresponding lemma for $A=\emptyset$ from \cite[Lemma 4.8]{hansen2023uniform}.
\begin{proof}[Second proof of \Cref{proposition:unique_crossing_with_fa_old_version}]
If $d=2,$ let $\mathtt{Circ}_N$ denote the event that there is a circuit of open edges in $\Lambda_N\setminus \Lambda_{N/2}.$ We note that
$$
\phi^0_{\Lambda_N,p}[\mathtt{UC}_N\mid \mathcal{F}_A]\geq \phi^0_{\Lambda_N,p}[\mathtt{Circ}_N\mid \mathcal{F}_A]\geq \phi^0_{\Lambda_N,p}[\mathtt{Circ}_N]=1-\phi^1_{\Lambda_N^*,p^*}[ \Lambda_{N/2}\cc \partial_v\Lambda_{N}]\geq 1-\exp(-CN)
$$
where the second inequality is FKG, the equality is planar duality \cite[Proposition 2.17]{DC17}, and the third inequality is sharpness of the phase transition for the subcritical FK-model \cite{DCsharpness}.

Now, consider $d\geq 3.$ 
The strategy is first to show the lemma if there are relatively few elements in $A$ and then use \Cref{The_Catching_Prop} to reduce the problem to the case when $A$ does not have that many vertices.

Let $\theta,c>0$ be such that $\phi_{\Lambda_k,p}^{0}[ |\mathcal{C}_v|\geq \frac{\theta}{2} |\Lambda_k|]\geq c$ for all $k\in \mathbb{N},$ and all $v\in \partial_v \Lambda_k$ (cf. \Cref{Touch_Cluster}). By FKG and a union bound, for any $A'\subseteq \partial_v \Lambda_k/\xi ,$  
$$
\phi^{\xi}_{\Lambda_k,p}\left[ \cap_{v\in A'} (v\cc \mathtt{Giant}_{\Lambda_k})\right]\geq \phi^{\xi}_{\Lambda_k,p}\left[ \cap_{v\in A'} (|\mathcal{C}_v|\geq \frac{\theta}{2}|\Lambda_k|)\cap \mathtt{Pis}_{\Lambda_k}\right]\geq c^{|A'|}-\exp(-c'k^{d-1}).
$$
If all vertices in $A'$ connect to the giant, the event $\mathcal{F}_{A'}$ automatically occurs and so  $\phi^{\xi}_{\Lambda_k,p}[\mathcal{F}_{A'}]\geq c^{|A'|}-\exp(-c'k^{d-1}).$ By a slight extension of \cite[Lemma 4.8]{hansen2023uniform}, there exists a $c>0$ such that for all $k$, 
\begin{align}\label{eq:last}
  \inf_{\xi} \phi^{\xi}_{\Lambda_k,p}[\mathtt{UC}(\Lambda_k,\Lambda_{2k/3})]\geq 1-\exp(-c''k),   
\end{align}
where, for $n>m$, $\mathtt{UC}(\Lambda_n,\Lambda_m)$ is the event that there is a unique crossing in the annulus $\Lambda_n\setminus \Lambda_m$ from inner to outer boundary. Accordingly, by a union bound,
$$
\phi^{\xi}_{\Lambda_k,p}[\mathtt{UC}(\Lambda_k, \Lambda_{2k/3})\mid \mathcal{F}_{A'}]\geq 1-\frac{\exp(-c''k)}{c^{|A'|}-\exp(-c'k^{d-1})}.
$$
Now fix the constant $\delta =\frac{\min\{c',c''\}}{2\log(1/c)}$. Whenever $|A'|\leq \delta k ,$ we have $\phi^{\xi}_{\Lambda_k,p}[\mathtt{UC}(\Lambda_k,\Lambda_{2k/3})\mid \mathcal{F}_{A'}]\geq 1-\exp(-ck/2).$

Let $\omega_{\mathtt{exp}}$ be the exploration of the components of $A$ inside $\omega\vert_{\Lambda_N\setminus\Lambda_{3N/4}}$, 
and define $E^\delta_A=\{|\mathcal{C}_A(\omega_{\mathtt{exp}})| < \delta N \}.$ 
By \eqref{SMP}, 
\begin{align*}
\phi^{0}_{\Lambda_N,p}[\mathtt{UC}_N \mid E^\delta_A, \mathcal{F}_A]
& = \sum_{\omega_{\mathtt{exp}}}\phi^{0}_{\Lambda_N,p}[ \mathtt{UC}_N \mid \omega_{\mathtt{exp}},E^\delta_A, \mathcal{F}_A] \phi^{0}_{\Lambda_N,p}[\omega_{\mathtt{exp}} \mid  E^\delta_A, \mathcal{F}_A].
\\
&= \sum_{\omega_{\mathtt{exp}}}\phi^{\xi(\omega_{\mathtt{exp}})}_{\Lambda_{N}\setminus \omega_{\mathtt{exp}},p}[ \mathtt{UC}_N \mid  \omega^{\xi(\omega_{\mathtt{\exp}})}\in\mathcal{F}_{A}] \phi^{0}_{\Lambda_N,p}[\omega_{\mathtt{exp}} \mid  E^\delta_A, \mathcal{F}_A] \\
&\geq  1- \exp(- c N/4). 
\end{align*}

Accordingly,
\begin{align*}
\phi^0_{\Lambda_N,p}[\mathtt{UC}_N\mid \mathcal{F}_A]&\geq \left(1-\exp\left(-c \frac{3N}{8}\right)\right) \left(1-\phi^0_{\Lambda_N,p}[|\mathscr{C}_A(\omega|_{\Lambda_N \setminus \Lambda_{3N/4}})|\geq \delta N\mid \mathcal{F}_A]\right)\\
&\geq 1-\exp(-c' N)
\end{align*}
for an adjusted constant $c'$, where the last inequality is due to \Cref{The_Catching_Prop}. 
\end{proof}

\end{appendix}

\bibliographystyle{abbrv}
\bibliography{bibliography.bib}

\end{document}